\pdfoutput=1
\documentclass[11pt]{article}
\usepackage{amsthm}
\usepackage{amssymb}
\usepackage{amsmath}
\usepackage{hyperref}
\usepackage{graphicx}
\usepackage{array}
\usepackage{enumitem}
\usepackage{mathtools}
\usepackage{etoolbox}
\usepackage[backend=biber, style=alphabetic]{biblatex}
\bibliography{G2_Cohomology}
\usepackage[all,cmtip]{xy}

\pretocmd{\section}{%
  }{}{}

\newcounter{iecounter}
\numberwithin{table}{section}

\newtheorem{theorem}{Theorem}[subsection]
\newtheorem*{theorem*}{Theorem}
\newtheorem{proposition}[theorem]{Proposition}
\newtheorem*{proposition*}{Proposition}

\newtheorem{lemma}[theorem]{Lemma}
\newtheorem*{lemma*}{Lemma}
\newtheorem{conjecture}[theorem]{Conjecture}
\theoremstyle{definition}
\newtheorem{definition}[theorem]{Definition}

\newtheorem*{exercise*}{Exercise}
\newtheorem{remark}[theorem]{Remark}

\numberwithin{equation}{subsection}
\numberwithin{figure}{section}

\let\hom\relax
\let\det\relax
\let\L\relax

\DeclareFontFamily{U}{matha}{\hyphenchar\font45}
\DeclareFontShape{U}{matha}{m}{n}{
      <5> <6> <7> <8> <9> <10> gen * matha
      <10.95> matha10 <12> <14.4> <17.28> <20.74> <24.88> matha12
      }{}
\DeclareSymbolFont{matha}{U}{matha}{m}{n}
\DeclareFontSubstitution{U}{matha}{m}{n}
\DeclareMathSymbol{\abxcup}{\mathbin}{matha}{'131}
\DeclareMathSymbol{\abxcap}{\mathbin}{matha}{'130}

\newcommand{\mb}{\mathbb}
\newcommand{\mbf}{\mathbf}
\newcommand{\mc}{\mathcal}
\newcommand{\mf}{\mathfrak}
\newcommand{\mr}{\mathrm}
\newcommand\sm[1]{{\tiny\arraycolsep=0.3\arraycolsep\ensuremath{\begin{pmatrix}#1\end{pmatrix}}}}
\newcommand\pmat[1]{\begin{pmatrix}#1\end{pmatrix}}
\newcommand\L[1]{\prescript{L}{}{#1}}

\newcommand{\Z}{\mathbb{Z}}
\newcommand{\Q}{\mathbb{Q}}
\newcommand{\R}{\mathbb{R}}
\newcommand{\C}{\mathbb{C}}

\newcommand{\A}{\mathbb{A}}
\renewcommand{\O}{\mathrm{O}}

\renewcommand{\sl}{\mathfrak{sl}}
\renewcommand{\sp}{\mathfrak{sp}}

\newcommand{\PGL}{\mathrm{PGL}}
\newcommand{\G}{\mathrm{G}}
\newcommand{\GSp}{\mathrm{GSp}}
\newcommand{\GL}{\mathrm{GL}}
\newcommand{\SL}{\mathrm{SL}}
\newcommand{\SO}{\mathrm{SO}}
\newcommand{\SU}{\mathrm{SU}}
\newcommand{\Sp}{\mathrm{Sp}}
\newcommand{\U}{\mathrm{U}}

\newcommand{\sset}[2]{\lbrace{#1}\,\,|\,\,{#2}\rbrace}
\newcommand{\Set}[1]{\left\lbrace{#1}\right\rbrace}

\DeclareMathOperator{\ad}{ad}
\DeclareMathOperator{\Ad}{Ad}

\DeclareMathOperator{\cusp}{cusp}
\DeclareMathOperator{\cyc}{cyc}

\DeclareMathOperator{\det}{det}
\DeclareMathOperator{\diag}{diag}

\DeclareMathOperator{\disc}{disc}

\DeclareMathOperator{\Eis}{Eis}
\DeclareMathOperator{\End}{End}

\DeclareMathOperator{\Fil}{Fil}

\DeclareMathOperator{\Frob}{Frob}

\DeclareMathOperator{\hom}{Hom}
\newcommand{\id}{\mathrm{id}}
\DeclareMathOperator{\im}{Im}
\DeclareMathOperator{\Ind}{Ind}

\DeclareMathOperator{\Lie}{Lie}
\newcommand{\modulo}[1]{\,\,(\mathrm{mod}\,\,{#1})}

\DeclareMathOperator{\re}{Re}

\DeclareMathOperator{\sph}{sph}

\DeclareMathOperator{\sss}{ss}

\DeclareMathOperator{\std}{Std}
\DeclareMathOperator{\Sym}{Sym}

\DeclareMathOperator{\tr}{Tr}

\begin{document}
\title{Multiplicity of Eisenstein series in cohomology and applications to $\GSp_4$ and $\G_2$}
\author{Sam Mundy}
\date{}
\maketitle
\begin{abstract}
In this paper, we set up a general framework to compute the exact multiplicity with which certain automorphic representations appear in both the cuspidal and Eisenstein cohomology of locally symmetric spaces. We apply this machinery to Eisenstein series on $\GSp_4$ and split $\G_2$. In the case of $\G_2$, we will also obtain new information about the archimedean components of certain CAP representations using Arthur's conjectures.
\end{abstract}
\section*{Introduction}
This paper grew out of work in progress of the author trying to bound from below the Selmer groups attached to the symmetric cube Galois representations arising from certain modular forms. There has been a lot of interest lately in these symmetric cube Selmer groups. For example, there is the recent work of Haining Wang \cite{wang} and the work of Loeffler--Zerbes \cite{LZBK}. Both of these papers work in the ``Euler system direction," establishing upper bounds on the ranks of the symmetric cube Selmer groups that they study.\\
\indent To establish lower bounds, however, one would like to work in the ``modular direction," analogously to the work of Skinner--Urban \cite{SUgsp4}, who obtain such bounds for the Selmer groups of (standard) Galois representations arising from modular forms. (See also \cite{SUunitary}, where the same authors carry out the same method for Galois representations attached to automorphic forms on unitary groups using Eisenstein series.) To make their method work, Skinner and Urban must assume that these modular forms satisfy certain hypotheses, including that the sign of the functional equation for their $L$-function is $-1$. For a given modular form $f$, this hypothesis in particular ensures the existence of a holomorphic cusp form for $\GSp_4$ which provides a functorial lifting of $f$ from the Siegel Levi subgroup. Such a cusp form is CAP, in the sense of Piatetski-Shapiro, and they are then able to $p$-adically deform this CAP form in a cuspidal family, ultimately applying Ribet's method to the associated Galois representations to obtain nontrivial elements in the correct Selmer group.\\
\indent If we would like to apply the method of Skinner--Urban to study Selmer groups for the symmetric cube of a modular form $f$, the first step would be to find a functorial lifting of $f$ which sees the symmetric cube. Actually, there is such a lifting to $\GSp_4$, and it is this lifting which is used in the works of Wang and Loeffler--Zerbes cited above. However, for the purposes of the Skinner--Urban method, this lifting will be of no use by itself. This is because we want to see on the Galois side not only the symmetric cube Galois representation, but also something degenerate like the trivial one dimensional representation, since the desired Selmer class will be constructed as a nontrivial extension of the symmetric cube by this trivial representation. But the Galois representation attached to the symmetric cube lifting is precisely the symmetric cube, and nothing more.\\
\indent So we need to find a different functorial lifting for modular forms which sees their symmetric cubes along with some more degenerate data. Luckily, it turns out that there are certain cuspidally induced Eisenstein series for the split exceptional group $\G_2$ which do just this. In fact, as was first observed by Langlands, the constant term of such an Eisenstein series involves a quotient involving the symmetric cube $L$-function of the form being induced, and the Riemann zeta function, which is better interpreted in this context as the $L$-function of the trivial one dimensional representation of the Galois group. And indeed, the Galois representations attached to these Eisenstein series are of exactly the right shape to be used in a Ribet-style construction to produce nontrivial elements of the symmetric cube Selmer group.\\
\indent To start the method of Skinner--Urban, we would therefore need to $p$-adically deform these $\G_2$ Eisenstein series in a cuspidal family. Unfortunately, in the absence of a $\G_2$ Shimura variety, our tools for doing this are essentially limited to those established in the paper of Urban \cite{urbanev}, where he constructs eigenvarieties for groups whose real points have discrete series representations. There is a difficulty encountered here in trying to show that our Eisenstein series lies on the cuspidal eigenvariety for $\G_2$. This has to do with the $p$-adic properties of our Eisenstein series at the place $p$. Urban's machinery, as is usual in the theory of $p$-adic deformations of automorphic forms, requires us to make the choice of a $p$-stabilization of our Eisenstein series, and in order to obtain in the end a family of Galois representations which has the correct properties coming from $p$-adic Hodge theory, we have to choose a $p$-stabilization which is critical, in the sense of Urban's paper.\\
\indent Deforming $p$-adically automorphic forms which have been $p$-stabilized noncritically is not so difficult using Urban's machinery. But it can be much harder for those which are critically $p$-stabilized. However, Urban does provide some tools for doing this, and the first step toward using these tools to study $p$-adic deformations of our Eisenstein series is to locate every instance of the Eisenstein series in the cohomology of the locally symmetric spaces attached to $\G_2$. The main purpose of this paper is to carry out this first step.\\
\indent We can be more precise at this point. An irreducible, finite dimensional representation $E$ of the complex group $\G_2(\C)$ gives rise to compatible local systems on the locally symmetric spaces attached to $\G_2$. The cohomology groups of these local systems form a direct system whose direct limit is an admissible representation of the group of finite adelic points $\G_2(\A_f)$. By work of Franke \cite{franke}, this representation can be constructed in the following way.\\
\indent Let $\mf{g}_2$ denote the complexified Lie algebra of $\G_2$, and fix a maximal compact subgroup $K_\infty$ in the group real points $\G_2(\R)$. One can define a certain space $\mc{A}_E(\G_2)$ of automorphic forms for $\G_2$ using $E$, in a way which we will not be precise about in this introduction. But it is a $\G_2(\A_f)\times(\mf{g}_2,K_\infty)$-module, making its cohomology
\[H^*(\mf{g}_2,K_\infty;\mc{A}_E(\G_2)\otimes E)\]
a $\G_2(\A_f)$-module. By a conjecture of Borel, which was proved by Franke in his paper, this module is exactly the direct limit discussed in the previous paragraph. Therefore, our goal will be to locate our Eisenstein series in the cohomology space displayed above. Let us be more precise about what we mean by this.\\
\indent Eisenstein series are related to parabolic inductions. Let $P$ be the so-called long root parabolic subgroup of $\G_2$, with Levi $M$. Then $M\cong\GL_2$, and so if we are given a cuspidal automorphic representation $\pi$ of $\GL_2(\A)$, we can view it as a representation of $M(\A)$. For $s\in\C$, form the unitary parabolic induction
\[\iota_{P(\A)}^{\G_2(\A)}(\pi,s)=\Ind_{P(\A)}^{\G_2(\A)}(\pi\otimes\delta_{P(\A)}^{s+1/2}),\]
where $\delta_{P(\A)}$ denotes the modulus character of $P(\A)$.\\
\indent Of interest for us is the Langlands quotient $\mc{L}(\pi,1/10)$ of the above parabolic induction at the point $s=1/10$. It is this representation that we are trying to deform for the purposes of the Skinner--Urban method and which we therefore want to locate in cohomology.\\
\indent Let $\mc{A}_E(\G_2)_{\cusp}$ be the space of cusp forms in $\mc{A}_E(\G_2)$. It has a natural complement $\mc{A}_E(\G_2)_{\Eis}$ which is built, in a way which can be made precise, from Eisenstein series, and the decomposition
\[\mc{A}_E(\G_2)=\mc{A}_E(\G_2)_{\cusp}\oplus\mc{A}_E(\G_2)_{\Eis}\]
is a decomposition of $\G_2(\A_f)\times(\mf{g}_2,K_\infty)$-modules. We therefore get a decomposition
\[H^*(\mf{g}_2,K_\infty;\mc{A}_E(\G_2)\otimes E)=H^*(\mf{g}_2,K_\infty;\mc{A}_E(\G_2)_{\cusp}\otimes E)\oplus H^*(\mf{g}_2,K_\infty;\mc{A}_E(\G_2)_{\Eis}\otimes E)\]
as $\G_2(\A_f)$-modules. The first of these factors is called the \it cuspidal cohomology \rm and the second is the \it Eisenstein cohomology. \rm We then have the following result, which is a consequence of Theorem \ref{thmeismultg2} of this paper.
\begin{theorem*}
Assume $\pi$ comes from a cuspidal holomorphic eigenform of weight $k\geq 4$ and trivial nebentypus. Assume $L(1/2,\pi,\Sym^3)=0$. Then there is only one representation $E$ for which the finite part $\mc{L}(\pi,1/10)_f$ of our Langlands quotient appears as a subquotient of the Eisenstein cohomology
\[H^*(\mf{g}_2,K_\infty;\mc{A}_E(\G_2)_{\Eis}\otimes E).\]
It appears exactly once in this cohomology space, in (middle) degree $4$, with multiplicity one.
\end{theorem*}
There are two steps to establishing this result. First, one must actually construct the representation $\mc{L}(\pi,1/10)_f$ as a subquotient of Eisenstein cohomology. This is made possible by a deeper analysis of the Eisenstein space $\mc{A}_E(\G_2)_{\Eis}$ as follows. For $Q$ another parabolic subgroup of $\G_2$, Franke and Schwermer \cite{FS} have defined an equivalence relation on the cuspidal automorphic representations of the Levi of $Q$. Let $\varphi$ denote one of these equivalence classes. Then Franke and Schwermer construct a subspace
\[\mc{A}_{E,Q,\varphi}(\G_2)\subset\mc{A}_E(\G_2)_{\Eis}\]
out of Eisenstein series induced from the representations in $\varphi$, along with their residues and derivatives. (Actually Franke--Schwermer work much more generally on an arbitrary reductive group.) There is a decomposition
\[\mc{A}_E(\G_2)_{\Eis}=\bigoplus_{Q}\bigoplus_{\varphi}\mc{A}_{E,Q,\varphi}(\G_2),\]
where the first sum is over a fixed set of parabolic subgroups which represent the associate classes of proper parabolic subgroups of $\G_2$.\\
\indent If none the Eisenstein series arising in the construction of the space $\mc{A}_{E,Q,\varphi}(\G_2)$ have a pole, then the space $\mc{A}_{E,Q,\varphi}(\G_2)$ has a very nice and explicit $\G_2(\A_f)\times(\mf{g}_2,K_\infty)$-module structure as a parabolically induced module. We can then compute the $(\mf{g}_2,K_\infty)$-cohomology of this module explicitly in terms of a representation parabolically induced from the finite part of a representation in $\varphi$.\\
\indent If $\varphi(\pi,1/10)$ is the class for $P$ which contains the representation $\pi\otimes\delta_{P(\A)}^{1/10}$, where $\pi$ is the same as in our theorem, then it turns out that none of the Eisenstein series appearing in the construction of $\mc{A}_{E,P,\varphi(\pi,1/10)}(\G_2)$ has a pole, and we can find $\mc{L}(\pi,1/10)_f$ as a quotient of the cohomology
\[H^*(\mf{g}_2,K_\infty;\mc{A}_{E,P,\varphi(\pi,1/10)}(\G_2)).\]
This is where we use the hypothesis that $L(1/2,\pi,\Sym^3)=0$; this vanishing allows us, via an examination of the constant term of our Eisenstein series, to conclude that these Eisenstein series do not have poles at $s=1/10$.\\
\indent This describes the first of the two steps we need to prove our theorem. The second of these steps is to show that no other summand $\mc{A}_{E,Q,\varphi}(\G_2)$ of the decomposition above, besides the summand for $Q=P$ and $\varphi=\varphi(\pi,1/10)$ just studied, contains any copy of $\mc{L}(\pi,1/10)_f$ in its cohomology. To do this, we need to study the cohomology of these summands in a way which is explicit enough to rule out an appearance of $\mc{L}(\pi,1/10)_f$.\\
\indent One runs into a problem here, as we only know the explicit structure of the space $\mc{A}_{E,Q,\varphi}(\G_2)$ as a parabolic induction when the Eisenstein series involved in its construction have no poles. But it may well be the case that certain Eisenstein series induced from $\varphi$ do have poles. Luckily, following Franke \cite{franke}, Grobner \cite{grobner} has defined a filtration on these spaces whose graded pieces are parabolically induced modules whose cohomology can be explicitly studied.\\
\indent So one just needs to show that $\mc{L}(\pi,1/10)_f$ doesn't appear in the cohomology of these graded pieces. To do this, we distinguish $\mc{L}(\pi,1/10)_f$ from the representations appearing in the cohomology of the graded pieces by assigning to them $\ell$-adic Galois representations for a fixed prime $\ell$. These Galois representations are only powerful enough to distinguish between near-equivalence classes of automorphic representations, that is, to tell them apart outside a set of finitely many primes. But actually this is enough for our purposes because we can appeal to strong multiplicity one theorems for the Levis of $\G_2$.\\
\indent The next thing to do would be to compute the multiplicity of $\mc{L}(\pi,1/10)_f$ in the cuspidal cohomology. This requires knowledge about the classification of CAP forms which are nearly equivalent to our Langlands quotient $\mc{L}(\pi,1/10)$. However, not enough about such things is known unless we assume some standard conjectures related to those of Arthur. So this is what we do.\\
\indent As explained by Gan and Gurevich \cite{gangurl}, assuming such conjectures, under the hypothesis still that $L(1/2,\pi,\Sym^3)=0$, precisely two kinds of CAP representations $\Pi$ with $\Pi_f\cong\mc{L}(\pi,1/10)_f$ should be able to appear in $\mc{A}_E(\G_2)_{\cusp}$, depending on the sign $\epsilon$ of the symmetric cube functional equation. They will appear with multiplicity one in either case. If $\epsilon=1$, then $\Pi_\infty\cong\mc{L}(\pi,1/10)_\infty$, and hence this appears in cuspidal cohomology exactly once in each of degrees $3$ and $5$. But Gan and Gurevich do not describe $\Pi_\infty$ when $\epsilon=-1$. So we must do this ourselves.\\
\indent One of the things which Gan and Gurevich explain, however, is that there is an Arthur parameter $\psi$ for $\G_2(\R)$ whose associated Arthur packet should consist of the two possible representations which can occur as $\Pi_\infty$. Upon examination of the parameter $\psi$, one sees that the one corresponding to $\epsilon=1$ must indeed be $\mc{L}(\pi,1/10)_\infty$, but Arthur's conjectures do not immediately tell us anything about the representation corresponding to $\epsilon=-1$.\\
\indent However, for certain types of Arthur parameters $\psi$, Adams and Johnson have been able to construct packets $\mr{AJ}_{\psi}$ which satisfy the conclusion of Arthurs conjectures for real groups. We show that our parameter $\psi$ is of this special type, and we explicitly compute $\mr{AJ}_\psi$. We find the following, which is the content of Theorem \ref{thmajpacketg2} in this paper.
\begin{theorem*}
The Adams--Johnson packet $\mr{AJ}_\psi$ contains the representation $\mc{L}(\pi,1/10)_\infty$ and the quaternionic discrete series representation of $\G_2(\R)$ of weight $k/2$, in the terminology of Gan--Gross--Savin \cite{ggs}.
\end{theorem*}
Thus if $\epsilon=-1$, it follows that our CAP representation $\Pi$ should again be cohomological, appearing in cuspidal cohomology exactly once in middle degree $4$.\\
\indent This paper is organized as follows. The first three chapters are devoted to a very general setup, working mostly for an arbitrary reductive group, and they will be used to make the main computations in Chapters \ref{chgsp4} and \ref{chg2}.\\
\indent In Chapter \ref{chES}, we review facts about Eisenstein series and the spaces they comprise, recalling the Franke--Schwermer decomposition and some facts about the Franke filtration. In Chapter \ref{chcoh}, we compute the cohomology of some of these spaces of Eisenstein series. In Chapter \ref{chgal}, we explain what we mean when we say an automorphic representation has attached to it an $\ell$-adic Galois representation.\\
\indent Chapter \ref{chgsp4} is then devoted to applying the tools set up in the first three chapters to compute the cohomological multiplicity of certain Langlands quotients for $\GSp_4$. Some of the results in this chapter were originally stated by Urban in \cite{urbanev}, Example 5.5.3. However, he made an error in that example which we take the opportunity to correct (see Remark \ref{remurban} in this paper).\\
\indent Chapter \ref{chg2} makes the same kind of computations for $\G_2$, and although a lot of the arguments there are completely analogous to the $\GSp_4$ case, the chapter is written in such a way that the reader can read it without having read Chapter \ref{chgsp4}.\\
\indent What is not completely analogous between these two chapters is that for $\GSp_4$, the CAP forms we need have been completely classified, and so the computation of the cuspidal multiplicity in the $\GSp_4$ case is unconditional. As we mentioned, this is not the case for $\G_2$, and Chapter \ref{chartpacket} is devoted to the aforementioned computation of the Adams--Johnson packet which makes our conditional results reasonable.
\subsection*{Acknowledgements}
I want to thank my advisor Eric Urban for suggesting the problem that eventually led to this work, and for his guidance in navigating the automorphic ocean. Many thanks are also due to Jeffrey Adams for his help with the problem whose solution became the final chapter of this paper.\\
\indent We would also like to thank Rapha\"el Beuzart-Plessis, Johan de Jong, Wee Teck Gan, Michael Harris, Herv\'e Jacquet, Stephen D. Miller, Aaron Pollack, Christopher Skinner, and Shuai Wang for helpful conversations.
\section*{Notation and conventions}
We now set the notation that will be used throughout the rest of this paper.
\subsubsection*{Groups and Lie algebras}
In Chapters \ref{chES} and \ref{chcoh}, $G$ will denote a reductive group over the field $\Q$ of rational numbers. In Chapter \ref{chgal}, $G$ will furthermore be split over $\Q$. In Chapter \ref{chgsp4}, we will specialize to the group $\GSp_4$ and, in Chapter \ref{chg2}, to $\G_2$. In Chapter \ref{chartpacket}, we will be working primarily with a real reductive Lie group, and we will denote that group by $\mbf{G}$.\\
\indent In general, our convention is to use uppercase roman letters to denote groups over $\Q$, such as $G$, to use uppercase boldface letters to denote real Lie groups, such as $\mbf{G}$, and to use the corresponding lowercase fraktur letters to denote complex Lie algebras. So for example, $\mf{g}$ will always denote the complexified Lie algebra of either the $\Q$-group $G$ or the real Lie group $\mbf{G}$. There will be a few exceptions to this convention, however. For example, when we have fixed a reductive $\Q$-group $G$, unless otherwise noted, we will simply write $G(\R)$ for the real Lie group consisting of its $\R$-points.\\
\indent When working with the group $G$, we will often fix a parabolic subgroup $P$ of $G$ along with a Levi decomposition $P=MN$. In this decomposition, $M$ will always denote the Levi factor and $N$ the unipotent radical. If we have another parabolic subgroup with fixed Levi decomposition, then we use subscripts on the notation for its fixed Levi factor and its unipotent radical to distinguish them from those of $P$; so if $Q$ is another parabolic subgroup, we will write $Q=M_Q N_Q$ for its Levi decomposition.\\
\indent For any parabolic $Q$ as above, the notation $A_Q$ will denote the maximal $\Q$-split torus in the center of the Levi $M_Q$ of $Q$. This applies in particular to $P$ and $G$; we use $A_G$ to denote the maximal $\Q$-split torus in the center of $G$, and $A_P$ that of $M$.\\
\indent Now we have the complexified Lie algebras $\mf{g}$, $\mf{p}$, $\mf{q}$, $\mf{m}$, $\mf{m}_Q$, $\mf{n}$, $\mf{n}_Q$, $\mf{a}_P$, and $\mf{a}_Q$ of, respectively, $G$, $P$, $Q$, $M$, $M_Q$, $N$, $N_Q$, $A_P$, and $A_Q$. We let $\mf{g}_0=[\mf{g},\mf{g}]$, the self-commutator of $\mf{g}$, and more generally, we write $\mf{m}_{Q,0}=[\mf{m}_Q,\mf{m}_Q]$, or $\mf{m}_0=[\mf{m},\mf{m}]$. We also write $\mf{q}_0=\mf{q}\cap\mf{g}_0$ and $\mf{a}_{Q,0}=\mf{a}_Q\cap\mf{g}_0$, and similarly for $\mf{p}_0$ and $\mf{a}_{P,0}$. Then there are decompositions
\[\mf{q}=\mf{m}_{Q,0}\oplus\mf{a}_Q\oplus\mf{n}_Q,\]
and
\[\mf{q}_0=\mf{m}_{Q,0}\oplus\mf{a}_{Q,0}\oplus\mf{n}_Q.\]
\indent When $P$ and $Q$ are fixed along with their respective Levi decompositions, we will write $W(P,Q)$ for the set equivalence classes of elements $w\in G(\Q)$ such that $w Mw^{-1}=M_Q$. where $w$ and $w'$ are considered equivalent if $w^{-1}w'$ centralizes $M$.\\
\indent We will always write $\rho_Q$ for the character $\rho_Q:\mf{a}_{Q,0}\to\C$ given by
\[\rho_Q(X)=\tr(\ad(X)|\mf{n}_Q),\qquad X\in\mf{a}_{Q,0},\]
and similarly for $\rho_P$.
\subsubsection*{Points of groups}
When $v$ is a place of $\Q$, we write $\Q_v$ for the completion of $\Q$ at $v$. Then $\R=\Q_\infty$. The group of $\Q_v$-points of any affine algebraic group over $\Q$ is always given the usual topology induced from $\Q_v$.\\
\indent We write $\A$ for the adeles of $\Q$ and $\A_f$ for the finite adeles. The groups of $\A$-points or $\A_f$-points of any affine algebraic group over $\Q$ are also given their standard topologies.\\
\indent When $P=MN$ is fixed as above, we will often consider the associated height function $H_P$. This is a function
\[H_P:G(\A)\to\mf{a}_{P,0}.\]
To define it, we must fix a maximal compact subgroup $K\subset G(\A)$. We assume $K=K_f K_\infty$ where $K_\infty$ is a fixed maximal compact subgroup of $G(\R)$ and $K_f=\prod_{v<\infty} K_v$ is a maximal compact subgroup of $G(\A_f)$ which we assume to be in good position with respect to a fixed minimal parabolic inside $P$. (Here the groups $K_v$ are maximal compact subgroups of $G(\Q_v)$.) In particular, the Iwasawa decomposition holds for $P(\A)$ and $K$.\\
\indent Write $\langle\cdot,\cdot\rangle$ for the natural pairing
\[\langle\cdot,\cdot\rangle:\mf{a}_{P,0}\times\mf{a}_{P,0}^\vee\to\C\]
given by evaluation, where $\mf{a}_{P,0}^\vee=\hom_{\C}(\mf{a}_{P,0},\C)$. Write $X^*(M)$ for the group of algebraic characters of $M$. Then $H_P$ is defined first on the subgroup $M(\A)$ by requiring
\[e^{\langle H_P(m),d\Lambda\rangle}=\vert\Lambda(m)\vert,\qquad m\in M(\A),\,\,\Lambda\in X^*(M),\]
where $d\Lambda$ denotes the restriction to $\mf{a}_{P,0}$ of the differential at the identity of the restriction of $\Lambda$ to $A_P(\R)$, and $\vert\cdot\vert$ is usual the adelic absolute value. Then $H_P$ is defined in general by declaring it to be left invariant with respect to $N(\A)$ and right invariant with respect to $K$.\\
\indent If $R$ is one of the rings $\Q_v$, $\A$, or $\A_f$, we use the notation $\delta_{P(R)}$ to denote the modulus character of $P(R)$, and similarly for other parabolics.
\subsubsection*{Automorphic representations}
We take the point of view that an ``automorphic representation" of $G(\A)$ is (among other things) an irreducible object in the category of admissible $G(\A_f)\times(\mf{g},K_\infty)$-modules. We often even view automorphic representations as $G(\A_f)\times(\mf{g}_0,K_\infty)$-modules by restriction. We let $\mc{A}(G)$ denote the space of all automorphic forms on $G(\A)$.\\
\indent If $\Pi$ is an automorphic representation of $G(\A)$ and $v$ is a place of $\Q$, we will denote by $\Pi_v$ the local component of $\Pi$ at $v$. If $v$ is finite, then this is an irreducible admissible representation of $G(\Q_v)$, and if $v=\infty$, then this is an irreducible admissible $(\mf{g},K_\infty)$-module.
\subsubsection*{Galois theory}
We will write $G_\Q$ for the absolute Galois group of $\Q$, and for any place $v$ of $\Q$, we will similarly write $G_{\Q_v}$ for the absolute Galois group of $\Q_v$. If $v$ is finite, we always view $G_{\Q_v}$ as a subgroup of $\Q$ via by fixing a decomposition group at $v$.\\
\indent Galois representations for us will always be into the $\overline\Q_\ell$-points of a fixed algebraic group. We always identify $\overline\Q_\ell$ with $\C$ via a fixed isomorphism.\\
\indent For $p$ a prime, $\Frob_p$ always denotes a fixed geometric Frobenius element at $p$ in $G_\Q$. If $\chi_{\cyc}$ denotes the $\ell$-adic cyclotomic character, then our conventions will be such that twists by $\vert\cdot\vert$ on the automorphic side correspond to twists by $\chi_{\cyc}$ on the Galois side; $\vert\cdot\vert$ sends $p\in\Q_p^\times$ to $p^{-1}$, and $\chi_{\cyc}$ also sends $\Frob_p$ to $p^{-1}$.
\subsubsection*{Duals}
We use the symbol $(\cdot)^\vee$ in various ways. If $\mf{a}$ is an abelian Lie algebra, $\mf{a}^\vee$ will denote the characters of $\mf{a}$. If $R$ is a complex representation of a group, then $R^\vee$ is the usual dual representation over $\C$. Similarly, if $\rho$ is an $\ell$-adic Galois representation, then $\rho^\vee$ is the usual dual representation over $\overline\Q_\ell$. If $G$ is our reductive $\Q$-group, then $G^\vee(\C)$ or $G^\vee(\overline\Q_\ell)$ will denote the dual group over either of the algebraically closed fields $\C$ or $\overline\Q_\ell$, respectively. Similarly, if $\mbf{G}$ is a real reductive Lie group, $\mbf{G}^\vee(\C)$ will denote its dual group.
\tableofcontents
\section{Eisenstein series and spaces of automorphic forms}
\label{chES}
This chapter will be devoted to studying spaces of automorphic forms in the style of Franke \cite{franke} and Franke--Schwermer \cite{FS}. We will state the Franke--Schwermer decomposition and study the structure of its pieces using the Franke filtration. But first, we recall some of the theory of Eisenstein series.
\subsection{Review of Eisenstein series}
\label{secES}
Let $P\subset G$ a parabolic $\Q$-subgroup of our reductive group $G$ (see the section on notation in the introduction) with fixed Levi decomposition $P=MN$. In this section, we will recall how to use automorphic representations of $M(\A)$ to construct Eisenstein series, and we will explain how to study these Eisenstein series using parabolically induced representations and intertwining operators.
\subsubsection*{Eisenstein series and their constant terms}
We start with a cuspidal automorphic representation $\pi$ of $M(\A)$ with central character $\chi_\pi$, and we assume $\chi_\pi$ is trivial on $A_G(\R)^\circ$. So if 
\[L^2(M(\Q)A_G(\R)^\circ\backslash M(\A),\chi_\pi)\]
denotes the space of functions on $M(\Q)A_G(\R)^\circ\backslash M(\A)$ which are square integrable modulo center and which transform under the center with respect to $\chi_\pi$, then $\pi$ occurs in the cuspidal spectrum
\[L_{\cusp}^2(M(\Q)A_G(\R)^\circ\backslash M(\A),\chi_\pi)\subset L^2(M(\Q)A_G(\R)^\circ\backslash M(\A),\chi_\pi).\]
\indent Write $d\chi_\pi:\mf{a}_{P,0}\to\C$ for the differential of the restriction of $\chi_\pi$ to $A_P(\R)^\circ/A_G(\R)^\circ$. The character $d\chi_\pi$ is an element of $\mf{a}_{P,0}^\vee$. Then we consider the automorphic representation
\[\tilde{\pi}=\pi\otimes e^{-\langle H_P(\cdot), d\chi_\pi\rangle}.\]
The representation $\tilde{\pi}$ is a unitary automorphic representation. If $\pi$ is realized on a space of functions 
\[V_\pi\subset L_{\cusp}^2(M(\Q)A_G(\R)^\circ\backslash M(\A),\chi_\pi),\]
then $\tilde{\pi}$ is realized on the space
\[V_{\tilde{\pi}}=\{e^{-\langle H_P(\cdot), d\chi_\pi\rangle}f\,\,|\,\,f\in V_\pi\},\]
which is a subspace of $L_{\cusp}^2(M(\Q)A_P(\R)^\circ\backslash M(\A))$.\\
\indent Now we let $W_{P,\tilde\pi}$ be the space of smooth, $K$-finite, $\C$-valued functions $\phi$ on
\[M(\Q)N(\A)A_P(\R)^\circ\backslash G(\A)\]
such that, for all $g\in G(\A)$, the function
\[m\mapsto \phi(mg)\]
of $m\in M(\A)$ lies in the space
\[L_{\cusp}^2(M(\Q)A_P(\R)^\circ\backslash M(\A))[\tilde\pi].\]
Here, the brackets denote an isotypic component.\\
\indent The space $W_{P,\tilde\pi}$ lets us build Eisenstein series. In fact, let $\phi\in W_{P,\tilde\pi}$. We define, for $\lambda\in\mf{a}_{P,0}^\vee$ and $g\in G(\A)$, the Eisenstein series $E(\phi,\lambda)$ by
\[E(\phi,\lambda)(g)=\sum_{\gamma\in P(\Q)\backslash G(\Q)}\phi(\gamma g) e^{\langle H_P(g),d\chi_\pi+\rho_P\rangle}.\]
This series only converges for $\lambda$ sufficiently far inside a positive Weyl chamber, but it defines a holomorphic function there in the variable $\lambda$ which continues meromorphically to all of $\mf{a}_{P,0}^\vee$; see \cite{Langlands}, \cite{MW}, or more recently \cite{BL}, where the proof has been greatly simplified.\\
\indent For each fixed $\phi$ and for each fixed $\lambda$ at which $E(\phi,\lambda)$ does not have a pole, the Eisenstein series $E(\phi,\lambda)$ is an automorphic form on $G(\A)$. It will be important for us in our examples of $\GSp_4$ and $\G_2$ to study when and how certain Eisenstein series have poles. The general theory which explains how to do this, as developed for instance in \cite{LEP} and \cite{shahidies}, goes through two steps. First, one reduces to studying the constant terms of Eisenstein series, and second, one computes the constant terms using local calculations involving intertwining operators.\\
\indent This first step is relatively easy to explain. Let $Q\subset G$ be another parabolic subgroup, this time with Levi decomposition $Q=M_QN_Q$. The \it constant term \rm of $E(\phi,\lambda)$ along $Q$ is, as usual, defined by
\[E_Q(\phi,\lambda)(g)=\int_{N_Q(\Q)\backslash N_Q(\A)}E(\phi,\lambda)(ng)\,dn.\]
It is meromorphic in $\lambda$. Furthermore, the Eisenstein series $E(\phi,\lambda)$ has a pole at a point $\lambda=\mu$ if and only if there is a proper parabolic subgroup $Q$ such that $E_Q(\phi,\lambda)$ has a pole at $\lambda=\mu$.\\
\indent Next, to proceed and compute the constant terms of Eisenstein series using local computations, we first need to express the space $W_{P,\tilde\pi}$ in terms of local pieces.
\subsubsection*{Induced representations}
The space $W_{P,\tilde\pi}$ is a parabolic induction space. In fact, let us view $\tilde\pi$ as acting on the subspace $V_{\tilde\pi}$ of $L_{\disc}^2(M(\Q)A_P(\R)^\circ\backslash M(\A))$. The pair $(\tilde\pi,V_{\tilde\pi})$ is an $M(\A_f)\times(\mf{m}_0,K_\infty\cap P(\R))$-module, and we extend this structure to a $P(\A_f)\times(\mf{p}_0,K_\infty\cap P(\R))$-module structure via the trivial action by the unipotent radical. We consider the parabolic induction functor
\[\Ind_{P(\A_f)\times(\mf{p}_0,K_\infty\cap P(\R))}^{G(\A_f)\times(\mf{g}_0,K_\infty)}\]
and, for $\lambda\in\mf{a}_{P,0}^\vee$, we write
\[\Ind_{P(\A)}^{G(\A)}(\tilde\pi\otimes e^{\langle H_P(\cdot),\lambda\rangle})=\Ind_{P(\A_f)\times(\mf{p}_0,K_\infty\cap P(\R))}^{G(\A_f)\times(\mf{g}_0,K_\infty)}(\tilde\pi\otimes e^{\langle H_P(\cdot),\lambda\rangle}),\]
for short. The space above is an unnormalized induction, and we can normalize it by writing
\[\iota_{P(\A)}^{G(\A)}(\tilde\pi,\lambda)=\Ind_{P(\A)}^{G(\A)}(\tilde\pi\otimes e^{\langle H_P(\cdot),\lambda+\rho_P\rangle}).\]
Then there is an isomorphism of $G(\A_f)\times(\mf{g}_0,K_\infty)$-modules
\[\iota_{P(\A)}^{G(\A)}(L_{\cusp}^2(M(\Q)A_P(\R)^\circ\backslash M(\A))[\tilde\pi],\lambda)\cong e^{\langle H_P(\cdot),\lambda+\rho_P\rangle}W_{P,\tilde\pi},\]
where the space on the right hand side is just defined by
\[e^{\langle H_P(\cdot),\lambda+\rho_P\rangle}W_{P,\tilde\pi}=\{e^{\langle H_P(\cdot),\lambda+\rho_P\rangle}f\,\,|\,\, f\in W_{P,\tilde\pi}\}.\]
Therefore, elements of the induction $\iota_{P(\A)}^{G(\A)}(\tilde\pi,\lambda)$ can also be used to define Eisenstein series as above.
\subsubsection*{Intertwining operators}
\indent We now need to define the intertwining operators, which will let us access the constant terms of Eisenstein series.\\
\indent Given another parabolic subgroup $Q=M_QN_Q$ of $G$, let given $w\in W(P,Q)$, let us identify $w$ with an element of $G(\Q)$. For $\lambda,\lambda'\in\mf{a}_{P,0}^\vee$ and $\phi_\lambda\in\iota_{P(\A)}^{G(\A)}(\tilde\pi,\lambda)$, define a new element $\phi_{\lambda'}\in\iota_{P(\A)}^{G(\A)}(\tilde\pi,\lambda')$
\[\phi_{\lambda'}=\phi_\lambda e^{\langle H_P(\cdot),\lambda'-\lambda\rangle}.\]
We say that this assignment $\lambda\mapsto\phi_\lambda$ is a \it flat section \rm of the induction.\\
\indent Now define (formally) for $\phi_\lambda$ varying in a flat section, the \it intertwining operator \rm $M(w,\cdot)$ by
\[M(w,\phi)_{w\lambda}(g)=\int_{(wNw^{-1}\cap N_Q)(\A)\backslash N_Q(\A)}\phi_\lambda(w^{-1}ng)\,dn.\]
When convergent, this defines a map of $G(\A_f)\times(\mf{g}_0,K_\infty)$-modules,
\[\iota_{P(\A)}^{G(\A)}(\tilde\pi,\lambda)\to\iota_{Q(\A)}^{G(\A)}(\tilde\pi^w,w\lambda),\]
where if $\sigma$ is an automorphic representation of $M(\A)$, then $\sigma^w$ denotes the automorphic representation of $M_Q(\A)$ defined by $\sigma^w(m)=\sigma(w^{-1}mw)$. It is a fact that the integral defining $M(w,\cdot)$ does converge for $\lambda$ in a certain cone in $\mf{a}_{P,0}^\vee$ and is holomorphic in $\lambda$ there, and that it continues meromorphically to all of $\mf{a}_{P,0}^\vee$.\\
\indent We can use the intertwining operators to describe the constant term. The following theorem is due to Langlands. See Section 6.2 of the book by Shahidi \cite{shahidies}.
\begin{theorem}
\label{thmconstterm}
Let $Q=M_QN_Q$ be a parabolic $\Q$-subgroup of $G$. Then
\[E_Q(\phi,\lambda)=\sum_{w\in W(P,Q)}M(w,\phi)_{w\lambda},\]
which is an equality of functions of $g\in G(\A)$ varying meromorphically in $\lambda$.
\end{theorem}
\subsubsection*{Local study of intertwining operators}
Now we make a local study of the intertwining operators in order to incorporate the theory of $L$-functions into our considerations. To do this, we first write the automorphic representation $\tilde\pi$ in terms of its local components as usual as
\[\tilde\pi\cong\sideset{}{'}\bigotimes_v\tilde\pi_v,\]
where the restricted tensor product is over all places $v$ of $\Q$; the representation $\tilde\pi_v$ is a smooth, admissible representation of $M(\Q_v)$ if $v$ is finite, and it is an admissible $(\mf{m}_0,K_\infty\cap P(\R))$-module is $v=\infty$.\\
\indent If $v$ is finite, $\lambda\in\mf{a}_{P,0}^\vee$, and $\sigma$ is a smooth admissible representation of $M(\Q_v)$, let us write
\[\iota_{P(\Q_v)}^{G(\Q_v)}(\sigma,\lambda)=\Ind_{P(\Q_v)}^{G(\Q_v)}(\sigma\otimes\lambda)\]
for the usual smooth, $K_v$-finite parabolic induction, where $\lambda$ is being viewed as a character of $M(\Q_v)$ via the canonical identification $\mf{a}_{P}^\vee\cong X^*(M)\otimes\C$ and the inclusion $\mf{a}_{P,0}^\vee\hookrightarrow\mf{a}_P^\vee$. Similarly if $\sigma$ is instead an admissible $(\mf{m}_0,K_\infty\cap P(\R))$-module, let us write
\[\iota_{P(\R)}^{G(\R)}(\sigma,\lambda)=\Ind_{(\mf{p}_0,K_\infty\cap P(\R))}^{(\mf{g}_0,K_\infty)}(\sigma\otimes\lambda)\]
for the usual archimedean parabolic induction, where this time $\lambda$ is being viewed as a character of $\mf{p}_0$ by letting it act trivially on $\mf{m}_0$ and $\mf{n}$. Then via the decomposition of $\tilde\pi$ above, we have an isomorphism
\[\iota_{P(\A)}^{G(\A)}(\tilde\pi,\lambda)\cong\sideset{}{'}\bigotimes_v\iota_{P(\Q_v)}^{G(\Q_v)}(\tilde\pi_v,\lambda).\]
\indent Let $Q=M_QN_Q$ again be another parabolic $\Q$-subgroup of $G$. For any $w\in W(P,Q)$ and any place $v$, there are also local intertwining operators
\[M_v(w,\cdot)_{w\lambda}:\iota_{P(\Q_v)}^{G(\Q_v)}(\sigma,\lambda)\to\iota_{Q(\Q_v)}^{G(\Q_v)}(\sigma^w,w\lambda),\]
defined by integrals analogous to the global intertwining operator above (at least in the nonarchimedean case). Here $\sigma^w$ is defined similarly as in the global case above.\\
\indent If $v$ is finite and $\sigma$ is a smooth admissible representation of $M(\Q_v)$, then any $\phi_\lambda\in\iota_{P(\Q_v)}^{G(\Q_v)}(\sigma,\lambda)$ can be made to vary with $\lambda$ in a unique way such that $\phi_\lambda|_{K_v}$ is independent of $\lambda$, because of the Iwasawa decomposition. We say in this case that $\phi$ is a \it flat section \rm of the induction.\\
\indent If $\sigma$ is furthermore irreducible and unramified, then $\iota_{P(\Q_v)}^{G(\Q_v)}(\sigma,\lambda)$ has a unique up to scalar $K_v$-fixed vector; given a $K_v$-fixed vector $v^{\sph}$ in the space $V_\sigma$ of $\sigma$, there is a unique $\phi_\lambda^{\sph}\in\iota_{P(\Q_v)}^{G(\Q_v)}(\sigma,\lambda)$ such that $\phi_\lambda^{\sph}(k)=v^{\sph}$ for any $k\in K_v$. Then $\phi_\lambda^{\sph}$ is $K_v$-fixed and forms a flat section. If $w\in W(P,Q)$, then $M_v(w,\phi)_{w\lambda}$ is also $K_v$-fixed, and hence is a scalar multiple of the $K_v$-fixed vector $\phi_{w\lambda}^{w,\sph}\in\iota_{Q(\Q_v)}^{G(\Q_v)}(\sigma^w,w\lambda)$ given by the property that $\phi_{w\lambda}^{w,\sph}(k)=v^{\sph}$ for any $k\in K_v$ (recall that $\sigma$ and $\sigma^w$ act on the same space). If we let $\lambda$ vary in a flat section, this scalar multiple will vary, and it is possible to say how in particular cases when $P$ is maximal. In fact, there is a classical formula of Gindikin--Karpelevich which expresses this multiple in terms of local $L$-functions.
\subsubsection*{$L$-functions and intertwining operators}
\indent We will not need the local formula of Gindikin--Karpelevich here, but we will need a global consequence of it, which is at the heart of the Langlands--Shihidi method. We need to set up some notation before we can state it, however, and we do this now.\\
\indent Assume for the rest of this section that $G$ is split and $P$ is maximal. Let $B\subset P$ be a Borel subgroup of $G$ with Levi $T$, and fix a set $\Phi$ of positive simple roots for $T$ in $G$ that makes $B$ standard. Assume $P$ corresponds to the subset of $\Phi$ obtained by omitting a single simple root $\gamma$. Let $w_0$ be the unique element of the Weyl group of $T$ in $G$ which sends every root in $\Phi\backslash\{\gamma\}$ to positive simple roots, and which sends $\gamma$ to a negative root. If $P'$ is the standard maximal parabolic with Levi $w_0Mw_0$, then $w_0\in W(P,P')$.\\
\indent View $\gamma$ as an element of $\mf{a}_{P,0}^\vee$ and write
\[\tilde\gamma=\langle\rho_P,\gamma\rangle^{-1}\rho_P\]
where $\langle\cdot,\cdot\rangle$ is the usual pairing on $\mf{a}_{P,0}^\vee$ induced from the Killing form. Then $\mf{a}_{P,0}^\vee$ is one dimensional, generated by $\tilde\gamma$.\\
\indent Let $P^\vee$ be the parabolic subgroup of the dual group $G^\vee$ corresponding to the set of coroots associated with the simple roots in $\Phi\backslash\{\gamma\}$. The dual group $M^\vee$ is the Levi of $P^\vee$, and we let $N^\vee$ be the unipotent radical of $P^\vee$. The group $M^\vee$ acts on $\Lie(N^\vee)$ via the adjoint action. For $i>0$ an integer, let $V_i\subset\Lie(N^\vee)$ generated by the coroots $\beta^\vee$ for which $\langle\tilde\gamma,\beta^\vee\rangle=i$. Then each $V_i$ is a representation of $M^\vee$, and we denote the corresponding action of $M^\vee$ by $R_i$.
\begin{theorem}
\label{thmLSmethod}
Let $P$ be maximal and let $w_0$, $P'$, $\tilde\gamma$, and $R_i$ be as above. Let $S$ be a set of places which includes all the ramified places for $\tilde\pi$ and the archimedean place. For $v\notin S$, fix $v^{\sph}$ a nonzero $K_v$-fixed vector in the space of $\tilde\pi_v$. Let $s\in\C$ and let $\phi_{v,s}^{\sph}\in\iota_{P(\Q_v)}^{G(\Q_v)}(\tilde\pi_v,s\tilde\gamma)$ and $\phi_{v,s}^{w_0,\sph}\in\iota_{P'(\Q_v)}^{G(\Q_v)}(\tilde\pi_v^{w_0},s(w_0\tilde\gamma))$ be spherical sections defined as above so that $\phi_{v,s}^{\sph}(k)=v^{\sph}=\phi_{v,s}^{w_0,\sph}(k)$.\\
\indent Assume $\phi_s\in\iota_{P(\A)}^{G(\A)}(\tilde\pi,s\tilde\gamma)$ decomposes as $\otimes_v\phi_{v,s}$ where $\phi_{v,s}=\phi_{v,s}^{\sph}$ for $v\notin S$. Then we have the formula
\[M(\phi,w_0)_{s(w_0\tilde\gamma)}=\prod_{j=1}^m\frac{L^S(js,\tilde\pi,R_i^\vee)}{L^S(js+1,\tilde\pi,R_i^\vee)}\bigotimes_{v\notin S}\phi_{v,s}^{w_0,\sph}\otimes\bigotimes_{v\in S}M_v(\phi_{v,s},w_0)_{s(w_0\tilde\gamma)},\]
where $L^S$ denotes a partial $L$-function, away from the places of $S$.
\end{theorem}
\begin{proof}
See Shahidi \cite{shahidies}, Theorem 6.3.1.
\end{proof}
Thus the theorem above, in combination with Theorem \ref{thmconstterm}, will later allow us to compute constant terms of maximal parabolically induced Eisenstein series along the maximal parabolics from which they are induced.
\subsection{The Franke--Schwermer decomposition}
\label{sectfsdecomp}
Let $E$ be a finite dimensional irreducible representation of $G(\C)$. Then the annihilator of $E$ in the center of the universal enveloping algebra of $\mf{g}$ is an ideal, and we denote it by $\mc{J}_E$. Denote by $\mc{A}_E(G)$ the space of automorphic forms on $G(\A)$ which are annihilated by a power of $\mc{J}_E$, and which transform trivially under $A_G(\R)^\circ$. The forms in $\mc{A}_E(G)$ are the ones that can possibly contribute to the cohomology of $E$, as we will discuss later.\\
\indent In \cite{FS}, Franke and Schwermer wrote down a decomposition of $\mc{A}_E(G)$ into pieces defined by certain parabolic subgroups of $G$ and cuspidal automorphic representations of their Levis. This decomposition is a direct sum decomposition of $G(\A_f)\times(\mf{g}_0,K_\infty)$-modules, and we describe it in this section.\\
\indent First, given two parabolic subgroups of $G$ defined over $\Q$, we say that they are \it associate \rm if their Levis are conjugate by an element of $G(\Q)$. Let $\mc{C}$ be the set of equivalence classes for this relation. It is a finite set. If $P$ is a parabolic $\Q$-subgroup of $G$, let $[P]$ denote its equivalence class in $\mc{C}$.\\
\indent Now fix $P$ a parabolic $\Q$-subgroup of $G$ with Levi decomposition $P=MN$. Given another parabolic $\Q$-subgroup $Q=M_QN_Q$ of $G$, we say a function $f\in\mc{A}_E(G)$ is \it negligible along $Q$ \rm if for any $g\in G(\A)$, the function given by
\[m\mapsto f(mg),\qquad m\in M_Q(\Q)A_G(\R)^\circ\backslash M_Q(\A),\]
is orthogonal to the space of cuspidal functions on $M_Q(\Q)A_G(\R)^\circ\backslash M_Q(\A)$. Let $\mc{A}_{E,[P]}(G)$ be the subspace of all functions in $\mc{A}_E(G)$ which are negligible along any parabolic subgroup $Q\notin[P]$. It is a theorem of Langlands that
\[\mc{A}_E(G)=\bigoplus_{C\in\mc{C}}\mc{A}_{E,C}(G)\]
as $G(\A_f)\times(\mf{g}_0,K_\infty)$-modules. The summand $\mc{A}_{E,[G]}(G)$ is the space of cusp forms in $\mc{A}_E(G)$.\\
\indent Franke and Schwermer refine this decomposition even further using cuspidal automorphic representations of the Levis of the parabolics in each class $C\in\mc{C}$. We briefly recall how.\\
\indent Let $\varphi$ be an \it associate class of cuspidal automorphic representations of $M$. \rm We do not recall here the exact definition of this notion, referring instead to \cite{FS} Section 1.2, or \cite{LS} Section 1.3. Each $\varphi$ is a collection of irreducible representations of the groups $M_{P'}(\A)$ for each $P'\in[P]$ with Levi decomposition $P'=M_{P'}N_{P'}$, finitely many for each such $P'$, and each such representation $\pi$ must occur in $L_{\cusp}^2(M_{P'}(\Q)\backslash M_{P'}(\A),\chi_\pi)$ where $\chi_\pi$ is the central character of $\pi$. Conversely, any irreducible representation $\pi$ of $M(\A)$ with central character $\chi_\pi$ occurring in $L_{\cusp}^2(M(\Q)\backslash M(\A),\chi_\pi)$ determines a unique $\varphi$. We let $\Phi_{E,[P]}$ denote the set of all associate classes of cuspidal automorphic representations of $M$.\\
\indent Now given a $\varphi\in\Phi_{E,[P]}$, let $\pi$ be one of the representations comprising $\varphi$; say $\pi$ is a representation of the $\A$-points of a Levi $M_{P'}$ for $P'$ a parabolic associate to $P$. Form the space $W_{P',\tilde\pi}$ as in Section \ref{secES}. Let $d\chi_{\pi}$ be the differential of the central character of $\pi$ at the archimedean place, viewed as an element of $\mf{a}_{P',0}^\vee$. Then for any $\phi\in W_{P',\tilde\pi}$ we can form the Eisenstein series $E(\phi,\lambda)$, $\lambda\in\mf{a}_{P',0}^\vee$.\\
\indent Depending on the choice of $\phi$, the Eisenstein series $E(\phi,\lambda)$ may have a pole at $\lambda=d\chi_\pi$. Nevertheless, one can still take residues of $E(\phi,\lambda)$ at $\lambda=d\chi_\pi$ to obtain residual Eisenstein series. We let $\mc{A}_{E,[P],\varphi}(G)$ to be the collection of all possible Eisenstein series, residual Eisenstein series, and partial derivatives of such with respect to $\lambda$, evaluated at $\lambda=d\chi_\pi$, built from any $\phi\in W_{P',\tilde\pi}$. (For a more precise description of this space, see \cite{FS}, Section 1.3, or \cite{LS}, Section 1.4. There is also a more intrinsic definition of this space, defined without reference to Eisenstein series, in \cite{FS}, Section 1.2, or \cite{LS}, Section 1.4, which is proved to be equivalent to this description in \cite{FS}.) One can use the functional equation of Eisenstein series to show that the space $\mc{A}_{E,[P],\varphi}(G)$ is independent of the $\pi$ in $\varphi$ used to define it.\\
\indent We can now state the Franke--Schwermer decomposition of $\mc{A}_E(G)$.
\begin{theorem}[Franke--Schwermer \cite{FS}]
\label{thmfsdecomp}
There is a direct sum decomposition of $G(\A_f)\times(\mf{g}_0,K_\infty)$-modules
\[\mc{A}_E(G)=\bigoplus_{C\in\mc{C}}\bigoplus_{\varphi\in\Phi_{E,C}}\mc{A}_{E,C,\varphi}(G).\]
\end{theorem}
\subsection{Structure of the pieces of the Franke--Schwermer decomposition}
We introduce in this section certain $G(\A_f)\times(\mf{g}_0,K_\infty)$-modules, whose structures as such modules are explicit, and explain how they can be related to the pieces of the Franke--Schwermer decomposition introduced just above. Almost everything in this section is done in Franke's paper \cite{franke}, pp. 218, 234, but without taking into consideration the associate classes $\varphi$.\\
\indent We consider again a parabolic $\Q$-subgroup $P$ of $G$ with Levi decomposition $P=MN$.
As before, let us fix $\pi$ a cuspidal automorphic representation of $M(\A)$, and let $\tilde\pi$ be its unitarization, as in Section \ref{secES}. Then $\tilde\pi$ occurs in $L_{\cusp}^2(M(\A)A_P(\R)^\circ\backslash M(\A))$. For brevity, let us write $V[\tilde\pi]$ for the smooth, $K$-finite vectors in the $\tilde\pi$-isotypic component of $L_{\cusp}^2(M(\A)A_P(\R)^\circ\backslash M(\A))$. Then $V[\tilde\pi]$ is a $M(\A_f)\times(\mf{m}_0,K_\infty\cap P(\R))$-module, and we extend this structure to one of a $P(\A_f)\times(\mf{p}_0,K_\infty\cap P(\R))$-module by letting $\mf{a}_{P,0}^\vee$ and $\mf{n}$ act trivially, as well as $A_P(\A_f)$ and $N(\A_f)$.\\
\indent Fix for the rest of this section a point $\mu\in\mf{a}_{P,0}^\vee$. Let $\Sym(\mf{a}_{P,0})_\mu$ be the symmetric algebra on the vector space $\mf{a}_{P,0}$; we view this space as the space of differential operators on $\mf{a}_{P,0}^\vee$ at the point $\mu$. So if $H(\lambda)$ is a holomorphic function on $\mf{a}_{P,0}^\vee$, then $D\in\Sym(\mf{a}_{P,0})_\mu$ acts on $H$ by taking a sum of iterated partial derivatives of $H$ and evaluating the result at the point $\mu$. So in this way, every $D\in\Sym(\mf{a}_{P,0})_\mu$ can be viewed as a distribution on holomorphic functions on $\mf{a}_{P,0}^\vee$ supported at the point $\mu$.\\
\indent With this point of view, these distributions can be multiplied by holomorphic functions on $\mf{a}_{P,0}^\vee$; just multiply the test function by the given holomorphic function before evaluating the distribution. With this in mind, we can define an action of $\mf{a}_{P,0}^\vee$ on $\Sym(\mf{a}_{P,0})_\mu$ by
\[(XD)(f)=D(\langle X,\cdot\rangle f),\qquad X\in\mf{a}_{P,0},\,\, D\in\Sym(\mf{a}_{P,0})_\mu.\]
We also let $\mf{m}_0$ and $\mf{n}$ act trivially on $\Sym(\mf{a}_{P,0})_\mu$, which gives us an action of $\mf{p}_0$ on $\Sym(\mf{a}_{P,0})_\mu$.\\
\indent We also let $K_\infty\cap P(\R)$ act trivially on $\Sym(\mf{a}_{P,0})_\mu$. Since the Lie algebra of $K_\infty\cap P(\R)$ lies in $\mf{m}_0$, this is consistent with the $\mf{p}_0$ action just defined and makes $\Sym(\mf{a}_{P,0})_\mu$ a $(\mf{p}_0,K_\infty\cap P(\R))$-module.\\
\indent Finally, let $P(\A_f)$ act on $\Sym(\mf{a}_{P,0})_\mu$ by the formula
\[(pD)(f)=D(e^{\langle H_P(p),\cdot\rangle}f),\qquad p\in P(\A_f),\,\, D\in\Sym(\mf{a}_{P,0})_\mu.\]
Then with the actions just defined, $\Sym(\mf{a}_{P,0})_\mu$ gets the structure of a $P(\A_f)\times(\mf{p}_0,K_\infty\cap P(\R))$-module.\\
\indent Now we form the tensor product $V[\tilde\pi]\otimes\Sym(\mf{a}_{P,0})_\mu$, which carries a natural $P(\A_f)\times(\mf{p},K_\infty\cap P(\R))$-module structure coming from those on the two factors. We will consider in what follows the induced $G(\A_f)\times(\mf{g}_0,K_\infty)$-module
\[\Ind_{P(\A)}^{G(\A)}(V[\tilde\pi]\otimes\Sym(\mf{a}_{P,0})_\mu).\]
This space turns out to be isomorphic to another $G(\A_f)\times(\mf{g}_0,K_\infty)$-module, which we now describe.\\
\indent Let $W_{P,\tilde\pi}$ be the induction space introduced in Section \ref{secES}; it is the unnormalized parabolic induction of the space $V[\tilde\pi]$ above. Form the tensor product
\[W_{P,\tilde\pi}\otimes\Sym(\mf{a}_{P,0})_\mu.\]
While the first factor in this tensor product is a $G(\A_f)\times(\mf{g}_0,K_\infty)$-module, the second is only a $P(\A_f)\times(\mf{p}_0,K_\infty\cap P(\R))$-module, and so we do not immediately get a $G(\A_f)\times(\mf{g}_0,K_\infty)$-module structure on the tensor product. However, one can endow this space with a $G(\A_f)\times(\mf{g}_0,K_\infty)$-module structure by viewing it as a space of distributions in a manner to be described now.\\
\indent We first introduce the space of functions on which we will consider distributions. These will be functions on $G(\A)\times\mf{a}_{P,0}^\vee$. Let us write $g$ for a variable in $G(\A)$ and $\lambda$ for a variable in $\mf{a}_{P,0}^\vee$. Let $\mc{S}$ be the space of functions $f(g,\lambda)$ on $G(\A)\times\mf{a}_{P,0}^\vee$ which are smooth and compactly supported in the variable $g$ when $\lambda$ is fixed, and which are holomorphic in the variable $\lambda$ when $g$ is fixed. Then we consider the space $\mc{D}(\mc{S})$ of distributions on $\mc{S}$ which are compactly supported in the variable $\lambda$.\\
\indent The space $W_{P,\tilde\pi}\otimes\Sym(\mf{a}_{P,0}^\vee)_\mu$ embeds naturally as a subspace of $\mc{D}(\mc{S})$. In fact, we can identify the simple tensor $\phi\otimes D$, where $\phi\in W_{P,\tilde\pi}$ and $D\in\Sym(\mf{a}_{P,0}^\vee)_\mu$, with the distribution given on functions $f\in\mc{S}$ by
\[(\phi\otimes D)(f)=D\left(\int_{G(\A)}\phi(g)f(g,\cdot)\,dg\right).\]
Here, $D$ is being viewed as a distribution on holomorphic functions on $\mf{a}_{P,0}^\vee$ as described above, so indeed the right hand side of this equality is a complex number.\\
\indent Now we describe a $G(\A_f)\times(\mf{g}_0,K_\infty)$-module structure on the space $W_{P,\tilde\pi}\otimes\Sym(\mf{a}_{P,0})_\mu$ through formulas that make sense in $\mc{D}(\mc{S})$. Let us give these formulas and then make comments on them afterward. For $\phi\in W_{P,\tilde\pi}$ and $D\in\Sym(\mf{a}_{P,0})_\mu$, we consider $\phi\otimes D$ as a distribution in the variables $(g,\lambda)$ and we define:
\[(X(\phi\otimes D))(g,\lambda)=((X\phi)\otimes D)(g,\lambda) + (\langle XH_P(g),\lambda\rangle(\phi\otimes D))(g,\lambda),\]
for $X\in\mf{g}_0$,
\[(k(\phi\otimes D))(g,\lambda)=(\phi\otimes D)(gk,\lambda),\]
for $k\in K_\infty$, and
\[(h(\phi\otimes D))(g,\lambda)=(e^{\langle H_P(gh)-H_P(g),\lambda\rangle}(\phi\otimes D))(gh,\lambda),\]
for $h\in G(\A_f)$.\\
\indent Now in the formulas defining the actions of $G(\A_f)$ and $\mf{g}_0$, there are distributions on the right hand side that have been multiplied by functions depending on both $g$ and $\lambda$. Therefore, it is not immediately obvious that these expressions define elements of the image of $W_{P,\tilde\pi}\otimes\Sym(\mf{a}_{P,0})_\mu$ in $\mc{D}(\mc{S})$; that is, it is not completely clear that these expressions can be written as a finite sum of simple tensors in $W_{P,\tilde\pi}\otimes\Sym(\mf{a}_{P,0})_\mu$. However, using properties of the function $H_P$, this can be checked. We omit the verification here for sake of brevity.\\
\indent Now we can relate the two $G(\A_f)\times(\mf{g}_0,K_\infty)$-modules defined in this section. We have the following proposition, whose proof we again omit.
\begin{proposition}
\label{propwpind}
There is an isomorphism of $G(\A_f)\times(\mf{g}_0,K_\infty)$-modules
\[W_{P,\tilde\pi}\otimes\Sym(\mf{a}_{P,0})_\mu\cong\Ind_{P(\A)}^{G(\A)}(V[\tilde\pi]\otimes\Sym(\mf{a}_{P,0})_\mu).\]
More generally, if $E$ is a finite dimensional representation of $G(\C)$, then we also have an isomorphism
\[W_{P,\tilde\pi}\otimes\Sym(\mf{a}_{P,0})_\mu\otimes E\cong\Ind_{P(\A)}^{G(\A)}(V[\tilde\pi]\otimes\Sym(\mf{a}_{P,0})_\mu\otimes E),\]
where on the left hand side, $E$ is being viewed as a $(\mf{g}_0,K_\infty)$-module, and on the right, it is viewed as a $(\mf{p}_0,K_\infty\cap P(\R))$-module by restriction.
\end{proposition}
The reason we introduce the representation $E$ in the second part of this proposition will become more apparent when we discuss cohomology later.\\
\indent Now we come back to Eisenstein series. Assume $\pi$ is such that there is an irreducible finite dimensional representation $E$ of $G(\C)$ such that the associate class $\varphi$ containing $\pi$ is in $\Phi_{E,[P]}$. Then we can construct elements of the piece $\mc{A}_{E,[P],\varphi}(G)$ of the Franke--Schwermer decomposition from Section \ref{sectfsdecomp} from elements of $W_{P,\tilde\pi}\otimes\Sym(\mf{a}_{P,0})_\mu$ using Eisenstein series as follows.\\
\indent Recall that, in the notation of Section \ref{secES}, we have
\[W_{P,\tilde\pi}\cong\Ind_{P(\A)}^{G(\A)}(V[\tilde\pi])=\iota_{P(\A)}^{G(\A)}(V[\tilde\pi],-\rho_P).\]
Elements $\phi\in\iota_{P(\A)}^{G(\A)}(V[\tilde\pi],-\rho_P)$ fit into flat sections $\phi_\lambda\in\iota_{P(\A)}^{G(\A)}(V[\tilde\pi],\lambda)$ where $\lambda$ varies in $\mf{a}_{P,0}^\vee$. Then for such $\phi$ we have $\phi=\phi_{-\rho_P}$. In what follows, we will identify elements of $W_{P,\tilde\pi}$ with elements of $\iota_{P(\A)}^{G(\A)}(V[\tilde\pi],-\rho_P)$, and then use this notation to vary them in flat sections.\\
\indent Let $d\chi_\pi$ denote the differential of the archimedean component of the central character of $\pi$. Then as in Section \ref{secES}, if we are given $\phi\in W_{P,\tilde\pi}$, we can form the Eisenstein series $E(\phi,\lambda)$ for $\lambda$ varying in $\mf{a}_{P,0}^\vee$. This is a family of automorphic forms which varies meromorphically in $\lambda$. Let $h_0$ be a holomorphic function on $\mf{a}_{P,0}^\vee$ such that, for any $\phi\in W_{P,\tilde\pi}$, the product $h_0(\lambda)E(\phi,\lambda)$ is holomorphic near $\lambda=d\chi_\pi$. Then we define a map
\[\mc{E}_{h_0}:W_{P,\tilde\pi}\otimes\Sym(\mf{a}_{P,0})_{d\chi_\pi+\rho_P}\to\mc{A}_{E,[P],\varphi}(G)\]
by
\[\phi\otimes D\mapsto D(h_0(\lambda)E(\phi,\lambda));\]
in other words, this map forms an Eisenstein series according to $\phi$, multiplies it by $h_0(\lambda)$ in order to cancel any poles, and then differentiates the result at the point $\lambda=d\chi_\pi$ according to $D$.\\
\indent The map $\mc{E}_{h_0}$ is surjective by our definition of $\mc{A}_{E,[P],\varphi}(G)$. If all the Eisenstein series $E(\phi,\lambda)$, for $\phi\in W_{P,\tilde\pi}$, are holomorphic at $\lambda=d\chi_\pi$, then we write $\mc{E}=\mc{E}_1$ for the map just defined with $h_0(\lambda)=1$.
\begin{proposition}
\label{propEh0}
The map $\mc{E}_{h_0}:W_{P,\tilde\pi}\otimes\Sym(\mf{a}_{P,0})_{d\chi_\pi+\rho_P}\to\mc{A}_{E,[P],\varphi}(G)$ defined just above is a surjective map of $G(\A_f)\times(\mf{g}_0,K_\infty)$-modules. Furthermore, if all the Eisenstein series $E(\phi,\lambda)$ arising from $\phi\in W_{P,\tilde\pi}$ are holomorphic at $\lambda=d\chi_\pi$, then the map $\mc{E}$ is an isomorphism.
\end{proposition}
\begin{proof}
To check that $\mc{E}_{h_0}$ is a map of $G(\A_f)\times(\mf{g}_0,K_\infty)$-modules, one just needs to use the formulas defining the $G(\A_f)\times(\mf{g}_0,K_\infty)$-module structure on $W_{P,\tilde\pi}\otimes\Sym(\mf{a}_{P,0})_\lambda$ and show they are preserved when forming Eisenstein series and taking derivatives; this can be checked when $\lambda$ is in the region of convergence for the Eisenstein series, and then this extends to all $\lambda$ by analytic continuation. We omit the precise details of this check.\\
\indent For the second claim in the proposition, that $\mc{E}$ is an isomorphism, this follows essentially from Theorem 14 in Franke's paper \cite{franke}; this theorem implies that $\mc{E}$ injective, since it equals the restriction of Franke's mean value map $\mathbf{MW}$ to $W_{P,\tilde\pi}\otimes\Sym(\mf{a}_{P,0})_{d\chi_\pi+\rho_P}$. Whence by surjectivity and the first part of the proposition, we are done.
\end{proof}
The spaces $\mc{A}_{E,[P],\varphi}(G)$ carry a filtration by $G(\A_f)\times(\mf{g}_0,K_\infty)$-modules which is due to Franke. We will not need for this paper the precise definition of this filtration, but just a rough description of its graded pieces. This is described in the following theorem.
\begin{theorem}
\label{thmfrfil}
There is a decreasing filtration
\[\dotsb\supset\Fil^i\mc{A}_{E,C,\varphi}(G)\supset\Fil^{i+1}\mc{A}_{E,C,\varphi}(G)\supset\dotsb\]
of $G(\A_f)\times(\mf{g}_0,K_\infty)$-modules on $\mc{A}_{E,C,\varphi}(G)$, for which we have
\[\Fil^0\mc{A}_{E,C,\varphi}(G)=\mc{A}_{E,C,\varphi}(G)\]
and
\[\Fil^m\mc{A}_{E,C,\varphi}(G)=0\]
for some $m>0$ (depending on $\varphi$), and whose graded pieces have the property described below.\\
\indent Fix $\pi$ in $\varphi$, and say $\pi$ is a representation of the $\A$-points of a Levi $M$ of a parabolic $P$ in $C$. Let $d\chi_\pi$ be the differential of the archimedean component of the central character of $\pi$. Let $\mc{M}$ be the set of quadruples $(Q,\nu,\Pi,\mu)$ where:
\begin{itemize}
\item $Q$ is a parabolic subgroup of $G$ which contains $P$;
\item $\nu$ is an element of $(\mf{a}_P\cap\mf{m}_{Q,0})^\vee$;
\item $\Pi$ is an automorphic representation of $M(\A)$ occurring in
\[L_{\disc}^2(M_Q(\Q)A_Q(\R)^\circ\backslash M_Q(\A))\]
and which is spanned by values at, or residues at, the point $\nu$ of Eisenstein series parabolically induced from $(P\cap M_Q)(\A)$ to $M_Q(\A)$ by representations in $\varphi$; and
\item $\mu$ is an element of $\mf{a}_{Q,0}^\vee$ whose real part in $\Lie(A_{G}(\R)\backslash A_{M_Q}(\R))$ is in the closure of the positive chamber, and such that the following relation between $\mu$, $\nu$ and $\pi$ holds: Let $\lambda_{\tilde\pi}$ be the infinitesimal character of the archimedean component of $\tilde\pi$. Then
\[\lambda_{\tilde\pi}+\nu+\mu\]
may be viewed as a collection of weights of a Cartan subalgebra of $\mf{g}_0$, and the condition we impose is that these weights are in the support of the infinitesimal character of $E$.
\end{itemize}
For such a quadruple $(Q,\nu,\Pi,\mu)\in\mc{M}$, let $V[\Pi]$ denote the $\Pi$-isotypic component of the space
\[L_{\disc}^2(M_Q(\Q)A_Q(\R)^\circ\backslash M_Q(\A))\cap\mc{A}_{E,[P\cap M_Q],\varphi|_{M_Q}}(M_P).\]
Then the property of the graded pieces of the filtration above is that, for every $i$ with $0\leq i<m$, there is a subset $\mc{M}_\varphi^i\subset\mc{M}$ and an isomorphism of $G(\A_f)\times(\mf{g}_0,K_\infty)$-modules
\[\Fil^i\mc{A}_{E,C,\varphi}(G)/\Fil^{i+1}\mc{A}_{E,C,\varphi}(G)\cong\bigoplus_{(Q,\nu,\Pi,\mu)\in\mc{M}_\varphi^i}\Ind_{Q(\A)}^{G(\A)}(V[\Pi]\otimes\Sym(\mf{a}_{Q,0})_{\mu+\rho_Q}).\]
\end{theorem}
\begin{proof}
While this essentially follows again from the work of Franke \cite{franke}, in this form, this theorem is a consequence of Theorem 4 in the paper of Grobner \cite{grobner}; the latter paper takes into account the presence of the class $\varphi$ while the former does not.
\end{proof}
\begin{remark}
In the context of Proposition \ref{propEh0} and Theorem \ref{thmfrfil}, when all the Eisenstein series $E(\phi,\lambda)$ arising from $\phi\in W_{P,\tilde\pi}$ are holomorphic at $\lambda=d\chi_\pi$, what happens is that the filtration of Theorem \ref{thmfrfil} collapses to a single step. The nontrivial piece of this filtration is then given by $\Ind_{P(\A)}^{G(\A)}(V[\tilde\pi]\otimes\Sym(\mf{a}_{P,0})_{d\chi_\pi+\rho_P})$ through the map $\mc{E}$ along with the isomorphism of Proposition \ref{propwpind}.
\end{remark}
When $P$ is a maximal parabolic, the filtration of Theorem \ref{thmfrfil} becomes particularly simple. To describe it, we set some notation.\\
\indent Assume $P$ is maximal. If $\tilde\pi$ is a unitary cuspidal automorphic representation of $M$ and $s\in\C$ with $\re(s)>0$, let us write
\[\mc{L}_{P(\A)}^{G(\A)}(\tilde\pi,s)\]
for the Langlands quotient of
\[\iota_{P(\A)}^{G(\A)}(\tilde\pi,2s\rho_P).\]
One definition of this is that it is the quotient of the induction above by the kernel of the intertwining operator
\[M(\cdot,w_0):\iota_{P(\A)}^{G(\A)}(\tilde\pi,s)\to\iota_{P'(\A)}^{G(\A)}(\tilde\pi,-s)\]
of Section \ref{secES}. Here, if we fix a minimal parabolic contained in $P$, then $w_0$ is the Weyl element that sends every simple root in $M$ to another positive simple root, and which sends the positive simple root not in $M$ to a negative root, and $P'$ is the standard parabolic with Levi $w_0Mw_0$. Then we have
\begin{theorem}[Grbac \cite{grbac}]
\label{thmgrbac}
In the setting above, with $P$ maximal and $\re(s)>0$, assume $\tilde\pi$ defines an associate class $\varphi\in\Phi_{E,[P]}$. If any of the Eisenstein series $E(\phi,\lambda)$ coming from $\phi\in W_{\tilde\pi}$ have a pole at $\lambda=2s_0\rho_P$, then there is an exact sequence of $G(\A_f)\times(\mf{g}_0,K_\infty)$-modules as follows:
\[0\to\mc{L}_{P(\A)}^{G(\A)}(\tilde\pi,s)\to\mc{A}_{E,[P],\varphi}(G)\to\Ind_{P(\A)}^{G(\A)}(V[\tilde\pi]\otimes\Sym(\mf{a}_{P,0})_{(2s+1)\rho_P})\to 0.\]
\end{theorem}
\begin{proof}
This follows from Theorem 3.1 in the paper of Grbac \cite{grbac}.
\end{proof}
\section{Cohomology}
\label{chcoh}
We now would like to study the cohomology of the pieces of the Franke--Schwermer decomposition. We can reduce this to studying the parabolically induced representations introduced in the previous chapter and applying a classical argument involving the Kostant decomposition, as in \cite{BW}, Theorem III.3.3. We start with a general discussion of cohomology.
\subsection{The cohomology of the space of automorphic forms}
We continue to use the notation set in the introduction, and in particular, we will resume working with our reductive $\Q$-group $G$. We have our maximal compact subgroup $K_\infty\subset G(\R)$, and we fix an open subgroup $K_\infty'$ of $K_\infty$. Then we necessarily have $K_\infty^\circ\subset K_\infty'\subset K_\infty$.\\
\indent We will be interested in the $(\mf{g}_0,K_\infty')$-cohomology of the space of automorphic forms on $G(\A)$. By Franke's resolution of Borel's conjecture (\cite{franke}, Theorem 18), this cohomology space (for suitable $K_\infty'$) computes the cohomology of certain locally symmetric spaces attached to $G$, and is therefore of arithmetic interest.\\
\indent So as before, let $E$ be an irreducible, finite dimensional, complex representation of $G(\C)$. We view $E$ as a $(\mf{g}_0,K_\infty)$-module via its restriction to $G(\R)$, and hence as a $G(\A_f)\times(\mf{g}_0,K_\infty)$-module by giving it a trivial $G(\A_f)$ action. Our goal is to study the $(\mf{g}_0,K_\infty')$-cohomology space
\[H^i(\mf{g}_0,K_\infty';\mc{A}_E(G)\otimes E)\]
for any $i$, which is naturally a $G(\A_f)$-module; see the standard reference by Borel--Wallach \cite{BW} for the definition of $(\mf{g}_0,K_\infty')$-cohomology and discussions of many of its most important properties.\\
\indent Actually, the cohomology space above is smooth and admissible as a $G(\A_f)$-module, as can be seen by comparing it to the cohomology of certain local systems on the locally symmetric spaces attached to $G$. By the results recalled in Section \ref{sectfsdecomp} and the Franke--Schwermer decomposition (Theorem \ref{thmfsdecomp}) we have a direct sum decomposition as $G(\A_f)$-modules
\[H^i(\mf{g}_0,K_\infty';\mc{A}_E(G)\otimes E)=\bigoplus_{C\in\mc{C}}\bigoplus_{\varphi\in\Phi_{E,C}}H^i(\mf{g}_0,K_\infty';\mc{A}_{E,C,\varphi}(G)\otimes E).\]
Each summand in the decomposition above is therefore a smooth, admissible $G(\A_f)$-module, and although there may be infinitely summands on the right hand side which don't vanish, only finitely many of them have nonzero $K_f'$-invariants for any given open compact subgroup $K_f'\subset K_f$.\\
\indent Let us write
\[H_{\cusp}^i(\mf{g}_0,K_\infty';\mc{A}_E(G)\otimes E)=\bigoplus_{\varphi\in\Phi_{E,[G]}}H^i(\mf{g}_0,K_\infty';\mc{A}_{E,[G],\varphi}(G)\otimes E)\]
for the \it cuspidal cohomology \rm of $E$. This is also the same as
\[H^i(\mf{g}_0,K_\infty';L_{\cusp}^2(G(\Q)A_G(\R)^\circ\backslash G(\A))\otimes E).\]
The natural complement to the cuspidal cohomology in the decomposition above is called the \it Eisenstein cohomology, \rm i.e.,
\[H_{\Eis}^i(\mf{g}_0,K_\infty';\mc{A}_E(G)\otimes E)=\bigoplus_{\substack{C\in\mc{C} \\ C\ne[G]}}\bigoplus_{\varphi\in\Phi_{E,C}}H^i(\mf{g}_0,K_\infty';\mc{A}_{E,C,\varphi}(G)\otimes E).\]
If $P$ is a proper parabolic subgroup of $G$ defined over $\Q$, let us define the \it $[P]$-Eisenstein cohomology \rm to be the summand corresponding to the class $[P]$, so
\[H_{[P]}^i(\mf{g}_0,K_\infty';\mc{A}_E(G)\otimes E)=\bigoplus_{\varphi\in\Phi_{E,[P]}}H^i(\mf{g}_0,K_\infty';\mc{A}_{E,[P],\varphi}(G)\otimes E).\]
\indent Now let $\mc{H}_G$ be the Hecke algebra of smooth, compactly supported, complex-valued functions on $G(\A_f)$,
\[\mc{H}_G=C_c^\infty(G(\A_f)).\]
Then $\mc{H}_G$ acts on any smooth, admissible $G(\A_f)$-module $(\sigma,V)$ via convolution. Furthermore, for any $f\in\mc{H}_G$ and any open compact subgroup $K_f'\subset K_f$ for which $f$ is $K_f'$-biinvariant, we can consider the trace $\tr(f|V^{K_f'})$ of $f$ acting as a linear operator on the $K_f'$ invariants of $V$. This is independent of the choice of $K_f'$ and defines an association
\[f\mapsto J_{\sigma}(f)=\tr(f|V^{K_f'}),\]
and we call $J_\sigma$ the \it character distribution \rm associated with $\sigma$. An irreducible admissible $G(\A_f)$-module is determined by its character distribution.
\begin{definition}
\label{defmult}
The \it multiplicity \rm of an irreducible admissible $G(\A_f)$-module $\sigma$ in the $i$th $(\mf{g}_0,K_\infty')$-cohomology of $\mc{A}_E(G)$ is the nonnegative integer $m^i(\sigma,K_\infty',E)$ such that
\[\tr(f|H^i(\mf{g}_0,K_\infty';\mc{A}_E(G)\otimes E)^{K_f'})=\sum_{\sigma}m^i(\sigma,K_\infty',E)J_\sigma(f)\]
for any $f\in\mc{H}_G$ and any open compact subgroup $K_f'\subset K_f$ for which $f$ is $K_f'$-biinvariant. Here, on the right hand side, the sum is over all irreducible admissible $G(\A_f)$-modules.\\
\indent Similarly we define $m_{\cusp}^i(\sigma,K_\infty',E)$, $m_{\Eis}^i(\sigma,K_\infty',E)$, and $m_{[P]}^i(\sigma,K_\infty',E)$ for a proper parabolic $\Q$-subgroup $P$ of $G$, by formulas similar to the one above, namely:
\[\tr(f|H_{\cusp}^i(\mf{g}_0,K_\infty';\mc{A}_E(G)\otimes E)^{K_f'})=\sum_{\sigma}m_{\cusp}^i(\sigma,K_\infty',E)J_\sigma(f),\]
\[\tr(f|H_{\Eis}^i(\mf{g}_0,K_\infty';\mc{A}_E(G)\otimes E)^{K_f'})=\sum_{\sigma}m_{\Eis}^i(\sigma,K_\infty',E)J_\sigma(f),\]
and
\[\tr(f|H_{[P]}^i(\mf{g}_0,K_\infty';\mc{A}_E(G)\otimes E)^{K_f'})=\sum_{\sigma}m_{[P]}^i(\sigma,K_\infty',E)J_\sigma(f).\]
We call these numbers, respectively, the \it cuspidal multiplicity, \rm the \it Eisenstein multiplicity, \rm and the \it $[P]$-Eisenstein multiplicity \rm of $\sigma$ in the $i$th cohomology of $E$.
\end{definition}
It follows immediately from the definitions that
\[m^i(\sigma,K_\infty',E)=m_{\cusp}^i(\sigma,K_\infty',E)+m_{\Eis}^i(\sigma,K_\infty',E)\]
and
\[m_{\Eis}^i(\sigma,K_\infty,E)=\sum_{\substack{C\in\mc{C} \\ C\ne[G]}}m_C^i(\sigma,K_\infty',E).\]
\indent The goal in the following will be to precisely compute, for certain choices of $G$, the multiplicity of the Langlands quotients of certain induced representations, induced from maximal parabolic subgroups, in the cohomology of particular $E$'s. These induced representations will show up in Eisenstein cohomology naturally, as we will explain in the next section. Perhaps more interestingly is that these Langlands quotients can also occur in cuspidal cohomology, and we will see examples of this in the cases of $\GSp_4$ and $\G_2$ later.
\subsection{The cohomology of induced representations}
We now calculate the cohomology of representations of $G$ that are parabolically induced from automorphic representations of Levi subgroups, and hence compute the cohomology of the graded pieces of the Franke filtration described in Theorem \ref{thmfrfil}. The computations done in this section were essentially carried out by Franke in \cite{franke}, Section 7.4, but not in so much detail. We fill in some of the details and give a sharper result, which we can give because we are focusing on one representation of the Levi at a time, and we can do this because we have access to the Franke--Schwermer decomposition, Theorem \ref{thmfsdecomp}. The method is essentially that of the proof of Theorem III.3.3 in \cite{BW}. This method also appears in the computations of Grbac--Grobner \cite{GG} and Grbac--Schwermer \cite{GS}.\\
\indent Let $P\subset G$ be a parabolic subgroup defined over $\Q$ with Levi decomposition $P=MN$. Fix an automorphic representation $\pi$ of $M(\A)$ with central character $\chi_\pi$, occurring in
\[L_{\disc}^2(M(\Q)\backslash M(\A),\chi_\pi).\]
Then the unitarization $\tilde\pi$ occurs in
\[L_{\disc}^2(M(\Q)A_P(\R)^\circ\backslash M(\A)).\]
Let $d\chi_\pi$ denote the differential of the archimedean component of $\chi_\pi$. Fix also an irreducible finite dimensional representation $E$ of $G(\C)$.\\
\indent As before, fix a compact subgroup $K_\infty'$ of $G(\R)$ such that $K_\infty^\circ\subset K_\infty'\subset K_\infty$. We will compute the $(\mf{g}_0,K_\infty')$-cohomology space
\[H^i(\mf{g}_0,K_\infty';\Ind_{P(\A)}^{G(\A)}(\tilde\pi\otimes\Sym(\mf{a}_{P,0})_{d\chi_\pi+\rho_P})\otimes E)\]
in terms of $(\mf{m}_0,K_\infty'\cap P(\R))$-cohomology spaces attached to $\pi$. We will require the following lemma.
\begin{lemma}
\label{lemcohsym}
Let $\mu,\mu'\in\mf{a}_{P,0}^\vee$. Let $\C_{\mu'}$ denote the one dimensional $\mf{a}_{P,0}$-module on which $X\in\mf{a}_{P,0}$ acts through multiplication by $\langle X,\mu'\rangle$. Then there is an isomorphism of $P(\A_f)$-modules
\[H^i(\mf{a}_{P,0},\Sym(\mf{a}_{P,0})_\mu\otimes\C_{\mu'})\cong\begin{cases}
\C(e^{\langle H_P(\cdot),\mu\rangle}) & \textrm{if }\lambda=-\mu\textrm{ and }i=0;\\
0 & \textrm{if }\lambda\ne-\mu\textrm{ or }i>0.
\end{cases}\]
Here, $\C(e^{\langle H_P(\cdot),\mu\rangle})$ is just the one dimensional representation of $P(\A_f)$ on which $p\in P(\A_f)$ acts via $e^{\langle H_P(p),\mu\rangle}$.
\end{lemma}
\begin{proof}
It will be convenient to work in coordinates. So let $\lambda=(\lambda_1,\dotsc,\lambda_r)$ be coordinates on $\mf{a}_{P,0}^\vee$; this is the same as fixing a basis of $\mf{a}_{P,0}$. Then the elements of $\Sym(\mf{a}_{P,0})_\mu$ may be viewed as polynomials in the variables $\lambda_1,\dotsc,\lambda_r$.\\
\indent Let $\alpha=(\alpha_1,\dotsc,\alpha_r)$ be a multi-index. By definition, the monomial $\lambda^\alpha=\lambda_1^{\alpha_1}\dotsb\lambda_r^{\alpha_r}$ acts as a distribution on holomorphic functions $f$ on $\mf{a}_{P,0}^\vee$ via the formula
\[\lambda^\alpha f=\frac{\partial^\alpha}{\partial\lambda^\alpha}f(\lambda)|_{\lambda=\mu}.\]
Also by definition, if $X\in\mf{a}_{P,0}$, then $X\lambda^\alpha$ acts as
\[(X\lambda^\alpha)f=\frac{\partial^\alpha}{\partial\lambda^\alpha}(\langle X,\lambda\rangle f(\lambda))|_{\lambda=\mu}.\]
\indent Let $P(\lambda)$ be a polynomial in $\lambda$. Then a quick induction using the above formulas shows that $X\in\mf{a}_{P,0}$ acts on $P(\lambda)$ as
\[X(P(\lambda))=\langle X,\mu\rangle P(\lambda)+\sum_{i=1}^r\frac{\partial}{\partial\lambda_i}P(\lambda).\]
Hence $X$ acts on the element $P(\lambda)\otimes 1$ in $\Sym(\mf{a}_{P,0}^\vee)_\mu\otimes\C_{\mu'}$ by
\[X(P(\lambda)\otimes 1)=\langle X,\mu+\mu'\rangle (P(\lambda)\otimes 1)+\sum_{i=1}^r\left(\frac{\partial}{\partial\lambda_i}P(\lambda)\otimes 1\right).\]
It follows from this that if $X_1,\dotsc,X_r$ is the basis of $\mf{a}_{P,0}$ corresponding to the coordinates $\lambda_1,\dotsc,\lambda_r$, then the decomposition
\[\mf{a}_{P,0}=\C X_1\oplus\dotsb\oplus\C X_r\]
realizes $\Sym(\mf{a}_{P,0})_\mu\otimes\C_{\mu'}$ as an exterior tensor product of analogous single-variable symmetric powers:
\[\Sym(\mf{a}_{P,0})_\mu\otimes\C_{\mu'}\cong(\Sym(\C X_1)_{\mu_1}\otimes\C_{\mu_1'})\otimes\dotsb\otimes(\Sym(\C X_r)_{\mu_r}\otimes\C_{\mu_r'}),\]
where $\mu_i,\mu_i'\in(\C X_i)^\vee$ are the $i$th components of $\mu,\mu'$ in the dual basis of $\mf{a}_{P,0}^\vee$ to $X_1,\dotsc,X_r$. To be explicit, here the space $\Sym(\C X_i)_{\mu_i}$ can be identified as the space of polynomials in the variable $\lambda_i$ with the structure of a module over the one-dimensional abelian Lie algebra $\C X_i$ given by
\[X_i(\lambda_i^n)=\langle X_i,\mu_i\rangle +n\lambda_i^{n-1}.\]
By the K\"unneth formula, if we ignore for now the $P(\A_f)$-action, we then reduce to checking the one-dimensional analog of the lemma, that
\[H^i(\C X_i,\Sym(\C X_i)_{\mu_i}\otimes\C_{\mu_i'})\cong\begin{cases}
\C & \textrm{if }\mu'=-\mu\textrm{ and }i=0;\\
0 & \textrm{if }\lambda\ne-\mu\textrm{ or }i>0.
\end{cases}\]
\indent To check this formula, we first note that by definition of Lie algebra cohomology, the space $H^*(\C X_i,\Sym(\C X_i)_{\mu_i}\otimes\C_{\mu_i'})$ is the cohomology of the complex
\[\Sym(\C X_i)_{\mu_i}\otimes\C_{\mu_i'}\to\hom_\C(\C X_i,\Sym(\C X_i)_{\mu_i}\otimes\C_{\mu_i'})\to 0\to\dotsb,\]
where the map between the first two terms is given by
\[(P\otimes 1)\mapsto(X_i\mapsto X_i(P\otimes 1)).\]
If $\mu'\ne -\mu$, then this map is an isomorphism since the action of $X_i$ on a polynomial preserves its degree. On the other hand, if $\mu'=-\mu$, then $X_i$ decreases the degree of a polynomial by one exactly, and therefore this map is surjective with kernel consisting of constant polynomials. This therefore proves our formula, at least without taking into account the $P(\A_f)$ action, and shows in fact that $H^0(\mf{a}_{P,0},\Sym(\mf{a}_{P,0})_\mu\otimes\C_{-\mu})$ can be identified with subspace of $\Sym(\mf{a}_{P,0})_\mu$ consisting of constants. By definition, this space has an action of $P(\A_f)$ given by the character $e^{\langle H_P(\cdot),\mu\rangle}$, which proves our lemma.
\end{proof}
Another ingredient we need is a well-known theorem of Kostant. To state it, we need to set some notation.\\
\indent Let $\mf{h}\subset\mf{g}$ be a Cartan subalgebra, and assume $\mf{h}\subset\mf{m}$. Fix an ordering on the roots of $\mf{h}$ in $\mf{g}$ which makes $\mf{p}$ standard. If $W(\mf{h},\mf{g})$ denotes the Weyl group of $\mf{h}$ in $\mf{g}$, then write
\[W^P=\sset{w\in W(\mf{h},\mf{g})}{w^{-1}\alpha>0\textrm{ for all positive roots }\alpha\textrm{ in }\mf{m}}.\]
Write $\rho$ for half the sum of the positive roots of $\mf{h}$ in $\mf{g}$.\\
\indent If $\Lambda\in\mf{h}^\vee$ is a dominant weight, write $E_\Lambda$ for the representation of $\mf{g}$ of highest weight $\Lambda$. If $\nu\in\mf{h}^\vee$ is a weight which is dominant for $\mf{m}$ we denote by $F_\nu$ the representation of $\mf{m}$ of highest weight $\nu$. In both cases, these weights may be nontrivial on the center, in which case these representations are considered to have central character given by the restriction of these weights to the respective centers. Then we have the following well-known theorem, whose proof we omit.
\begin{theorem}[Kostant]
\label{thmkostant}
With notation as above, let $\Lambda\in\mf{h}^\vee$ be a dominant weight. Then, as representations of $\mf{m}$, we have an isomorphism
\[H^i(\mf{n},E_\Lambda)\cong\bigoplus_{\substack{w\in W^P\\ \ell(w)=i}}F_{w(\Lambda+\rho)-\rho},\]
where $\ell(w)$ denotes the length of the Weyl group element $w$.
\end{theorem}
Now we are ready to state and prove the main theorem of this chapter. Its proof follows the strategy in Borel--Wallach \cite{BW}, Theorem III.3.3.
\begin{theorem}
\label{thmcohind}
Notation as above, let $\Lambda\in\mf{h}^\vee$ be a dominant weight such that $E=E_\Lambda$. Assume that the cohomology space
\[H^*(\mf{g}_0,K_\infty';\Ind_{P(\A)}^{G(\A)}(\tilde\pi\otimes\Sym(\mf{a}_{P,0})_{d\chi_\pi+\rho_P})\otimes E)\]
is nontrivial. Then there is a unique $w\in W^P$ such that
\[-w(\Lambda+\rho)|_{\mf{a}_{P,0}}=d\chi_\pi\]
and such that the infinitesimal character of the archimedean component of $\tilde\pi$ contains $-w(\Lambda+\rho)|_{\mf{h}\cap\mf{m_0}}$. Furthermore, if $\ell(w)$ is the length of such an element $w$, then for any $i$ we have
\begin{multline*}
H^i(\mf{g}_0,K_\infty';\Ind_{P(\A)}^{G(\A)}(\tilde\pi\otimes\Sym(\mf{a}_{P,0})_{d\chi_\pi+\rho_P})\otimes E)\\
\cong\iota_{P(\A_f)}^{G(\A_f)}(\pi_f)\otimes H^{i-\ell(w)}(\mf{m}_0,K_\infty'\cap P(\R);\tilde\pi_\infty\otimes F_{w(\Lambda+\rho)-\rho,0}),
\end{multline*}
where $\iota$ denotes a normalized parabolic induction functor, and $F_{w(\Lambda+\rho)-\rho,0}$ denotes the restriction to $\mf{m}_0$ of the representation of $\mf{m}$ of highest weight $w(\Lambda+\rho)-\rho$.
\end{theorem}
\begin{proof}
Let us first prove the uniqueness of the element $w$ in the theorem. Note first that
\[\mf{h}\cap\mf{g}_0=\mf{a}_{P,0}\oplus(\mf{h}\cap\mf{m}_0).\]
Because $\Lambda$ is dominant, we know $(\Lambda+\rho)$ is regular, and the conditions in the theorem therefore pin down the element $w(\Lambda+\rho)$ uniquely up to the Weyl group $W(\mf{h}\cap\mf{m}_0,\mf{m}_0)$ of $\mf{h}\cap\mf{m}_0$ in $\mf{m}_0$. But it is well known that $W^P$ is a set of representatives for $W(\mf{h},\mf{g})$ modulo $W(\mf{h}\cap\mf{m}_0,\mf{m}_0)$. Therefore $w(\Lambda+\rho)$ lies in a unique Weyl chamber, and so $w$ is determined.\\
\indent Let $i$ be an integer. We now begin to compute the cohomology space
\[H^i(\mf{g}_0,K_\infty';\Ind_{P(\A)}^{G(\A)}(\tilde\pi\otimes\Sym(\mf{a}_{P,0})_{d\chi_\pi+\rho_P})\otimes E).\]
First, Proposition \ref{propwpind} allows us to pull the tensor product with $E$ inside the induction, whence by Frobenius reciprocity, we have
\begin{multline}
\label{eqgkpkcoh}
H^i(\mf{g}_0,K_\infty';\Ind_{P(\A)}^{G(\A)}(\tilde\pi\otimes\Sym(\mf{a}_{P,0})_{d\chi_\pi+\rho_P})\otimes E)\\
\cong\Ind_{P(\A_f)}^{G(\A_f)}(H^i(\mf{p}_0,K_\infty'\cap P(\R);\tilde\pi\otimes\Sym(\mf{a}_{P,0})_{d\chi_\pi+\rho_P}\otimes E)).
\end{multline}
It is our goal, therefore, to compute
\[H^i(\mf{p}_0,K_\infty'\cap P(\R);\tilde\pi\otimes\Sym(\mf{a}_{P,0})_{d\chi_\pi+\rho_P}\otimes E).\]
\indent Now, as $(\mf{p}_0,K_\infty'\cap P(\R))$-modules, the space $\tilde\pi$ comes from a $(\mf{m}_0,K_\infty'\cap P(\R))$-module and $\Sym(\mf{a}_{P,0})_{d\chi_\pi+\rho_P}$ comes from an $\mf{a}_{P,0}$-module. Thus, using
\[\mf{p}_0=(\mf{m}_0\oplus\mf{a}_{P,0})\oplus\mf{n},\]
we get a spectral sequence whose $E_2$ page is
\[E_2^{j,k}=H^j(\mf{m}_0\oplus\mf{a}_{P,0},K_\infty'\cap P(\R);\tilde\pi\otimes\Sym(\mf{a}_{P,0})_{d\chi_\pi+\rho_P}\otimes H^k(\mf{n};E))\]
and which degenerates to the cohomology space above with $i=j+k$. We will eventually be able to say that this spectral sequence degenerates on its $E_2$ page, but this will follow from the vanishing of enough of its terms. So we compute this page now.\\
\indent By the Kostant decomposition (Theorem \ref{thmkostant}), the $(j,k)$-term on this $E_2$ page is
\[\bigoplus_{\substack{w'\in W^P\\ \ell(w)=k}}H^j(\mf{m}_0\oplus\mf{a}_{P,0},K_\infty'\cap P(\R);\tilde\pi\otimes\Sym(\mf{a}_{P,0})_{d\chi_\pi+\rho_P}\otimes F_{w'(\Lambda+\rho)-\rho}).\]
Write $\nu(w')=(w'(\Lambda+\rho)-\rho)|_{\mf{a}_{P.0}}$, As an $(\mf{m}_0\oplus\mf{a}_{P,0})$-module, the representation $F_{w'(\Lambda+\rho)-\rho}$ decomposes as
\[F_{w'(\Lambda+\rho)-\rho}=F_{w'(\Lambda+\rho)-\rho,0}\otimes\C_{\nu(w')},\]
as an exterior tensor product over the direct sum $\mf{m}_0\oplus\mf{a}_{P,0}$. Thus by the K\"unneth formula, we get
\begin{multline*}
H^*(\mf{m}_0\oplus\mf{a}_{P,0},K_\infty'\cap P(\R);\tilde\pi\otimes\Sym(\mf{a}_{P,0})_{d\chi_\pi+\rho_P}\otimes F_{w'(\Lambda+\rho)-\rho})\\
\cong H^*(\mf{m}_0,K_\infty'\cap P(\R);\tilde\pi\otimes F_{w'(\Lambda+\rho)-\rho,0})\otimes H^*(\mf{a}_{P,0},\Sym(\mf{a}_{P,0})_{d\chi_\pi+\rho_P}\otimes\C_{\nu(w')}).
\end{multline*}
By Lemma \ref{lemcohsym}, the second factor here is nonvanishing if and only if
\[d\chi_\pi+\rho_P=-\nu(w'),\]
and the first factor is nonvanishing only if the infinitesimal character of $F_{w'(\Lambda+\rho)-\rho,0}$ matches the negative of that of the archimedean component of $\tilde\pi$. Since $\mf{p}$ is standard, we have $\rho_P=\rho|_{\mf{a}_{P,0}}$, which implies
\[\nu(w')=w'(\Lambda+\rho)|_{\mf{a}_{P.0}}-\rho_P\]
and so this first nonvanishing condition is equivalent to
\[=w'(\Lambda+\rho)|_{\mf{a}_{P.0}}=d\chi_\pi;\]
the second of these nonvanishing conditions is just that $-w'(\Lambda+\rho)$ occurs in the infinitesimal character of the archimedean component of $\tilde\pi$. As shown at the beginning of this proof, there is only one $w'$ satisfying these two conditions, and we will denote it by $w$.\\
\indent Thus, by Lemma \ref{lemcohsym}, we get
\begin{multline*}
H^*(\mf{m}_0,K_\infty'\cap P(\R);\tilde\pi\otimes F_{w(\Lambda+\rho)-\rho,0})\otimes H^*(\mf{a}_{P,0},\Sym(\mf{a}_{P,0})_{d\chi_\pi+\rho_P}\otimes\C_{\nu(w)})\\
\cong H^*(\mf{m}_0,K_\infty'\cap P(\R);\tilde\pi\otimes F_{w(\Lambda+\rho)-\rho,0})\otimes\C(e^{\langle H_P(\cdot),d\chi_\pi+\rho_P\rangle}),
\end{multline*}
where the factor $\C(e^{\langle H_P(\cdot),d\chi_\pi+\rho_P\rangle})$ is concentrated in degree zero.\\
\indent Retracing our steps, we have thus computed the $E_2$ page of our spectral sequence. It is
\[E_2^{j,k}\cong
\begin{cases}
H^j(\mf{m}_0,K_\infty'\cap P(\R);\tilde\pi\otimes F_{w(\Lambda+\rho)-\rho,0})\otimes\C(e^{\langle H_P(\cdot),d\chi_\pi+\rho_P\rangle})&\textrm{if }k=\ell(w);\\
0&\textrm{if }k\ne\ell(w).\\
\end{cases}\]
The $E_2$ page therefore consists only of one row, and thus our spectral sequence degenerates. Hence we have shown
\begin{multline*}
H^i(\mf{p}_0,K_\infty'\cap P(\R);\tilde\pi\otimes\Sym(\mf{a}_{P,0})_{d\chi_\pi+\rho_P}\otimes E)\\
\cong H^{i-\ell(w)}(\mf{m}_0,K_\infty'\cap P(\R);\tilde\pi\otimes F_{w(\Lambda+\rho)-\rho,0})\otimes\C(e^{\langle H_P(\cdot),d\chi_\pi+\rho_P\rangle})
\end{multline*}
\indent Now we rewrite
\begin{align*}
H^{i-\ell(w)}(\mf{m}_0,K_\infty'&\cap P(\R);\tilde\pi\otimes F_{w(\Lambda+\rho)-\rho,0})\otimes\C(e^{\langle H_P(\cdot),d\chi_\pi+\rho_P\rangle})\\
&\cong\tilde\pi_f\otimes\C(e^{\langle H_P(\cdot),d\chi_\pi+\rho_P\rangle})\otimes H^{i-\ell(w)}(\mf{m}_0,K_\infty'\cap P(\R);\tilde\pi_\infty\otimes F_{w(\Lambda+\rho)-\rho,0})\\
&\cong\pi_f\otimes\C(e^{\langle H_P(\cdot),\rho_P\rangle})\otimes H^{i-\ell(w)}(\mf{m}_0,K_\infty'\cap P(\R);\tilde\pi_\infty\otimes F_{w(\Lambda+\rho)-\rho,0}),
\end{align*}
so that
\begin{multline*}
H^i(\mf{p}_0,K_\infty'\cap P(\R);\tilde\pi\otimes\Sym(\mf{a}_{P,0})_{d\chi_\pi+\rho_P}\otimes E)\\
\cong\pi_f\otimes\C(e^{\langle H_P(\cdot),\rho_P\rangle})\otimes H^{i-\ell(w)}(\mf{m}_0,K_\infty'\cap P(\R);\tilde\pi_\infty\otimes F_{w(\Lambda+\rho)-\rho,0}).
\end{multline*}
We therefore have, by \eqref{eqgkpkcoh},
\begin{multline*}
H^i(\mf{g}_0,K_\infty';\Ind_{P(\A)}^{G(\A)}(\tilde\pi\otimes\Sym(\mf{a}_{P,0})_{d\chi_\pi+\rho_P})\otimes E)\\
\cong\Ind_{P(\A_f)}^{G(\A_f)}(\pi_f\otimes\C(e^{\langle H_P(\cdot),\rho_P\rangle}))\otimes H^{i-\ell(w)}(\mf{m}_0,K_\infty'\cap P(\R);\tilde\pi_\infty\otimes F_{w(\Lambda+\rho)-\rho,0}),
\end{multline*}
which is what we wanted to prove.
\end{proof}
The above theorem will allow us to produce Eisenstein cohomology classes. To distinguish the representations of $G(\A_f)$ generated by these classes, we will need to see what might correspond to them on the Galois side. We set up the tools to do this in the next chapter.
\section{Galois representations}
\label{chgal}
We now recall the facts we need about $\ell$-adic Galois representations. The reason for introducing Galois representations into the picture is that they will allow us to distinguish the automorphic representations to which they will be attached.\\
\indent Our notion of what it means for a Galois representation to be attached to an automorphic representation is relatively weak, but it will suffice for our purposes.
\subsection{Galois representations attached to automorphic representations}
\label{secsatgal}
We continue to use the notation set previously, and in particular we will continue working with our reductive $\Q$-group $G$, but with one modification: We now assume that $G$ is split. This will simplify our discussion of Satake parameters, and it will also allow us to work only with the Galois group of $\Q$ instead of that of some finite extension.\\
\indent We explain in this section what we mean when we say that an automorphic representation of $G(\A)$ has attached to it a Galois representation. Our version of this notion will be a weak one, in the sense that it will only depend on the automorphic representation in question at all but finitely many of its unramified places. But this will suffice for our purposes.\\
\indent So to get started, fix a prime $p$. We will recall some of the theory of unramified representations of $G(\Q_p)$ due to Langlands, Satake, and others.\\
\indent First we fix a split maximal torus $T\subset G$ and a Borel subgroup $B\subset G$ containing $T$. Write $U$ for the unipotent radical of $B$. Let
\[W=N_G(T)/T\]
be the Weyl group of $G$. Let $\delta_{B(\Q_p)}$ denote the modulus character of $B(\Q_p)$.\\
\indent Next, fix a model of $G$ over $\Z_p$. Write $K_p=G(\Z_p)$; this is a hyperspecial maximal compact subgroup of $G(\Q_p)$. Let $\mc{H}(K_p)$ be the spherical Hecke algebra, defined as the convolution algebra of smooth, compactly supported, $K_p$-biinvariant, $\C$-valued functions on $G(\Q_p)$.\\
\indent Fix an irreducible admissible representation $\sigma$ of $G(\Q_p)$ which is spherical, i.e., which has a $K_p$-fixed vector. Then the $K_p$-invariant subspace $\sigma^{K_p}$ is one dimensional. Thus we get a character of the Hecke algebra
\[\omega_\sigma^{\mr{H}}:\mc{H}(K_p)\to\End(\sigma^{K_p})\cong\C.\]
\indent On the other hand, we have the Satake transform $\mc{S}$, which is an isomorphism from $\mc{H}(K_p)$ to the Weyl group invariants of the analogously defined Hecke algebra $\mc{H}(T(\Z_p))$. In more detail, the Hecke algebra $\mc{H}(T(\Z_p))$ is defined to be the convolution algebra of smooth, compactly supported, $T(\Z_p)$-biinvariant, $\C$-valued functions on $T(\Q_p)$. Because $T$ is abelian, this is the same as the group algebra $\C[T(\Q_p)/T(\Z_p)]$. Of course, $W$ acts on $T$ and therefore gives compatible actions on both $\mc{H}(T(\Z_p))$ and $\C[T(\Q_p)/T(\Z_p)]$.\\
\indent The Satake transform
\[\mc{S}:\mc{H}(K_p)\to\mc{H}(T(\Z_p))\]
is defined by
\[\mc{S}(f)(t)=\delta_{B(\Q_p)}(t)^{1/2}\int_{U(\Q_p)}f(tu)\,du.\]
It is a theorem that the image of $\mc{S}$ is contained in the Weyl group invariants $\mc{H}(T(\Z_p))^W$ and, in fact, is an isomorphism when $\mc{H}(T(\Z_p))^W$ is considered at its target. Thus, through the identifications above, we get an isomorphism
\[\mc{H}(K_p)\cong\C[T(\Q_p)/T(\Z_p)]^W.\]
We can therefore transfer the character $\omega_\sigma^{\mr{H}}$ defined above to $\C[T(\Q_p)/T(\Z_p)]^W$ and obtain a character
\[\omega_\sigma^{\mr{S}}:\C[T(\Q_p)/T(\Z_p)]^W\to\C.\]
\indent There is another construction that gives a character of $\C[T(\Q_p)/T(\Z_p)]^W$ starting from the representation $\sigma$, which we describe now. It is a theorem that $\sigma$, since it is spherical, occurs as a subquotient of a principal series representation
\[\Ind_{B(\Q_p)}^{G(\Q_p)}(\chi\cdot\delta_{B(\Q_p)})\]
for some character $\chi$ of $T(\Q_p)$ which is trivial on $T(\Z_p)$. The character $\chi$ with this property is unique only up to the action of $W$. But in any case, the character $\chi$, when viewed as a character $T(\Q_p)/T(\Z_p)$, gives naturally a character
\[\tilde\omega:\C[T(\Q_p)/T(\Z_p)]\to\C.\]
The restriction of this character to the Weyl invariants will be written as
\[\omega_\sigma^{\mr{I}}:\C[T(\Q_p)/T(\Z_p)]^W\to\C.\]
While there is a choice involved in selecting the character $\chi$, and hence in defining $\tilde\omega$, the character $\omega_\sigma^{\mr{I}}$ does not depend on this choice and is well defined.\\
\indent We state the following well known result as a proposition.
\begin{proposition}
\label{propsatind}
In the setting above, the two characters
\[\omega_\sigma^{\mr{S}},\omega_\sigma^{\mr{I}}:\C[T(\Q_p)/T(\Z_p)]^W\to\C\]
coincide.
\end{proposition}
Let us denote by $\omega_\sigma$ the common character $\omega_\sigma^{\mr{S}}=\omega_\sigma^{\mr{I}}$.\\
\indent Now the group $T(\Q_p)/T(\Z_p)$ can be naturally identified with the cocharacter group $X_*(T)$; the identification is given by evaluating a cocharacter $\lambda\in X_*(T)$ at a uniformizer in $\Q_p^\times$. Also, if we fix a maximal torus $T^\vee$ in the dual group $G^\vee$, we have a natural identification $X_*(T)=X^*(T^\vee)$ of the cocharacter group of $T$ with the character group of $T^\vee$.\\
\indent Therefore the character $\omega_\sigma$ just constructed may well be viewed as a character
\[\omega_\sigma:\C[X^*(T^\vee)]^W\to\C.\]
Now given a finite dimensional representation $V$ of $G^\vee(\C)$, we can view its character $\chi_V$ as an element of $\C[X^*(T^\vee)]^W$. Then the character $\omega_\sigma$ gives a conjugacy class $s(\sigma)$ in $G^\vee(\C)$; it is the unique conjugacy class with the property that
\[\omega_\sigma(\chi_V)=\tr(s(\sigma)|V)\]
for any finite dimensional representation $V$ of $G^\vee(\C)$.
We call $s(\sigma)$ the \it Satake parameter \rm or \it Langlands parameter \rm attached to $\sigma$.\\
\indent We now fix a prime $\ell$ different from $p$. Since $\overline{\Q}_\ell$ is isomorphic to $\C$, everything above could be done over $\overline{\Q}_\ell$ instead. In particular, we may view $\omega_\sigma$ as a character of $\overline{\Q}_\ell[T(\Q_p)/T(\Z_p)]^W\cong\overline{\Q}_\ell[X^*(T^\vee)]^W$, and we may view the Satake parameter $s(\sigma)$ as a conjugacy class in $G^\vee(\overline{\Q}_\ell)$.\\
\indent We need to make this change of field because our Galois representations will have as their target the group $G^\vee(\overline{\Q}_\ell)$. In fact, we are ready to give the following definition.
\begin{definition}
\label{defgalrep}
Let $\Pi$ be an automorphic representation of $G(\A)$. We will say that a continuous representation
\[\rho:G_\Q\to G^\vee(\overline{\Q}_\ell)\]
is \it attached to $\Pi$ \rm if there is a finite set $S$ of places of $\Q$ containing $\ell$, the archimedean place, and all the ramified places for $\Pi$, such that for any prime $p\notin S$, $\rho$ is unramified at $p$ and we have
\[\rho(\Frob_p)^{\sss}\in s(\Pi_p),\]
where $\Frob_p$ is any choice of (geometric) Frobenius element at $p$, the element $\rho(\Frob_p)^{\sss}$ is the semisimplification of $\rho(\Frob_p)$, and the Satake parameter $s(\Pi_p)$ of the local component of $\Pi$ at $p$ is viewed as a conjugacy class in $G^\vee(\overline{\Q}_\ell)$.
\end{definition}
We remark that in the definition, the semisimplification $\rho(\Frob_p)^{\sss}$ may be defined to be the semisimple element of $G^\vee(\overline\Q_\ell)$ whose image in any finite dimensional representation of $G^\vee(\overline\Q_\ell)$ has the same characteristic polynomial as $\rho(\Frob_p)$.\\
\indent Now in the case of the group $\GL_2$ a lot is known about when such Galois representations exist. Let us recall some results in this direction.\\
\indent Let $F$ be a holomorphic cuspidal eigenform of weight $k\geq 1$, conductor $N\geq 1$, and nebentypus $\omega_F$. Then $F$ gives rise to a unitary automorphic representation $\tilde\pi$ of $\GL_2(\A)$. This representation $\tilde\pi$ has central character given by the adelization of $\omega_F$. Write
\[\pi=\tilde\pi\otimes\vert\det\vert^{(k-1)/2}\]
This normalization is necessary to recover the usual Galois representation attached to $F$. In fact, we have the following theorem.
\begin{theorem}
\label{thmesdds}
With the setting as in the above paragraph, fix a prime $\ell$ not dividing $N$. Then there is a continuous Galois representation
\[\rho_\pi:G_\Q\to\GL_2(\overline{\Q}_\ell)\]
which is attached to $\pi$ in the sense of Definition \ref{defgalrep}; in fact the set $S$ in that definition can be taken to be the set of primes dividing $N$, $\ell$, and $\infty$. This representation $\rho_\pi$ is unique up to conjugation by elements of $\GL_2(\overline{\Q}_\ell)$, and it is irreducible. Furthermore, $\rho_\pi$ is Hodge--Tate (in fact, de Rham) at $\ell$ with Hodge--Tate weights $0$ and $k-1$.
\end{theorem}
\begin{remark}
The above is a classical theorem which (except for the final claim about $\rho_\pi$ being Hodge--Tate) is due to Eichler--Shimura when $k=2$, to Deligne when $k>2$, and to Deligne--Serre when $k=1$. Usually when recalling this theorem one states explicitly the properties that, for $p\notin S$ we have
\[\tr(\rho_\pi(\Frob_p))=a_p,\]
where $a_p$ is the $p$th Hecke eigenvalue of $F$, and
\[\det(\rho_\pi)=\omega_F\chi_{\cyc}^{k-1},\]
where $\chi_{\cyc}:G_\Q\to\Z_\ell^\times$ denotes the $\ell$-adic cyclotomic character and $\omega_F$ is viewed as a finite order Galois character by class field theory. Actually, these assertions follow from our statement of the theorem once we know $\pi_p$ explicitly enough to know the characteristic polynomial of $s(\pi_p)$ for $p\notin S$.
\end{remark}
We conclude this section with a proposition which will be useful for us later when distinguishing between different automorphic representations. To state it, we recall the following definition.
\begin{definition}
Let $\Pi,\Pi'$ be two automorphic representations of a reductive group $G$, with respective local components $\Pi_v,\Pi_v'$ at places $v$. We say $\Pi$ and $\Pi'$ are \it nearly equivalent \rm if, for all but finitely many places $v$, there is an isomorphism $\Pi_v\cong\Pi_v'$.
\end{definition}
\begin{proposition}
\label{propdistgalrep}
Let $\Pi,\Pi'$ be two automorphic representations of $G(\A)$ with respective Galois representations
\[\rho,\rho':G_\Q\to G^\vee(\overline{\Q}_\ell).\]
Assume $\Pi$ and $\Pi'$ are nearly equivalent. Let
\[R:G^\vee\to\GL_n\]
be a finite dimensional representation of $G^\vee$. Then the semisimplified Galois representations 
\[(R\circ\rho)^{\sss},\,\, (R\circ\rho')^{\sss},\]
which are semisimple representations of $G_\Q$ into $\GL_n(\overline{\Q}_\ell)$, are equivalent.
\end{proposition}
\begin{proof}
By the hypotheses, there is a finite set $S$ of places, including $\ell$ and the archimedean place, such that for $p\notin S$, the local components $\Pi_p$ and $\Pi_p'$ of our automorphic representations at $p$ are unramified and isomorphic. Therefore we have an equality of Satake parameters for $p\notin S$,
\[s(\Pi_p)=s(\Pi_p').\]
After possibly enlarging $S$, we have then that for $p\notin S$, the semisimple elements
\[\rho(\Frob_p)^{\sss},\,\,\rho'(\Frob_p)^{\sss}\]
are conjugate in $G^\vee(\overline{\Q}_\ell)$. Therefore we have an equality of traces
\[\tr(R(\rho(\Frob_p)))=\tr(R(\rho'(\Frob_p)))\]
By continuity and Chebotarev, this implies an equality of characters
\[\tr(R\circ\rho)=\tr(R\circ\rho'),\]
which in turn implies the conclusion of our proposition.
\end{proof}
\begin{remark}
The above proposition may be summarized as saying that $(R\circ\rho)^{\sss}$ is a near-equivalence invariant of automorphic representations (at least when $\rho$ exists). It is therefore also an isomorphism invariant; that is, the proposition can be applied when $\Pi\cong\Pi'$. This is useful, since it is possible for an automorphic representation to have many Galois representations attached to it in the sense of our definition. This is especially possible when $\rho$ is reducible (i.e., factors through a proper parabolic subgroup of $G^\vee(\overline{\Q}_\ell)$).
\end{remark}
\subsection{Galois representations and induced representations}
In this section we explain how to attach Galois representations to subquotients of parabolically induced representations. This will therefore give us a way of attaching Galois representations to Eisenstein series.\\
\indent We continue with the notation of the previous section, and in particular we will work with our split reductive $\Q$-group $G$ and a choice of split maximal torus $T\subset G$ and Borel subgroup $B\subset G$ containing $T$. As we did before, we choose a split maximal torus $T^\vee$ in the dual group $G^\vee$ and a Borel $B^\vee$ containing $T^\vee$.\\
\indent Now let $P\subset G$ be a parabolic subgroup containing $B$, and let $M$ be its standard Levi. The parabolic $P$ corresponds to a subset of the set of simple roots of $T$ in $G$, and the set of corresponding coroots gives us a standard parabolic $P^\vee$ in $G^\vee$. Its standard Levi $M^\vee$ is, as this notation suggests, identified with the dual group of $M$.
\begin{proposition}
\label{propindgalrep}
Let $\pi$ be an automorphic representation of $M(\A)$. Assume that $\pi$ has attached to it a Galois representation
\[\rho_\pi:G_\Q\to M^\vee(\overline{\Q}_\ell),\]
in the sense of Definition \ref{defgalrep}. Let $\Pi$ be an automorphic representation of $G(\A)$ which is a subquotient of the induced representation
\[\Ind_{P(\A)}^{G(\A)}(\pi\otimes\delta_{P(\A)}^{1/2}),\]
where $\delta_{P(\A)}$ is the modulus character of $P(\A)$. Let $i_M$ be the inclusion map
\[i_M:M^\vee(\overline{\Q}_\ell)\hookrightarrow G^\vee(\overline{\Q}_\ell).\]
Then the Galois representation
\[\rho_\Pi:G_\Q\to G^\vee(\overline{\Q}_\ell)\]
given by
\[\rho_\Pi=i_M\circ\rho_\pi\]
is attached to $\Pi$, again in the sense of Definition \ref{defgalrep}.
\end{proposition}
\begin{proof}
Decompose $\pi$ and $\Pi$ into their local components,
\[\pi\cong\sideset{}{'}\bigotimes_v\pi_v,\qquad\Pi\cong\sideset{}{'}\bigotimes_v\Pi_v.\]
Let $S_0$ be a finite set of places of $\Q$, which contains $\ell$ and the archimedean place, and such that for $p\notin S_0$, the condition
\[\rho(\Frob_p)^{\sss}\in s(\pi_p)\]
of Definition \ref{defgalrep} is satisfied for $\pi_p$. Let $S$ be the set of primes containing all those in $S$, as well as any place $v$ for which $\Pi_v$ is not spherical. We are to verify that
\begin{equation}
\label{eqpfsequalss}
i_M(s(\pi_p))\subset s(\Pi_p).
\end{equation}
for $p\notin S$.\\
\indent Let $W_G$ be the Weyl group of $T$ in $G$, and $W_M$ that of $T$ in $M$, and let
\[\omega_{\pi_p}:\overline{\Q}_\ell[X^*(T^\vee)]^{W_M}\to\overline{\Q}_\ell,\qquad\omega_{\Pi_p}:\overline{\Q}_\ell[X^*(T^\vee)]^{W_G}\to\overline{\Q}_\ell\]
be the characters constructed in Proposition \ref{propsatind}. Let $V$ be any finite dimensional representation of $G$, and let $V|_M$ be the same representation but viewed as a representation of $M$. By the characterizing property of the Satake parameter, checking \eqref{eqpfsequalss} is the same as checking that
\[\omega_{\pi_p}(\chi_V)=\omega_{\Pi_p}(\chi_V)\]
where $\chi_V$ is the character of $V$. This will of course follow if we show $\omega_{\Pi_p}$ is the restriction of $\omega_{\pi_p}$ to the $W_G$-invariants $\overline{\Q}_\ell[T(\Q_p)/T(\Z_p)]^{W_G}$.\\
\indent Recall the construction of $\omega_{\pi_p}$ via normalized induction; the representation $\pi_p$ occurs as the irreducible spherical subquotient of a Borel induction
\[\Ind_{(B\cap M)(\Q_p)}^{M(\Q_p)}(\chi\cdot\delta_{(B\cap M)(\Q_p)}).\]
But by induction in stages, $\Pi_p$ is the irreducible spherical subquotient of
\[\Ind_{B(\Q_p)}^{G(\Q_p)}(\chi\cdot\delta_{B(\Q_p)}).\]
This shows then that $\omega_{\pi_p}$ is the restriction of the character $\overline{\Q}_\ell[T(\Q_p)/T(\Z_p)]\to\overline{\Q}_\ell$ induced from $\chi$ to the $W_M$-invariants, and similarly $\omega_{\Pi_p}$ is the restriction of the same character to $\overline{\Q}_\ell[T(\Q_p)/T(\Z_p)]^{W_G}$. Once we pass through the identification $T(\Q_p)/T(\Z_p)=X^*(T^\vee)$, this is exactly what we wanted to show.
\end{proof}
\section{The case of $\GSp_4$}
\label{chgsp4}
We now apply the theory of the previous three chapters to the case when $G=\GSp_4$. We will define certain Langlands quotients of parabolically induced representations, induced from the Siegel parabolic, and study their multiplicities in Eisenstein and cuspidal cohomology.
\subsection{The group $\GSp_4$}
\label{secgroupgsp4}
We fix in this section some notation that will be used throughout this chapter.\\
\indent Let $J$ be the matrix
\[J=\pmat{&&1&\\ &&&1\\ -1&&&\\ &-1&&}.\]
Define $\GSp_4$ to be the group over $\Q$ defined matricially for $\Q$-algebras $A$ by
\[\GSp_4(A)=\sset{g\in\GL_4(A)}{\prescript{t}{}{g}Jg=\nu J\textrm{ for some }\nu=\nu(g)\in A^\times}.\]
The group $\GSp_4$ is reductive and split. In fact, a split maximal torus $T$ is given by the subgroup of all diagonal matrices in $\GSp_4$.\\
\indent The assignment $g\mapsto\nu(g)$, where $\nu(g)$ is as in the definition above, defines a character of $\GSp_4$, called the \it similitude character, \rm and which we denote simply by $\nu$. We also denote by the same letter the restriction of $\nu$ to the maximal torus $T$.\\
\indent The group $\GSp_4$ contains the subgroup $\Sp_4$, defined as
\[\Sp_4=\sset{g\in\GSp_4}{\nu(g)=1}.\]
The group $\Sp_4$ is the split simple group of type $C_2$, with a choice of split maximal torus $T_0=T\cap\Sp_4$, given again by diagonal matrices. Let us now study this group from the perspective of its root lattice.
\subsubsection*{The root lattice}
The Dynkin diagram of $\Sp_4$ is as in Figure \ref{figgsp4dynkin}.
\begin{figure}[h]
\centering
\includegraphics[scale=.2]{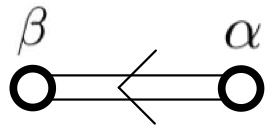}
\caption{The Dynkin diagram of $\GSp_4$}
\label{figgsp4dynkin}
\end{figure}
So we are writing $\alpha$ for the long simple root and $\beta$ for the short simple root. This way of labelling the roots will be consistent with our notation for the simple roots of $\G_2$ later.\\
\indent Explicitly, any element of $T_0$ is a diagonal matrix of the form
\[\diag(a,b,a^{-1},b^{-1}),\]
and the characters $\alpha$ and $\beta$ act on these matrices by
\[\alpha(\diag(a,b,a^{-1},b^{-1}))=b^2,\qquad \beta(\diag(a,b,a^{-1},b^{-1}))=ab^{-1}.\]
The character $\alpha$ has an obvious square root, which we write additively as $\alpha/2$, which picks out the $b$ entry of a diagonal matrix as above. Then $\alpha/2$ and $\beta$ generate the character group $X^*(T_0)$.\\
\indent The inner product space $X^*(T_0)\otimes\R$ is isometric to $\R^2$ with its usual inner product, and an isometry is given by $\alpha\mapsto(0,2)$ and $\beta\mapsto(1,-1)$. Thus we get a picture of the root lattice as in Figure \ref{figgsp4chamber}; there, the dominant chamber is shaded.\\
\begin{figure}[h]
\centering
\includegraphics[scale=.25]{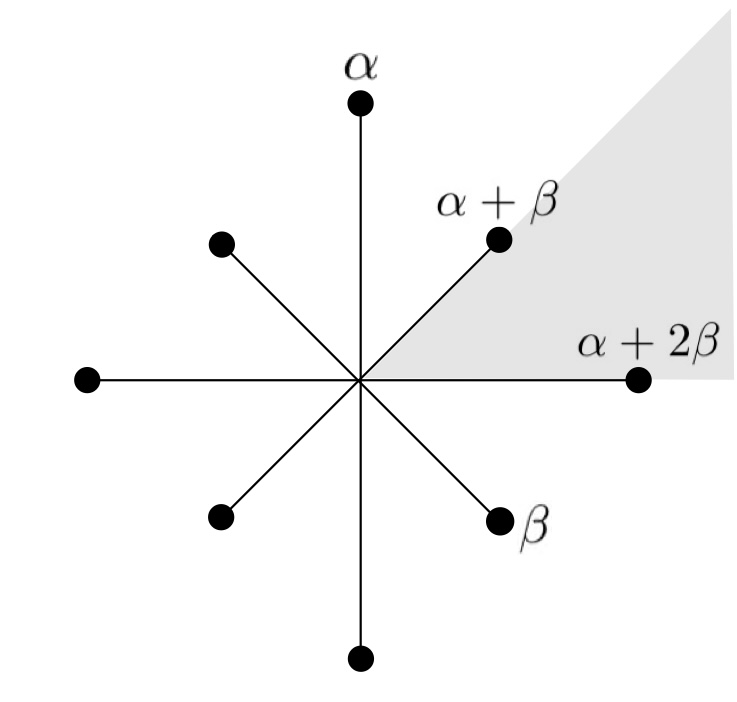}
\caption{The root lattice of $\GSp_4$}
\label{figgsp4chamber}
\end{figure}
\indent We can extend the characters $\alpha$ and $\beta$ to characters of the torus $T\subset\GSp_4$ as follows. Every element of $T$ can be written as a diagonal matrix of the form
\[\diag(a,b,ca^{-1},cb^{-1}),\]
and for such matrices, we let
\[\alpha(\diag(a,b,ca^{-1},cb^{-1}))=b^2c^{-1},\qquad \beta(\diag(a,b,ca^{-1},cb^{-1}))=ab^{-1}.\]
By its definition, the character $\nu$ acts on these matrices as
\[\nu(\diag(a,b,ca^{-1},cb^{-1}))=c.\]
The characters $\alpha,\beta,\nu$ generate an index $2$ subgroup in the character group $X^*(T)$, and the character $\alpha+\nu$ (we write the group law in $X^*(T)$ additively) has a square root.\\
\indent The center of $\GSp_4$ is equal to the center of $\GL_4$; it is just the subgroup of invertible multiples of the identity matrix $I$. The center of $\Sp_4$ has order $2$ and is equal to $\{\pm I\}$.\\
\indent Let us write $\Delta$ for the set of roots of $T$ in $\GSp_4$ obtained above, or for the set of roots in $T_0$ in $\Sp_4$. Write $\Delta^+$ for the positive roots. So
\[\Delta^+=\{\alpha,\beta,\alpha+\beta,\alpha+2\beta\}.\]
\subsubsection*{Parabolic subgroups}
For $\gamma\in\Delta$, write $\mathbf{x}_\gamma$ for the unipotent root group homomorphism
\[\mathbf{x}_\gamma:\mb{G}_a\to\GSp_4.\]
Here $\mb{G}_a$ denotes as usual the additive group scheme. Then we have the following matrix formulas for each $\mathbf{x}_\gamma$,
\begin{equation}
\label{eqrootgroups}
\begin{gathered}
\mathbf{x}_\alpha(a)=\pmat{1&&&\\ &1&&a\\ &&1&\\ &&&1},\qquad
\mathbf{x}_\beta(a)=\pmat{1&a&&\\ &1&&\\ &&1&\\ &&-a&1},\\
\mathbf{x}_{\alpha+\beta}(a)=\pmat{1&&&a\\ &1&a&\\ &&1&\\ &&&1},\qquad
\mathbf{x}_{\alpha+2\beta}(a)=\pmat{1&&a&\\ &1&&\\ &&1&\\ &&&1},
\end{gathered}
\end{equation}
and then
\[\mathbf{x}_{-\gamma}(a)=\prescript{t}{}{\mathbf{x}}_\gamma(a)\]
for $\gamma\in\Delta$.\\
\indent Let $P_\alpha\subset\GSp_4$ be the standard parabolic subgroup whose Levi contains the image of $\mathbf{x}_\alpha$. Write $P_\alpha=M_\alpha N_\alpha$ for its Levi decomposition. We similarly define $P_\beta$ and write $P_\beta=M_\beta N_\beta$ for its Levi decomposition. We write $B$ for the standard Borel and $B=TU$ for its Levi decomposition. Then by \eqref{eqrootgroups} it follows that $B$, $P_\alpha$ and $P_\beta$ take the following forms:
\[B=\Set{\pmat{*&*&*&*\\ &*&*&*\\ &&*&\\ &&*&*}},\qquad P_\alpha=\Set{\pmat{*&*&*&*\\ &*&*&*\\ &&*&\\ &*&*&*}},\qquad P_\beta=\Set{\pmat{*&*&*&*\\ *&*&*&*\\ &&*&*\\ &&*&*}},\]
Along with $\GSp_4$ itself, these comprise all the standard parabolic subgroups of $\GSp_4$. The parabolics $P_\alpha$ and $P_\beta$ are both maximal and have Levis isomorphic to $\GL_2\times\GL_1$. Explicit isomorphisms 
\[i_\alpha:\GL_2\times\GL_1\to M_\alpha,\quad\textrm{and}\quad i_\beta:\GL_2\times\GL_1\to M_\beta\]
are given by, for $A=\sm{a&b\\ c&d}\in\GL_2$ and $t\in\GL_1$,
\begin{equation}
\label{eqgl2gl1levis}
i_\alpha(A,t)=\pmat{t^{-1}\det(A)&&&\\ &a&&b\\ &&t&\\ &c&&d}\in M_\alpha,\qquad i_\beta(A,t)=\pmat{A&*\\ &t\prescript{t}{}{A}^{-1}}\in M_\beta.
\end{equation}
As is often done, we call $P_\alpha$ the \it Klingen parabolic \rm and $P_\beta$ the \it Siegel parabolic.\rm
\subsubsection*{Duality}
The group $\GSp_4$ is self dual. Identifying $\GSp_4$ with its dual group switches the long and short simple roots. For us this will mean that certain data associated with the Siegel parabolic will become associated with the Klingen parabolic on the dual side, and vice-versa.\\
\indent This can be made explicit as follows. There are isomorphisms $\GL_1\cong\GL_1^\vee$, $\GL_2\cong\GL_2^\vee$, and $\GSp_4\cong\GSp_4^\vee$ such that the diagrams below commute. Identify $M_\alpha$ and $M_\beta$ with $\GL_2\times\GL_1$ via the maps $i_\alpha$ and $i_\beta$ of \eqref{eqgl2gl1levis}. Then $M_\alpha^\vee$ and $M_\beta^\vee$ are identified with $\GL_2^\vee\times\GL_1^\vee$ as well, and these latter identifications fit into a commutative diagram as follows. We have
\begin{equation}
\label{eqdualizealpha1}
\xymatrix{
\GL_2^\vee\times\GL_1^\vee\ar[rr]^-\sim \ar[d]^\sim && M_\alpha^\vee \ar@{^{(}->}[r] \ar[d]^\sim & \GSp_4^\vee\ar[d]^\sim \\
\GL_2\times\GL_1 \ar[r]^-{\varphi_\alpha} &\GL_2\times\GL_1\ar[r]^-{i_\beta}& M_\beta \ar@{^{(}->}[r] & \GSp_4}
\end{equation}
where the map $\varphi_\alpha$ is the map given by
\begin{equation}
\label{eqdualizealpha2}
\varphi_\alpha(A,t)=(A,\det(A)t).
\end{equation}
Similarly, we have a commutative diagram
\begin{equation}
\label{eqdualizebeta1}
\xymatrix{
\GL_2^\vee\times\GL_1^\vee\ar[rr]^-\sim \ar[d]^\sim && M_\beta^\vee \ar@{^{(}->}[r] \ar[d]^\sim & \GSp_4^\vee\ar[d]^\sim \\
\GL_2\times\GL_1 \ar[r]^-{\varphi_\beta} & \GL_2\times\GL_1 \ar[r]^-{i_\alpha} & M_\alpha \ar@{^{(}->}[r] & \GSp_4}
\end{equation}
where the map $\varphi_\beta$ given by
\begin{equation}
\label{eqdualizebeta2}
\varphi_\beta(A,t)=(tA,t).
\end{equation}
Finally, for the Borel, the map $i_0:\GL_1^3\to T$ given by
\begin{equation}
\label{eqtorusabc}
i_0(a,b,c)=\diag(a,b,ca^{-1},cb^{-1})
\end{equation}
is an isomorphism which identifies $T$ with $\GL_1^3$, and hence also $T^\vee$ and $(\GL_1^\vee)^3$. This latter identification fits into the commutative diagram
\begin{equation}
\label{eqdualizeborel1}
\xymatrix{
(\GL_1^\vee)^3\ar[rr]^-\sim \ar[d]^\sim && T^\vee \ar@{^{(}->}[r] \ar[d]^\sim & \GSp_4^\vee\ar[d]^\sim \\
(\GL_1)^3 \ar[r]^-{\varphi_0} &(\GL_1)^3\ar[r]^-{i_0}& T \ar@{^{(}->}[r] & \GSp_4}
\end{equation}
where the map $\varphi_0$ is given by
\begin{equation}
\label{eqdualizeborel2}
\varphi_0(t_1,t_2,t_3)=(t_1t_2t_3,t_1t_3,t_1t_2t_3^2).
\end{equation}
\subsubsection*{The Weyl group}
Let $W=W(T,\GSp_4)$ be the Weyl group of $\GSp_4$. The group $W$ is isomorphic to the dihedral group $D_4$ with eight elements acting naturally on the root lattice.\\
\indent For $\gamma\in\Delta$, let $w_\gamma$ be the reflection about the line perpendicular to $\gamma$. Then $W$ is generated by the simple reflections $w_\alpha$ and $w_\beta$. Let us amalgamate products of these reflections into a single notation: Write $w_{\alpha\beta}=w_\alpha w_\beta$, $w_{\alpha\beta\alpha}=w_\alpha w_\beta w_\alpha$, and so on. Then
\[W=\{1,w_\alpha,w_\beta,w_{\alpha\beta},w_{\beta\alpha}, w_{\alpha\beta\alpha},w_{\beta\alpha\beta},w_{-1}\}.\]
The elements above are written minimally in terms of products of the simple reflections $w_\alpha$ and $w_\beta$, except for the final element $w_{-1}$. This element $w_{-1}$ is the element that acts by negation on the root lattice, and it of length $4$, equal to both $w_{\alpha\beta\alpha\beta}$ and $w_{\beta\alpha\beta\alpha}$.\\
\indent For $P=MN$ one of the standard parabolic subgroups of $\GSp_4$, let
\[W^P=\sset{w\in W}{w^{-1}\gamma>0\textrm{ for all positive roots }\gamma\textrm{ in }M}.\]
This is the set of minimal length representatives for the quotient $W(T,M)\backslash W$. Then
\[W^{P_\alpha}=\{1,w_\beta,w_{\beta\alpha},w_{\beta\alpha\beta}\},\qquad W^{P_\beta}=\{1,w_\alpha,w_{\alpha\beta},w_{\alpha\beta\alpha}\},\]
and $W^B=W$.\\
\indent Finally, we note for future reference that the action of $W$ on $T$ is given by
\begin{equation}
\label{eqweylactiongsp4}
\begin{gathered}
\diag(a,b,ca^{-1},cb^{-1})^{w_\alpha}=\diag(a,cb^{-1},ca^{-1},b),\\
\diag(a,b,ca^{-1},cb^{-1})^{w_\beta}=\diag(b,a,cb^{-1},ca^{-1}).
\end{gathered}
\end{equation}
\subsubsection*{The group $\GSp_4(\R)$}
The real Lie group $\GSp_4(\R)$ has discrete series representations, but is disconnected. However, $\Sp_4(\R)$ is connected as a real Lie group. Therefore it will be easier to describe the classification of the discrete series representations of $\Sp_4(\R)$ first and then use it to classify those of $\GSp_4(\R)$. For a review of Harish-Chandra's classification of discrete series, the reader may jump ahead to Section \ref{sectds}, specifically Theorem \ref{thmHC}.\\
\indent Fix first a maximal compact subgroup $K_\infty$ in $\GSp_4(\R)$. Then the connected component $K_\infty^\circ$ of the identity is a maximal compact subgroup of $\Sp_4(\R)$.\\
\indent The group $K_\infty^\circ$ is isomorphic to the real unitary group $\mr{U}(2)$. Therefore any maximal torus in $K_\infty^\circ$ is two dimensional. Fix $T_c\subset K_\infty^\circ$ such a maximal torus. Then $T_c$ is also a maximal torus in $\Sp_4(\R)$.\\
\indent Let $\mf{t}_c$ be the complexified Lie algebra of $T_c$ and $\mf{k}$ that of $K_\infty^\circ$. Abusing notation, we let $\Delta=\Delta(\mf{t}_c,\sp_4)$ be the roots of $\mf{t}_c$ in $\sp_4$, and let $\Delta_c=\Delta(\mf{t}_c,\mf{k})\subset\Delta$ be the set of compact roots. There are two roots in $\Delta_c$ and they are short. Pick one, and again by abuse of notation, call it $\beta$. Choose a long root $\alpha$ in $\Delta$ such that $\alpha$ and $\beta$ are a pair of simple roots. The roots $\beta$ and $\alpha/2$ generate the lattice of analytically integral weights in $\mf{t}_c^\vee$.\\
\indent The compact Weyl group $W_c=W(\mf{t}_c,\mf{k})$ has two elements and is generated by the simple reflection $w_\beta$ across the line perpendicular to $\beta$. If we write $W=W(\mf{t}_c,\sp_4)$ for the Weyl group of $\Delta$, then $W_c$ has index $4$ in $W$. Therefore the discrete series representations of $\Sp_4(\R)$ are parametrized by analytically integral weights that lie far enough inside the four chambers below the line perpendicular to $\beta$.\\
\indent The element $w_{-1}$ is in the Weyl group $W$, and the element $w_\beta\circ w_{-1}$ is equal to the simple reflection $w_{\alpha+\beta}$ across the line perpendicular to $\alpha+\beta$. If a discrete series representation $V$ has Harish-Chandra parameter $\lambda$, then the contragredient $V^\vee$ has Harish-Chandra parameter $-\lambda$; but if the weight $\lambda$ is in one of the four chambers under the line perpendicular to $\beta$, then $-\lambda$ will lie above this line. Therefore we should choose $w_\beta(-\lambda)=w_{\alpha+\beta}\lambda$ as the parameter for $V^\vee$.\\
\indent Now there is an element $k_0$ of order $2$ in the nonidentity component of $K_\infty$ such that the adjoint action of $k_0$ on $K_\infty^\circ$ preserves $T_c$ and acts as inversion there. Write $\GSp_4(\R)^+$ for the subgroup of $\GSp_4(\R)$ given by
\[\GSp_4(\R)^+=\sset{g\in\GSp_4(\R)}{\nu(g)>0}.\]
Then
\[\GSp_4(\R)^+\cong\Sp_4(\R)\times\R_{>0},\]
and each discrete series representation $V$ of $\Sp_4(\R)$ can be extended to a representation $V_+$ of $\GSp_4(\R)^+$ by letting the $\R_{>0}$ component act trivially. Then we can induce to $\GSp_4(\R)$ to get a representation $\widetilde{V}$,
\[\widetilde{V}=\Ind_{\GSp_4(\R)^+}^{\GSp_4(\R)}(V_+).\]
As a representation of $\GSp_4(\R)^+$, $\widetilde{V}$ splits as
\[\widetilde{V}\cong V_+\oplus V_+^\vee,\]
with $k_0$ switching between the two summands. It follows that, up to twists, the discrete series representations of $\GSp_4(\R)$ are parametrized by orbits of certain analytically integral weights under the action of the four element subgroup
\[\{1,w_\alpha,w_{\alpha+\beta},w_{-1}\}\subset W,\]
and the discrete series representations obtained in the same manner as $\widetilde{V}$, without twisting, are self-dual.
\subsection{Near equivalence and induced representations}
In this section we introduce the induced representations whose Langlands quotients we will be interested in. These representations will be induced from the maximal parabolics of $\GSp_4$, and when computing the Eisenstein multiplicity of their Langlands quotients it will be enough, by multiplicity one theorems, to distinguish them up to near equivalence.\\
\indent Now by Theorem \ref{thmcohind}, the pieces of the Franke filtration that can contribute to Eisenstein cohomology are those which are induced from a cuspidal representation of a Levi subgroup which itself has cohomology. For the Levis of the maximal parabolic subgroups of $\GSp_4$, which are both isomorphic to $\GL_2\times\GL_1$, such representations are given by pairs $(F,\psi)$, where $F$ is a holomorphic cuspidal eigenform form of weight at least $2$, and $\psi$ is a Dirichlet character.\\
\indent For such a pair $(F,\psi)$, let us view $\psi$ as a character of $\GL_1(\A)$ in the usual way, and let us write $\tilde{\pi}_F$ for the unitary automorphic representation of $\GL_2(\A)$ attached to $F$. Write $k$ for the weight of $F$ and $\omega_F$ for the nebentypus. Then $\omega_F$ is identified with the central character of $\tilde{\pi}_F$.\\
\indent The archimedean component $\tilde\pi_{F,\infty}$ of $\tilde\pi_F$ is the discrete series representation of $\GL_2(\R)$ of weight $k$ with trivial character on the central $\R_{>0}$. This representation is the sum of the two discrete series representations of $\SL_2(\R)$ with Harish-Chandra parameters $\pm(k-1)$, and occurs in the cohomology of the representation of $\GL_2$ of highest weight $k-2$.\\
\indent Now we have the representation $\tilde\pi_F\boxtimes\psi$ of $\GL_2(\A)\times\GL_1(\A)$; the symbol $\boxtimes$ here is meant to signal that this is an exterior tensor product. We identify the maximal Levis $M_\alpha$ and $M_\beta$ with $\GL_2\times\GL_1$ via the isomorphisms of \eqref{eqgl2gl1levis}. Let $\delta_{P_\alpha(\A)}$ and $\delta_{P_\beta(\A)}$ be the respective modulus characters of $P_\alpha(\A)$ and $P_\beta(\A)$. We note that for $A\in\GL_2(\A)$ and $t\in\GL_1(\A)$, we have
\begin{equation}
\label{eqdeltagsp4}
\delta_{P_\alpha(\A)}(A,t)=\vert\det(A)\vert^2\vert t\vert^{-4},\qquad\delta_{P_\beta(\A)}(A,t)=\vert\det(A)\vert^3t^{-3}.
\end{equation}
\indent Now let $s\in\C$. We define the normalized induced representations
\begin{equation}
\label{eqgsp4normindmax}
\iota_{P_\gamma(\A)}^{\GSp_4(\A)}(\tilde\pi_F\boxtimes\psi,s)=\Ind_{P_\gamma(\A)}^{\GSp_4(\A)}((\tilde\pi_F\boxtimes\psi)\otimes\delta_{P_\gamma}^{s+1/2}),\qquad\gamma\in\{\alpha,\beta\}.
\end{equation}
These representations are trivial on $A_{\GSp_4}(\R)^\circ$.
\begin{proposition}
\label{propequalmaxpara}
Let $\gamma\in\{\alpha,\beta\}$ be one of the simple roots of $\GSp_4$. Let $F,F'$ be holomorphic cuspidal eigenforms, let $\psi,\psi'$ be Dirichlet characters, and let $s,s'\in\R_{>0}$. If there are irreducible subquotients
\[\Pi\textrm{ of }\iota_{P_\gamma(\A)}^{\GSp_4(\A)}(\tilde\pi_F\boxtimes\psi,s)\]
and
\[\Pi'\textrm{ of }\iota_{P_\gamma(\A)}^{\GSp_4(\A)}(\tilde\pi_{F'}\boxtimes\psi',s')\]
such that $\Pi$ and $\Pi'$ are nearly equivalent, then $\tilde\pi_F=\tilde\pi_{F'}$, $\psi=\psi'$, and $s=s'$.
\end{proposition}
\begin{proof}
We will prove this proposition for the Siegel parabolic $P_\beta$; the proof in the Klingen case is analogous.\\
\indent Let $S$ be a finite set of places, including the archimedean place, such that for $p\notin S$, the local components $\Pi_p$ and $\Pi_p'$ are unramified and isomorphic. Then for such $p$, we have in particular that $\tilde{\pi}_{F,p}$ and $\tilde{\pi}_{F',p}$ are unramified, and so are $\psi_p$ and $\psi_p'$. Write $T_2$ for the standard diagonal torus of $\GL_2$ and $B_2$ for the standard upper triangular Borel in $\GL_2$. Then we know that there are unramified characters $\chi_1,\chi_2$ of $\Q_p^\times$ such that $\tilde\pi_{F,p}$ is the unramified subquotient of
\[\Ind_{B_2(\Q_p)}^{\GL_2(\Q_p)}((\chi_1\boxtimes\chi_2)\otimes\delta_{B_2(\Q_p)}^{1/2}),\]
where $\chi_1\boxtimes\chi_2$ is the character of $T_2(\Q_p)$ defined by
\[(\chi_1\boxtimes\chi_2)(\diag(x,y))=\chi_1(x)\chi_2(y)\]
and $\delta_{B_2(\Q_p)}$ is the usual modulus character of $B(\Q_p)$. Similarly, there are also unramified characters $\chi_1',\chi_2'$ of $\Q_p^\times$ such that $\tilde\pi_{F',p}$ is the unramified subquotient of
\[\Ind_{B_2(\Q_p)}^{\GL_2(\Q_p)}((\chi_1'\boxtimes\chi_2')\otimes\delta_{B_2(\Q_p)}^{1/2}).\]
Furthermore, by temperedness, we know that $\chi_1,\chi_2,\chi_1',\chi_2'$ are all unitary.\\
\indent For $x,y,t\in\Q_p^\times$, consider the element of $T(\Q_p)$ given in $M_\beta\cong\GL_2\times\GL_1$ by $(\diag(x,y),t)$. Write $\chi_1\boxtimes\chi_2\boxtimes\psi_p$ for the character of $T(\Q_p)$ given on such elements by
\[(\chi_1\boxtimes\chi_2\boxtimes\psi_p)(\diag(x,y),t)=\chi_1(x)\chi_2(y)\psi_p(t).\]
By induction in stages, we have that $\Pi_p$ is the unramified subquotient of
\[\Ind_{B(\Q_p)}^{\GSp_4(\Q_p)}((\chi_1\boxtimes\chi_2\boxtimes\psi_p)\otimes\delta_{P_\beta(\Q_p)}^s\otimes\delta_{B(\Q_p)}^{1/2}),\]
and similarly $\Pi_p'$ is the unramified subquotient of
\[\Ind_{B(\Q_p)}^{\GSp_4(\Q_p)}((\chi_1'\boxtimes\chi_2'\boxtimes\psi_p')\otimes\delta_{P_\beta(\Q_p)}^{s'}\otimes\delta_{B(\Q_p)}^{1/2}).\]
By the theory of the Satake isomorphism recalled in Section \ref{secsatgal}, since $\Pi_p\cong\Pi_p'$, the characters
\[(\chi_1\boxtimes\chi_2\boxtimes\psi_p)\otimes\delta_{B(\Q_p)}^{s}\quad\textrm{and}\quad(\chi_1'\boxtimes\chi_2'\boxtimes\psi_p')\otimes\delta_{B(\Q_p)}^{s'}\]
are equal up to the Weyl group $W$; that is, there is a $w\in W$ such that for all $x,y,t\in\Q_p^\times$, we have
\begin{equation}
\label{eqcharsdeltas}
((\chi_1\boxtimes\chi_2\boxtimes\psi_p)\otimes\delta_{P_\beta(\Q_p)}^{s})((\diag(x,y),t)^w)=((\chi_1'\boxtimes\chi_2'\boxtimes\psi_p')\otimes\delta_{P_\beta(\Q_p)}^{s'})(\diag(x,y),t).
\end{equation}
\indent First, let us take the absolute value of both sides of \eqref{eqcharsdeltas}. Since all characters involved except $\delta_{P_\beta(\Q_p)}$ are unitary, this gives
\begin{equation}
\label{eqabsvaldeltas1}
\delta_{P_\beta(\Q_p)}^s((\diag(x,y),t)^w)=\delta_{P_\beta(\Q_p)}^{s'}(\diag(x,y),t),
\end{equation}
By the local analogue of \eqref{eqdeltagsp4}, this becomes
\begin{equation}
\label{eqabsvaldeltas2}
\delta_{P_\beta(\Q_p)}^s((\diag(x,y),t)^w)=\vert xy\vert^{3s'}\vert t\vert^{-3s'}.
\end{equation}
Now we compute, using \eqref{eqweylactiongsp4}, the following identities
\begin{align*}
(\diag(x,y),t)^{w_\alpha}&=(\diag(x,ty^{-1}),t),&(\diag(x,y),t)^{w_{\alpha\beta}}&=(\diag(y,tx^{-1}),t),\\
(\diag(x,y),t)^{w_{\beta\alpha}}&=(\diag(ty^{-1},x),t),&(\diag(x,y),t)^{w_{\alpha\beta\alpha}}&=(\diag(ty^{-1},tx^{-1}),t),\\
(\diag(x,y),t)^{w_{\beta\alpha\beta}}&=(\diag(tx^{-1},y),t),&(\diag(x,y),t)^{w_{-1}}&=(\diag(tx^{-1},ty^{-1}),t).
\end{align*}
From these identities follow
\begin{align*}
\delta_{P_\beta(\Q_p)}^{s}((\diag(x,y),t)^{w_\alpha})&=\vert xy^{-1}\vert^{3s},
&\delta_{P_\beta(\Q_p)}^{s}((\diag(x,y),t)^{w_{\alpha\beta}})&= \vert x^{-1}y\vert^{3s},\\
\delta_{P_\beta(\Q_p)}^{s}((\diag(x,y),t)^{w_{\beta\alpha}})&=\vert xy^{-1}\vert^{3s},
&\delta_{P_\beta(\Q_p)}^{s}((\diag(x,y),t)^{w_{\alpha\beta\alpha}})&=\vert xy\vert^{-3s}\vert t\vert^{3s},\\
\delta_{P_\beta(\Q_p)}^{s}((\diag(x,y),t)^{w_{\beta\alpha\beta}})&=\vert x^{-1}y\vert^{3s},
&\delta_{P_\beta(\Q_p)}^{s}((\diag(x,y),t)^{w_{-1}})&=\vert xy\vert^{-3s}\vert t\vert^{3s}.
\end{align*}
Letting $(x,y,t)=(1,1,p)$ in the equations above and using \eqref{eqabsvaldeltas2} then gives
\[p^{3s'}=\begin{cases}
1&\textrm{if }w\in\{w_\alpha,w_{\alpha\beta},w_{\beta\alpha},w_{\alpha\beta\alpha}\};\\
p^{-3s}&\textrm{if }w\in\{w_{\alpha\beta\alpha},w_{-1}\}.
\end{cases}\]
Since $s,s'>0$, this is impossible, which forces $w=1$ or $w=w_\beta$. In either case, an analogous computation as above then gives
\[p^{3s'}=p^{3s},\]
from which we conclude $s=s'$. Then we can cancel the modulus characters in \eqref{eqcharsdeltas} and get
\[(\chi_1\boxtimes\chi_2\boxtimes\psi_p)((\diag(x,y),t)^w)=(\chi_1'\boxtimes\chi_2'\boxtimes\psi_p')(\diag(x,y),t),\qquad\textrm{for some }w\in\{1,w_\beta\}.\]
In the case that $w=1$, we conclude that $\chi_1=\chi_1'$, $\chi_2=\chi_2'$, and $\psi_p=\psi_p'$. If instead $w=w_\beta$, then
\[(\diag(x,y),t)^w=(\diag(y,x),t),\]
and we conclude $\chi_1=\chi_2'$, $\chi_2=\chi_1'$, and $\psi_p=\psi_p'$. In either case we have $\psi_p=\psi_p'$ and that
\[\Ind_{B_2(\Q_p)}^{\GL_2(\Q_p)}((\chi_1\boxtimes\chi_2)\otimes\delta_{B_2(\Q_p)}^{1/2})\quad\textrm{and}\quad\Ind_{B_2(\Q_p)}^{\GL_2(\Q_p)}((\chi_1'\boxtimes\chi_2')\otimes\delta_{B_2(\Q_p)}^{1/2})\]
have the same unramified subquotients, which means $\tilde\pi_{F,p}\cong\tilde\pi_{F',p}$. Since this is true for any $p\notin S$, strong multiplicity one for $\GL_2$ (and $\GL_1$) finishes the proof.
\end{proof}
Next we want to distinguish between representations induced from different standard parabolics, including the Borel. So let us first describe how we will induce characters from the torus $T$.\\
\indent First identify $T$ with $(\GL_1)^3$ via the map \eqref{eqtorusabc}. Let $\psi_1,\psi_2,\psi_3$ be Dirichlet characters, viewed as characters of $\GL_1(\A)$. Write $\psi_1\boxtimes\psi_2\boxtimes\psi_3$ for the character of $T(\A)$ given by
\[(\psi_1\boxtimes\psi_2\boxtimes\psi_3)(t_1,t_2,t_3)=\psi_1(t_1)\psi_2(t_2)\psi_3(t_3).\]
Let $\delta_{B(\A)}$ be the modulus character of $B(\A)$. When restricted to $T(\A)$, this gives
\[\delta_{B(\A)}(t_1,t_2,t_3)=\vert t_1\vert^4\vert t_2\vert^2\vert t_3\vert^{-3}.\]
\indent More generally, for $s_1,s_2\in\C$, we will consider the character of $B(\A)$ given by
\[e^{\langle H_B(\cdot),s_1\alpha+s_2\beta\rangle}.\]
If $\rho=\frac{1}{2}(3\alpha+4\beta)$ is half the sum of the positive roots, then we have
\[\delta_{B(\A)}^{1/2}=e^{\langle H_B(\cdot),\rho\rangle}.\]
We write
\begin{equation}
\label{eqgsp4normindmin}
\iota_{B(\A)}^{\GSp_4(\A)}(\psi_1\boxtimes\psi_2\boxtimes\psi_3;s_1,s_2)=\Ind_{B(\A)}^{\GSp_4(\A)}((\psi_1\boxtimes\psi_2\boxtimes\psi_3)\otimes e^{\langle H_B(\cdot),s_1\alpha+s_2\beta+\rho\rangle})
\end{equation}
for the normalized induction.\\
\indent To distinguish between representations induced from different parabolics, we will attach to them Galois representations and distinguish between those. The next three propositions will do this for $B$, $M_\alpha$, and $M_\beta$, respectively.\\
\indent Fix any prime $\ell$ and fix an isomorphism of $\C$ with $\overline{\Q}_\ell$.
\begin{proposition}
\label{propgsp4galrepB}
Let $\psi_1,\psi_2,\psi_3$ be Dirichlet characters, and let $m_1,m_2\in\Z$. Let $\Pi$ be an irreducible subquotient of
\[\iota_{B(\A)}^{\GSp_4(\A)}(\psi_1\boxtimes\psi_2\boxtimes\psi_3;m_1/2,m_2).\]
Let $j_{T,\GSp_4}$ be the inclusion $T\hookrightarrow\GSp_4$. Then $\Pi\otimes\vert\nu\vert^{m_1/2}$ has attached to it the Galois representation $G_\Q\to\GSp_4(\overline\Q_\ell)$ given by
\[j_{T,\GSp_4}\circ\left((\psi_1\psi_2\psi_3\chi_{\cyc}^{m_2})\times(\psi_1\psi_3\chi_{\cyc}^{m_1-m_2})\times(\psi_1\psi_2\psi_3^2)\right),\]
where we have viewed $\psi_1,\psi_2,\psi_3$ as Galois characters via class field theory.
\end{proposition}
\begin{proof}
Let $p$ be a prime different from $\ell$ which is unramified for $\Pi$, and hence which not divide the conductors of the $\psi_i$'s. Let $\lambda_i=\psi_i(p)$ for $i=1,2,3$. (This is the Satake parameter of $\psi_i$ at $p$.) Then the character
\[(\psi_1\boxtimes\psi_2\boxtimes\psi_3)\otimes e^{\langle H_B(\cdot),(m_1/2)\alpha+m_2\beta\rangle}\]
of $GL_1(\A)^3$ has Satake parameter at $p$
\[(p^{-m_2}\lambda_1,p^{-(m_1-m_2)}\lambda_2,p^{m_1/2}\lambda_3)\in\GL_1(\overline{\Q}_\ell)^3.\]
Therefore
\begin{equation}
\label{eqpfborelcharaut}
(\psi_1\boxtimes\psi_2\boxtimes\psi_3)\otimes e^{\langle H_B(\cdot),(m_1/2)\alpha+m_2\beta\rangle}\otimes\vert\nu\vert^{m_1/2}
\end{equation}
has Satake parameter at $p$
\[(p^{-m_2}\lambda_1,p^{-(m_1-m_2)}\lambda_2,\lambda_3).\]
When identifying $(\GL_1)^3$ with $T$ on the dual side via the map $\varphi_B$ of \eqref{eqdualizeborel1} and \eqref{eqdualizeborel2}, this implies that the character \eqref{eqpfborelcharaut} has attached to it the Galois representation into $T(\overline{\Q}_\ell)$ given by
\[(\psi_1\psi_2\psi_3\chi_{\cyc}^{m_2})\times(\psi_1\psi_3\chi_{\cyc}^{m_1-m_2})\times(\psi_1\psi_2\psi_3^2):G_\Q\to T(\overline\Q_\ell)\]
\indent Now we can pass the similitude twist inside the induction and get that $\Pi\otimes\vert\nu\vert^{m_1/2}$ is a subquotient of the normalized induction of the character \eqref{eqpfborelcharaut}, whence an appeal to Proposition \ref{propindgalrep} finishes the proof.
\end{proof}
\begin{proposition}
\label{propgsp4galrepalpha}
Let $F$ be a holomorphic cuspidal eigenform of weight $k$ with central character $\omega_F$, and let $\psi$ be a Dirichlet character. Let $m\in\Z$, and let $\Pi$ be any irreducible subquotient of
\[\iota_{P_\alpha(\A)}^{\GSp_4(\A)}(\tilde\pi_F\boxtimes\psi,m/4).\]
Let $j_{M_\beta,\GSp_4}$ be the inclusion $M_\beta\hookrightarrow\GSp_4$. Then $\Pi\otimes\vert\nu\vert^{(k-1-m)/2}$ has attached to it the Galois representation $G_\Q\to\GSp_4(\overline\Q_\ell)$ given by
\[j_{M_\beta,\GSp_4}\circ(\rho_F\times\psi\omega_F\chi_{\cyc}^{k-1-m}),\]
where $\rho_F$ is the Galois representation attached to $F$ by Eichler--Shimura, Deligne, and Deligne--Serre (Theorem \ref{thmesdds}), and $\omega_F$ and $\psi$ are identified with Galois characters via class field theory.
\end{proposition}
\begin{proof}
The proof will be similar to the previous proposition. Let $p$ be a prime different from $\ell$ which is unramified for $\Pi$, and hence which is unramified for $\tilde\pi_F$ and $\psi$. Let $\diag(\lambda_1,\lambda_2)\in\GL_2(\overline\Q_\ell)$ be a diagonal representative of the Satake parameter of $\tilde\pi_F$ at $p$, and let $\lambda_3=\psi(p)$. Then the automorphic representation of $M_\alpha(\A)$ given by
\[(\tilde\pi_F\boxtimes\psi)\otimes\delta_{P_\alpha(\A)}^{m/4}\]
has Satake parameter at $p$ represented by
\[(p^{-m/2}\diag(\lambda_1,\lambda_2),p^m\lambda_3)\in\GL_2(\overline\Q_\ell)\times\GL_1(\overline\Q_\ell),\]
by \eqref{eqdeltagsp4}. Thus
\begin{equation}
\label{eqpfalpharepaut}
(\tilde\pi_F\boxtimes\psi)\otimes\delta_{P_\alpha(\A)}^{m/4}\otimes\vert\nu\vert^{(k-1-m)/2}
\end{equation}
has Satake parameter at $p$ represented by
\[(p^{-(k-1)/2}\diag(\lambda_1,\lambda_2),p^m\lambda_3)\in\GL_2(\overline\Q_\ell)\times\GL_1(\overline\Q_\ell),\]
because $\vert\nu(A,t)\vert=\vert\det(A)\vert$ for $(A,t)\in M_\alpha(\A)$.\\
\indent Now we pass through the map $\varphi_\alpha$ of \eqref{eqdualizealpha1} and \eqref{eqdualizealpha2} to get that the representation of \eqref{eqpfalpharepaut} has Satake parameter represented by
\[(p^{-(k-1)/2}\diag(\lambda_1,\lambda_2),p^{-(k-1-m)}\lambda_1\lambda_2\lambda_3)\in\GL_2(\overline\Q_\ell)\times\GL_1(\overline\Q_\ell),\]
and therefore has the Galois representation $G_\Q\to M_\beta(\overline\Q_\ell)$ given by
\[\rho_F\times\psi\omega_F\chi_{\cyc}^{k-1-m}\]
attached to it. Thus we are done by Proposition \ref{propindgalrep}.
\end{proof}
\begin{proposition}
\label{propgsp4galrepbeta}
Let $F$ be a holomorphic cuspidal eigenform of weight $k$, and let $\psi$ be a Dirichlet character. Let $m\in\Z$, and assume $m\equiv k-1\modulo{2}$. Let $\Pi$ be any irreducible subquotient of
\[\iota_{P_\beta(\A)}^{\GSp_4(\A)}(\tilde\pi_F\boxtimes\psi,m/6).\]
Let $j_{M_\alpha,\GSp_4}$ be the inclusion $M_\alpha\hookrightarrow\GSp_4$. Then $\Pi\otimes\vert\nu\vert^{(k-1)/2}$ has attached to it the Galois representation $G_\Q\to\GSp_4(\overline\Q_\ell)$ given by
\[j_{M_\alpha,\GSp_4}\circ((\rho_F\otimes\psi)\times\psi\chi_{\cyc}^{(k-1-m)/2}),\]
where $\rho_F$ is the Galois representation attached to $F$ by Eichler--Shimura, Deligne, and Deligne--Serre (Theorem \ref{thmesdds}), and $\psi$ is identified with a Galois character via class field theory.
\end{proposition}
\begin{proof}
The proof will again be very similar to the previous two propositions. Let $p$ be a prime different from $\ell$ which is unramified for $\Pi$, and hence which is unramified for $\tilde\pi_F$ and $\psi$. Let $\diag(\lambda_1,\lambda_2)\in\GL_2(\overline\Q_\ell)$ be a diagonal representative of the Satake parameter of $\tilde\pi_F$ at $p$, and let $\lambda_3=\psi(p)$. Then the automorphic representation of $M_\alpha(\A)$ given by
\[(\tilde\pi_F\boxtimes\psi)\otimes\delta_{P_\beta(\A)}^{m/6}\]
has Satake parameter at $p$ represented by
\[(p^{-m/2}\diag(\lambda_1,\lambda_2),p^{m/2}\lambda_3)\in\GL_2(\overline\Q_\ell)\times\GL_1(\overline\Q_\ell),\]
by \eqref{eqdeltagsp4}. Thus
\begin{equation}
\label{eqpfbetarepaut}
(\tilde\pi_F\boxtimes\psi)\otimes\delta_{P_\alpha(\A)}^{m/4}\otimes\vert\nu\vert^{(k-1)/2}
\end{equation}
has Satake parameter at $p$ represented by
\[(p^{-m/2}\diag(\lambda_1,\lambda_2),p^{-(k-1-m)/2}\lambda_3)\in\GL_2(\overline\Q_\ell)\times\GL_1(\overline\Q_\ell),\]
because $\vert\nu(A,t)\vert=\vert t\vert$ for $(A,t)\in M_\beta(\A)$.\\
\indent Now we pass through the map $\varphi_\beta$ of \eqref{eqdualizebeta1} and \eqref{eqdualizebeta2} to get that the representation of \eqref{eqpfbetarepaut} has Satake parameter represented by
\[(p^{-(k-1)/2}\lambda_3\diag(\lambda_1,\lambda_2),p^{-(k-1-m)}\lambda_3)\in\GL_2(\overline\Q_\ell)\times\GL_1(\overline\Q_\ell),\]
and therefore has the Galois representation $G_\Q\to M_\alpha(\overline\Q_\ell)$ given by
\[\rho_F\times\psi\omega_F\chi_{\cyc}^{(k-1-m)/2}\]
attached to it. Thus we are done once again by Proposition \ref{propindgalrep}.
\end{proof}
\begin{proposition}
\label{propdistsiegkling}
Let $F_\alpha,F_\beta$ be two holomorphic cuspidal eigenforms of weights $k_\alpha$ and $k_\beta$, respectively. Let $\psi_\alpha,\psi_\beta,\psi_1,\psi_2,\psi_3$ be Dirichlet characters, and let $m_\alpha,m_\beta,m_1,m_2\in\Z$. Assume that $m_\beta\equiv k_\beta-1\modulo{2}$. Then given any irreducible subquotients
\[\Pi_\alpha\textrm{ of }\iota_{P_\alpha(\A)}^{\GSp_4(\A)}(\tilde\pi_{F_\alpha}\boxtimes\psi_\alpha,m_\alpha/4)\]
and
\[\Pi_\beta\textrm{ of }\iota_{P_\beta(\A)}^{\GSp_4(\A)}(\tilde\pi_{F_\beta}\boxtimes\psi_\beta,m_\beta/6)\]
and
\[\Pi_0\textrm{ of }\iota_{B(\A)}^{\GSp_4(\A)}(\psi_1\boxtimes\psi_2\boxtimes\psi_3;m_1/2,m_2),\]
we have that no two of $\Pi_\alpha$, $\Pi_\beta$, and $\Pi_0$ are nearly equivalent.
\end{proposition}
\begin{proof}
We first prove the proposition in the case that
\begin{equation}
\label{eqparitiesgsp4}
k_\alpha-1-m_\alpha\equiv k_\beta-1\equiv m_1\modulo{2}.
\end{equation}
Assume moreover that the quantities of \eqref{eqparitiesgsp4} are all even. Then Propositions \ref{propgsp4galrepalpha}, \ref{propgsp4galrepbeta}, and \ref{propgsp4galrepB} attach to $\Pi_\alpha$, $\Pi_\beta$, and $\Pi_0$, respectively, a Galois representation (which, by the parity assumption just made, are central twists by an integral power of the cyclotomic character of the Galois representations from those propositions). Let $\rho_\alpha$, $\rho_\beta$, and $\rho_0$, respectively, be these Galois representations. Denote by $\std$ the standard representation of $\GSp_4$ into $\GL_4$ which we used to define the group $\GSp_4$. Then we have the following formulas for our Galois representations when composed with $\std$:
\[\std(\rho_\alpha\otimes\chi_{\cyc}^{(k_\alpha-1-m_\alpha)/2})=\rho_{F_\alpha}\oplus(\rho_{F_\alpha}^\vee\otimes(\psi_\alpha\omega_{F_\alpha}\chi_{\cyc}^{k_\alpha-1-m_\alpha}));\]
\[\std(\rho_\beta\otimes\chi_{\cyc}^{(k_\beta-1})/2)=(\rho_{F_\beta}\otimes\psi_\beta)\oplus(\omega_{F_\beta}\psi_\beta\chi_{\cyc}^{(k_\beta-1+m_\beta)})\oplus(\psi_\beta\chi_{\cyc}^{(k_\beta-1-m_\beta)/2});\]
\[\std\circ(\rho_0\otimes\chi_{\cyc}^{m_1/2})=(\psi_1\psi_2\psi_3\chi_{\cyc}^{m_2})\oplus(\psi_1\psi_3\chi_{\cyc}^{m_1-m_2})\oplus (\psi_3\chi_{\cyc}^{-m_2})\oplus(\psi_2\psi_3^2\chi_{\cyc}^{m_2-m_1}).\]
These follow from the usual formulas we use to include the Levis $M_\alpha$, $M_\beta$, and $T$ into $\GSp_4$; here, of course, $\omega_{F_\alpha}$ is the nebentypus of $F_\alpha$ and all Dirichlet characters are identified with Galois characters via class field theory.\\
\indent Now since $\rho_{F_\alpha}$ and $\rho_{F_\beta}$ are irreducible, the three representations above (and any twist of them) are semisimple. In particular, $\std\circ\rho_\alpha$ is the sum of two irreducible representations, $\std\circ\rho_\beta$ is the sum of three, and $\std\circ\rho_0$ is the sum of four. Therefore these representations are pairwise non-isomorphic, and we can now appeal to Proposition \ref{propdistgalrep} to conclude when the quantities of \eqref{eqparitiesgsp4} are all even.\\
\indent On the other hand, when the quantities of \eqref{eqparitiesgsp4} are all odd, we can just apply a completely analogous argument to attach Galois representations to the twisted representations $\Pi_\alpha\otimes\vert\nu\vert^{1/2}$, $\Pi_\beta\otimes\vert\nu\vert^{1/2}$, and $\Pi_0\otimes\vert\nu\vert^{1/2}$, and we conclude in this case too.\\
\indent Finally, assume that one of the quantities in \eqref{eqparitiesgsp4} is even and another is odd. Let $\Pi$ be the representation of $\Pi_\alpha$, $\Pi_\beta$, and $\Pi_0$ corresponding to the even quantity, and let $\Pi'$ be the one corresponding to the odd quantity. Then, as we just saw, $\Pi$ has attached to it a Galois representation $\rho_\Pi:G_\Q\to\GSp_4(\overline\Q_\ell)$ such that $\std\circ\rho_\Pi$ is semisimple and Hodge-Tate. But $\Pi'$ may not have a Galois representation attached to it; we only know that $\Pi'\otimes\vert\nu\vert^{1/2}$ does. If there is no such Galois representation, then we are done. Otherwise, it does have a Galois representation, call it $\rho_{\Pi'}$, and we may assume $\ell$ is odd. We then restrict to $G_{\Q(\zeta_\ell)}$, where $\zeta_\ell$ is a primitive $\ell$th root of unity. Then $\chi_{\cyc}$ has a square root, and the representation of $G_{\Q(\zeta_\ell)}$ given by $(\rho_{\Pi'}\otimes\chi_{\cyc}^{1/2})^{\sss}$ must be the restriction to $G_{\Q(\zeta_\ell)}$ of the Galois representation attached to $\Pi'\otimes\vert\nu\vert^{1/2}$; they are both semisimple and their traces agree on Frobenius elements $\Frob_p$ with $p\equiv 1\modulo{\ell}$. But since the Galois representation attached to $\Pi'\otimes\vert\nu\vert^{1/2}$ is Hodge--Tate, $\rho_{\Pi'}$ cannot be, and this distinguishes $\Pi'$ from $\Pi$, as desired.
\end{proof}
\begin{remark}
Some of the assumptions above on the parameters in the proof may look strange at first, but there is an explanation for them. If one computes exactly which representations induced from $M_\alpha$, $M_\beta$, and $T$ can have cohomology for a given representation $E$ of $\GSp_4(\C)$, their parameters will satisfy \eqref{eqparitiesgsp4}. More precisely, if $E$ has highest weight $\Lambda$, and if the representations $\Pi_\alpha$, $\Pi_\beta$, and $\Pi_0$ of the proposition appear in the cohomology of $E$, then the quantities of \eqref{eqparitiesgsp4} are all even if $\Lambda+\rho$ is in the integral span of the root lattice, and they are all odd otherwise. In either of these cases, the quantity $k_\beta-1-m_\beta$ is always even.\\
\indent This is to be expected for the following reason. In the case that $\Lambda+\rho$ is in the integral span of the root lattice, the Galois representations attached to the automorphic representations appearing in the cohomology of $E$ should be de Rham with Hodge--Tate weights given by the cocharacter of $T^\vee$ corresponding to the infinitesimal character of these automorphic representations at infinity. This infinitesimal character must then match that of $E$, and is therefore given by the integral parameter $\Lambda+\rho$. In the case of $\Pi_\alpha$, $\Pi_\beta$, and $\Pi_0$, these Galois representations are described up to a twist by a power of the cyclotomic character respectively by Propositions \ref{propgsp4galrepalpha}, \ref{propgsp4galrepbeta}, and \ref{propgsp4galrepB}, and because the quantities of \eqref{eqparitiesgsp4} are all even, the power we are twisting by is integral. Are applying that twist, the Hodge--Tate weights of these Galois representations will match the cocharacter of $T^\vee$ given by $\Lambda+\rho$.\\
\indent On the other hand, if $\Lambda+\rho$ is not in the integral span of the root lattice, then the automorphic representations appearing in the cohomology of $E$ only have associated Galois representations (at least ones which are de Rham) after twisting by a half power of the similitude character. Correspondingly, since the quantities of \eqref{eqparitiesgsp4} are all odd in this case, the representations $\Pi_\alpha$, $\Pi_\beta$, and $\Pi_0$ must also be twisted by a half power of the similitude character to obtain nice Galois representations.
\end{remark}
\subsection{Eisenstein multiplicity of Langlands quotients}
\label{seceismultgsp4}
In this section we introduce the Langlands quotients we are interested in and compute their multiplicities in Eisenstein cohomology. Before we do that, however, let us compute the cohomology of certain induced representations of the kind considered in Theorem \ref{thmcohind}. We do this in the next proposition for representations induced from the Siegel parabolic of $\GSp_4$.\\
\indent In what follows, we will be considering the $(\mf{g}_0,K_\infty^\circ)$-cohomology of representations when $G=\GSp_4$. In this case we have $\mf{g}_0=\sp_4$, the complexified Lie algebra of $\Sp_4$. As discussed in Section \ref{secgroupgsp4}, the group $K_\infty$ has two connected components, and so the cohomology spaces we obtain will be modules for the two element group of components of $K_\infty$, as well as for the group $\GSp_4(\A_f)$.\\
\indent We will also consider the normalized induction functors $\iota_{P(\A)}^{\GSp_4(\A)}$, for $P$ a standard parabolic, defined in \eqref{eqgsp4normindmax} and \eqref{eqgsp4normindmin}, and also their finite adelic analogues $\iota_{P(\A_f)}^{\GSp_4(\A_f)}$ which are defined similarly.\\
\indent The following proposition is essentially proved by Grbac and Grobner in \cite{GG}, Proposition 4.2, using the same techniques as the ones we use. The main differences are that Grbac and Grobner work with $\Sp_4$ instead of $\GSp_4$, which is not a serious difference, and that they also obtain results for totally real fields instead of just $\Q$. Actually, we have set things up so that it is possible to use the results in this paper to obtain results over totally real fields as well, but we are content with working over $\Q$ for simplicity.
\begin{proposition}
\label{propcohindbeta}
Let $E$ be an irreducible, finite dimensional representation of $\GSp_4(\C)$, and say that $E$ has highest weight $\tilde\Lambda$. Let $\Lambda=\tilde\Lambda|_{T_0}$, so that there are $c_1,c_2\in\Z_{\geq 0}$ such that
\[\Lambda=\frac{c_1}{2}(\alpha+2\beta)+c_2(\alpha+\beta).\]
Let $F$ be a holomorphic cuspidal eigenform of weight $k$ and trivial nebentypus, and let $s\in\C$ with $\re(s)\geq 0$. Assume
\[H^i(\sp_4,K_\infty^\circ;\Ind_{P_\beta(\A)}^{\GSp_4(\A)}((\tilde\pi_F\boxtimes 1)\otimes\Sym(\mf{a}_{P_\beta,0})_{(2s+1)\rho_{P_\beta}})\otimes E)\ne 0.\]
Then either:
\begin{enumerate}[label=(\roman*)]
\item We have
\[i=3,\qquad k=c_1+2c_2+4,\qquad s=\frac{c_1+1}{6},\]
and
\begin{multline*}
H^3(\sp_4,K_\infty^\circ;\Ind_{P_\beta(\A)}^{\GSp_4(\A)}((\tilde\pi_F\boxtimes 1)\otimes\Sym(\mf{a}_{P_\beta,0})_{(2s+1)\rho_{P_\beta}})\otimes E)\\
\cong\iota_{P_\beta(\A_f)}^{\GSp_4(\A_f)}(\tilde\pi_{F,f}\boxtimes 1,(c_1+1)/6),
\end{multline*}
or,
\item We have
\[i=4,\qquad k=c_1+2,\qquad s=\frac{c_1+2c_2+3}{6},\]
and
\begin{multline*}
H^4(\sp_4,K_\infty^\circ;\Ind_{P_\beta(\A)}^{\GSp_4(\A)}((\tilde\pi_F\boxtimes 1)\otimes\Sym(\mf{a}_{P_\beta,0})_{(2s+1)\rho_{P_\beta}})\otimes E)\\
\cong\iota_{P_\beta(\A_f)}^{\GSp_4(\A_f)}(\tilde\pi_{F,f}\boxtimes 1,(c_1+2c_2+3)/6).
\end{multline*}
\end{enumerate}
In both cases the cohomology spaces have the trivial action of the component group of $K_\infty$.
\end{proposition}
\begin{proof}
We will apply Theorem \ref{thmcohind} to our present situation with $\pi=(\tilde\pi_F\boxtimes 1)\otimes\delta_{P_\beta(\A)}^s$ and $\mf{h}=\mf{t}$, the complexified Lie algebra of $T$. In fact, it suffices to do all our computations restricted to the complexified Lie algebra $\mf{t}_0$ of $T_0$, which is a Cartan subalgebra of $\sp_4$. We have
\[W^{P_\beta}=\{1,w_\alpha,w_{\alpha\beta},w_{\alpha\beta\alpha}\},\]
and one readily computes
\begin{align*}
-(\Lambda+\rho)&=-(c_1+1)\frac{\beta}{2}-(c_1+2c_2+3)\frac{\alpha+\beta}{2},\\
-w_\alpha(\Lambda+\rho)&=-(c_1+2c_2+3)\frac{\beta}{2}-(c_1+1)\frac{\alpha+\beta}{2},\\
-w_{\alpha\beta}(\Lambda+\rho)&=-(c_1+2c_2+3)\frac{\beta}{2}+(c_1+1)\frac{\alpha+\beta}{2},\\
-w_{\alpha\beta\alpha}(\Lambda+\rho)&=-(c_1+1)\frac{\beta}{2}+(c_1+2c_2+3)\frac{\alpha+\beta}{2}.
\end{align*}
Note that we have a decomposition
\[\mf{t}_0=(\mf{m}_{\beta,0}\cap\mf{t}_0)\oplus\mf{a}_{P_\beta,0},\]
and note also that $(\alpha+\beta)$ acts as zero on the first summand, while $\beta$ acts as zero on the second.\\
\indent Now by Theorem \ref{thmcohind}, in order for our cohomology space to be nontrivial, we need there to be a $w\in W^{P_\beta}$ with
\[-w(\Lambda+\rho)|_{\mf{a}_{P_\beta},0}=2s\rho_{P_\beta}=6s\frac{\alpha+\beta}{2},\]
and
\[-w(\Lambda+\rho)|_{\mf{m}_{\beta,0}}=\pm(k-1)\frac{\beta}{2}.\]
Therefore, because $\re(s)\geq 0$, we see from the formulas for $-w(\Lambda+\rho)|_{\mf{a}_{P_\beta},0}$ that $w$ can only equal $w_{\alpha\beta}$ or $w_{\alpha\beta\alpha}$.\\
\indent In the case that $w=w_{\alpha\beta}$, we obtain by matching coefficients that
\[k-1=+(c_1+2c_2+3),\]
with this choice of sign because $k-1\geq 0$, and
\[6s=c_1+1.\]
We have that the length $\ell(w_{\alpha\beta})$ of $w_{\alpha\beta}$ is $2$. Also, since
\[\rho=\frac{\beta}{2}+3\frac{\alpha+\beta}{2},\]
we have
\[(w_{\alpha\beta}(\Lambda+\rho)-\rho)|_{\mf{m}_{\beta,0}}= (c_1+2c_2+2)\frac{\beta}{2}=(k-2)\frac{\beta}{2}.\]
Therefore, the isomorphism of Theorem \ref{thmcohind} in our case is
\begin{multline*}
H^i(\sp_4,K_\infty^\circ;\Ind_{P_\beta(\A)}^{\GSp_4(\A)}((\tilde\pi_F\boxtimes 1)\otimes\Sym(\mf{a}_{P_\beta,0})_{(2s+1)\rho_{P_\beta}})\otimes E)\\
\cong\iota_{P_\beta(\A_f)}^{\GSp_4(\A_f)}(\tilde\pi_{F,f}\boxtimes 1,(c_1+1)/6)\otimes H^{i-2}(\mf{m}_{\beta,0},K_\infty^\circ\cap P_\beta(\R);(\tilde\pi_{F,\infty}\boxtimes 1)\otimes F_{k-2}),
\end{multline*}
where $F_{k-2}$ is the representation of $\mf{m}_{\beta,0}$ of highest weight $(k-2)(\beta/2)$.\\
\indent Now, since $k-1=c_1+2c_2+3>0$, the representation $\tilde\pi_{F,\infty}$ is the discrete series representation of $\GL_2(\R)$ of weight $k$, and therefore has nontrivial cohomology when tensored with $F_{k-2}$ in degree $1$ and degree $1$ only. Since $K_\infty^\circ\cap\GL_2(\R)$ is a maximal compact subgroup of $\GL_2(\R)$ (instead of just being its identity component) the cohomology of $\tilde\pi_{F,\infty}$ in degree $1$ is $1$ dimensional (instead of being $2$ dimensional). The claim (i) of our proposition is now immediate.\\
\indent The computation which uses $w_{\alpha\beta\alpha}$ and which proves the claim (ii) of the proposition is completely similar, and we omit it. If instead we decided to take $(\sp_4,K_\infty)$-cohomology, rather than $(\sp_4,K_\infty^\circ)$-cohomology, then we would obtain the same results. This is because we decided to induce the trivial character on the $\GL_1$ component of $M_\beta$, and the maximal compact subgroup $\{\pm 1\}$ of $\GL_1(\R)$ acts trivially via this character. It follows that the component group of $K_\infty$ acts trivially on the cohomology, and this finishes the proof.
\end{proof}
Now fix $F$ a holomorphic cuspidal eigenform of weight $k\geq 2$ and trivial nebentypus. For $s\in\C$ with $\re(s)>0$, let us write
\[\mc{L}_\beta(\tilde\pi_F,s)=\textrm{Langlands quotient of }\iota_{P_\beta(\A)}^{\GSp_4(\A)}(\tilde\pi_F\boxtimes 1,s).\]
This notion was introduced just before Theorem \ref{thmgrbac}.\\
\indent The Langlands quotient $\mc{L}_\beta(\tilde\pi_F,s)$ is irreducible, and under a vanishing assumption on the $L$-function of $\tilde\pi_F$, we will calculate the multiplicity of the finite part $\mc{L}_\beta(\tilde\pi_F,s)_f$ in the Eisenstein cohomology of $\GSp_4$. The following lemma will be key to this.
\begin{lemma}
\label{lemholoESgsp4}
For any flat section $\phi_s\in\iota_{P_\beta(\A)}^{\GSp_4(\A)}(\tilde\pi_F\boxtimes 1,s)$, the Eisenstein series $E(\phi,2s\rho_{P_\beta})$ does not have a pole for $\re(s)>0$ except perhaps if $s=1/6$. If furthermore
\[L(\tilde\pi_F,1/2)=0,\]
then $E(\phi,2s\rho_{P_\beta})$ is also holomorphic at $s=1/6$.
\end{lemma}
\begin{proof}
This is an easy consequence of what is done in the paper of Kim \cite{kim}, but let us quickly explain how this is proved, since we have set up the tools to do so already.\\
\indent It suffices to prove the lemma for $\phi=\bigotimes_v\phi_v$ decomposable into local sections. Write $E(\phi,s)= E(\phi,2s\rho_{P_\beta})$. By Theorem \ref{thmconstterm}, the constant term of $E(\phi,s)$ along $P_\alpha$ (and hence along $B$) is zero, and the constant term along $P_\beta$ is
\[E_{P_\beta}(\phi,s)=\phi_s+M(\phi,w_{\alpha\beta\alpha})_{-2s\rho_{P_\beta}}.\]
Then we apply Theorem \ref{thmLSmethod}; in our current setting the root $\gamma$ of that theorem is $\beta$, and $\tilde\beta=\rho_{P_\beta}/3$, and adjusting for this gives
\[M(\phi,w_{\alpha\beta\alpha})_{-2s\rho_{P_\beta}}=\prod_{j=1}^m\frac{L^S(3js,\tilde\pi_F,R_i^\vee)}{L^S(3js+1,\tilde\pi_F,R_i^\vee)}\bigotimes_{v\notin S}\phi_{v,s}^{w_{\alpha\beta\alpha},\sph}\otimes\bigotimes_{v\in S}M_v(\phi_{v,s},w_{\alpha\beta\alpha})_{-2s\rho_{P_\beta}},\]
where $S$ is a finite set of places such that for $v\notin S$, $\phi_{v,s}$ is spherical, and $\phi_{v,s}^{w_{\alpha\beta\alpha},\sph}$ are certain spherical vectors. Also, the representations $R_i$ of $M_\beta^\vee$ can be determined from the action of the Levi of $P_\alpha$ on its unipotent radical; there are two of them, and $R_1$ is the standard representation of $\GL_2$, and $R_2$ is the determinant. Thus the quotient of $L$-functions is
\[\frac{L^S(3s,\tilde\pi_F)\zeta^S(6s)}{L^S(3s+1,\tilde\pi_F)\zeta^S(6s+1)}.\]
\indent Now by Harish-Chandra, the local intertwining operators are all holomorphic for $\re(s)>0$ since $\tilde\pi_F$ is tempered. So we only have to worry about the poles and zeros of the $L$-functions in the quotient above. Again since $\re(s)>0$, the $L$-functions in the denominator do not vanish as they are in the range of convergence, and the only pole in the numerator comes from the $\zeta$-function at $s=1/6$. But if $L(\tilde\pi_F,1/2)=0$, this zero cancels with the pole from the $\zeta$-function.\\
\indent Since the poles of $E(\phi,s)$ are determined by the poles of the constant term at all standard proper parabolics, we are done.
\end{proof}
We are now ready to put everything together and compute the Eisenstein multiplicity of $\mc{L}_\beta(\tilde\pi_F\boxtimes 1,s)$ for $\re(s)>0$. See Definition \ref{defmult} for the definition of this multiplicity.
\begin{theorem}
\label{thmeismultgsp4}
Let $E$ be an irreducible representation of $\GSp_4(\C)$, and say that $E$ has highest weight $\tilde\Lambda$. Let $\Lambda=\tilde\Lambda|_{T_0}$, so that there are $c_1,c_2\in\Z_{\geq 0}$ such that
\[\Lambda=\frac{c_1}{2}(\alpha+2\beta)+c_2(\alpha+\beta).\]
Let $F$ be a holomorphic cuspidal eigenform of weight $k$ and trivial nebentypus, and let $s\in\C$ with $\re(s)>0$. If $c_1=0$ and $k=2c_2+4$, also assume that
\[L(\tilde\pi_F,1/2)=0.\]
Then
\[m_{[P_\beta]}^i(\mc{L}_\beta(\tilde\pi_F\boxtimes 1,s)_f,K_\infty^\circ,E)=\begin{cases}
1&\textrm{if }i=3,\,\,k=c_1+2c_2+4,\,\,s=(c_1+1)/6\\
&\textrm{or if }i=4,\,\,k=c_1+2,\,\,s=(c_1+2c_2+3)/6;\\
0&\textrm{otherwise};
\end{cases}\]
and
\[m_{[P_\alpha]}^i(\mc{L}_\beta(\tilde\pi_F\boxtimes 1,s)_f,K_\infty^\circ,E)=m_{[B]}^i(\mc{L}_\beta(\tilde\pi_F\boxtimes 1,s)_f,K_\infty^\circ,E)=0.\]
Therefore we also have
\[m_{\Eis}^i(\mc{L}_\beta(\tilde\pi_F\boxtimes 1,s)_f,K_\infty^\circ,E)=m_{[P_\beta]}^i(\mc{L}_\beta(\tilde\pi_F\boxtimes 1,s)_f,K_\infty^\circ,E).\]
Finally, all of these multiplicities are the same if we replace $K_\infty^\circ$ by $K_\infty$.
\end{theorem}
\begin{proof}
There are four associate classes of parabolics for $\GSp_4$ and they are equal to the conjugacy classes of such. From the Franke--Schwermer decomposition (Theorem \ref{thmfsdecomp}) we have that the Eisenstein cohomology decomposes as
\[H_{\Eis}^i(\mf{sp}_4,K_\infty^\circ;\mc{A}_E(\GSp_4)\otimes E)=\bigoplus_{P\in\{P_\alpha,P_\beta,B\}}\bigoplus_{\varphi\in\Phi_{E,[P]}}H^i(\mf{sp}_4,K_\infty^\circ;\mc{A}_{E,[P],\varphi}(\GSp_4)\otimes E).\]
We will study the summands corresponding to $P_\beta$, $P_\alpha$, and $B$ in what follows. The strategy for the $P_\beta$ summand will be to show that if the representation $\mc{L}_\beta(\tilde\pi_F\boxtimes 1,s)$ occurs as a subquotient of one of these summands, then the corresponding associate class in $\Phi_{E,[P_\beta]}$ is the unique one that contains $(\tilde\pi\boxtimes 1)\otimes\delta_{M_\beta(\A)}^{s}$. Then Proposition \ref{propcohindbeta} will allow us to deduce the $[P_\beta]$-Eisenstein multiplicity claimed. In the remaining cases of $P_\alpha$ and $B$, we just show that none of the summands of the cohomology corresponding to these parabolic subgroups can contain $\mc{L}_\beta(\tilde\pi_F\boxtimes 1,s)$ as a subquotient, the key input being Proposition \ref{propdistsiegkling}.\\
\indent \it Case of $P_\beta$. \rm Let $\varphi'$ be an associate class of cuspidal automorphic representations for $E$ and $[P_\beta]$ as in Section \ref{sectfsdecomp}. Then $\varphi'$ contains a cuspidal automorphic representation of $M_\beta(\A)$ which tranforms trivially under $A_{\GSp_4}(\R)^\circ$, and which therefore must be of the form
\[(\tilde\pi'\boxtimes\psi')\otimes\delta_{M_\beta(\A)}^{s'}\]
where $\tilde\pi'$ is a unitary cuspidal automorphic representation of $\GL_2(\A)$, $\psi'$ is a Dirichlet character, and $s'\in\C$. After possibly conjugating by $w_{\alpha\beta\alpha}$, we may even assume $\re(s')\geq 0$.\\
\indent We will study the piece $\mc{A}_{E,[P_\beta],\varphi'}(\GSp_4)$ of the Franke--Schwermer decomposition using Theorem \ref{thmgrbac} of Grbac. But first, we note that the infinitesimal character of $\mc{A}_{E,[P_\beta],\varphi'}(\GSp_4)$ as an $(\sp_4,K_\infty)$-module must match that of $E$. The former is given in terms of the representations in $\varphi'$ by the Weyl orbit of $\lambda_{\tilde\pi'}+2s'\rho_{P_\beta}$, where $\lambda_{\tilde\pi'}$ is the infinitesimal character of $\tilde\pi'$, and the latter is given by $\Lambda+\rho$. But the weight $\Lambda+\rho$ is regular and real, and so since $\lambda_{\pi'}$ is a multiple of the root $\beta$ and $\rho_{P_\beta}$ is a multiple of the root $\alpha+\beta$, it follows that $\lambda_{\pi'}$ and $s'$ are real and nonzero. In particular, $s'>0$ since we assumed $\re(s')\geq 0$.\\
\indent Now we apply Theorem \ref{thmgrbac} and Proposition \ref{propEh0} to find that the cohomology space
\[H^*(\sp_4,K_\infty^\circ;\mc{A}_{E,[P_\beta],\varphi'}(\GSp_4)\otimes E),\]
if nontrivial, is made up of subquotients of the cohomology spaces
\begin{equation}
\label{eqcohlbeta}
H^*(\sp_4,K_\infty^\circ;\mc{L}_\beta(\tilde\pi'\boxtimes\psi',s')\otimes E)
\end{equation}
and
\begin{equation}
\label{eqcohindbeta}
H^*(\sp_4,K_\infty^\circ;\otimes\Ind_{P_\beta(\A)}^{\GSp_4(\A)}((\tilde\pi'\boxtimes\psi')\otimes\Sym(\mf{a}_{P_\beta,0})_{(2s'+1)\rho_{P_\beta}})\otimes E).
\end{equation}
\indent We claim that if \eqref{eqcohlbeta} is nonzero, then $\tilde\pi'$ is cohomological. This will imply that $\tilde\pi'$ is attached to a cuspidal holomorphic eigenform of weight at least $2$. To start, we split into two cases: Either $\tilde\pi_\infty'$ is tempered or nontempered. Of course, by Selberg's conjecture, the latter possibility should not occur, but we will use the following ad-hoc argument to bypass a dependence on this conjecture.\\
\indent So assume now, for sake of contradiction, both that the cohomology space \eqref{eqcohlbeta} is nontrivial and that $\tilde\pi_\infty'$ is nontempered. By the Langlands classification for real groups, $\tilde\pi_\infty'$ is the Langlands quotient of a representation induced from a character, say $\chi$, of $T(\R)$, and then $\mc{L}_\beta(\tilde\pi'\boxtimes\psi',s')_\infty$ is the Langlands quotient of a representation induced from $\chi\delta_{P_\beta(\R)}^{s'}$. If $\mc{L}_\beta(\tilde\pi'\boxtimes\psi',s')_\infty\otimes E$ has nontrivial $(\sp_4,K_\infty^\circ)$-cohomology, then by \cite{BW}, Theorem VI.1.7 (iii) (or rather, the analogue of this theorem with twisted coefficients) so does the (normalized) induced representation
\[\iota_{B(\R)}^{\GSp_4(\R)}(\chi\delta_{P_\beta(\R)}^{s'}).\]
By \cite{BW}, Theorem III.3.3 and induction in stages, the induction
\[\iota_{(B\cap\GL_2)(\R)}^{\GL_2(\R)}(\chi\delta_{P_\beta(\R)}^{s'})\]
has nontrivial $(\sl_2,\O(2))$-cohomology when twisted by some finite dimensional representation of $\GL_2(\C)$, and hence so does
\[\iota_{(B\cap\GL_2)(\R)}^{\GL_2(\R)}(\chi)\]
since $\delta_{P_\beta(\R)}$ is trivial on $\SL_2(\R)$. Thus by \cite{BW}, Theorem VI.1.7 (ii), $\tilde\pi_\infty'$, which is the Langlands quotient of this induction, also has cohomology. But the cohomological cusp forms for $\GL_2$ are the holomorphic modular forms, which are in particular tempered at infinity. This is a contradiction.\\
\indent Therefore, still assuming \eqref{eqcohlbeta} is nonzero, we must have $\tilde\pi_\infty'$ is tempered. Then by (the twisted version of) \cite{BW}, Lemma VI.1.5,
\[H^*(\sp_4,K_\infty^\circ;\iota_{P_\beta(\R)}^{\GSp_4(\R)}((\tilde\pi'\boxtimes\psi')_\infty,s')\otimes E)\ne 0.\]
But by \cite{BW}, Theorem III.3.3, this is computed in terms of the cohomology of $\tilde\pi_\infty'$, and we conclude that $\tilde\pi'$ is cohomological, as desired.\\
\indent If instead \eqref{eqcohindbeta} is nonzero, then we can use Theorem \ref{thmcohind} to conclude that $\tilde\pi'$ is cohomological. In any case, if
\[H^*(\sp_4,K_\infty^\circ;\mc{A}_{E,[P_\beta],\varphi'}(\GSp_4)\otimes E)\ne 0,\]
then $\tilde\pi'=\tilde\pi_{F'}$ for some  cuspidal holomorphic eigenform $F'$ of weight at least $2$. Furthermore, any irreducible subquotient of this cohomology space must be an irreducible subquotient of either \eqref{eqcohindbeta} or \eqref{eqcohlbeta}. The former, by Theorem \ref{thmcohind} is a sum of copies of
\[\iota_{P_\beta(\A_f)}^{\GSp_4(\A_f)}((\tilde\pi_{F'}\boxtimes\psi')_f,s'),\]
while the latter is a sum of copies of the Langlands quotient of this induction. In particular, they are all nearly equivalent and occur in this induction.\\
\indent So if we now assume that
\[H^*(\sp_4,K_\infty^\circ;\mc{A}_{E,[P_\beta],\varphi'}(\GSp_4)\otimes E)\]
contains $\mc{L}_\beta(\tilde\pi_{F}\boxtimes\psi,s)_f$ as a subquotient, then since we have shown $s'>0$, by Proposition \ref{propequalmaxpara}, $\tilde\pi'=\tilde\pi_F$, $\psi'=1$, and $s=s'$.\\
\indent Therefore we have just shown that $\varphi'$ contains $(\tilde\pi_F\boxtimes 1)\otimes\delta_{P_\beta(\A)}^s$. Since no two classes $\varphi'$ overlap, this determines $\varphi'$ uniquely. By Proposition \ref{propEh0}, Proposition \ref{lemholoESgsp4} and our vanishing assumption on the $L$-function of $\tilde\pi_F$, we have
\[\mc{A}_{E,[P_\beta],\varphi}(\GSp_4)\cong\Ind_{P_\beta(\A)}^{\GSp_4(\A)}((\tilde\pi_F\boxtimes 1)\otimes\Sym(\mf{a}_{P_\beta,0})_{(2s+1)\rho_{P_\beta}}),\]
and then Proposition \ref{propcohindbeta} gives the $[P_\beta]$-Eisenstein multiplicities claimed.\\
\indent \it Case of $P_\alpha$. \rm Let $\varphi'$ this time be an associate class for $E$ and $P_\alpha$. Then $\varphi'$ contains a representation of the form
\[(\tilde\pi'\boxtimes\psi')\otimes\delta_{M_\alpha(\A)}^{s'}\]
with $\tilde\pi'$ a unitary cuspidal automorphic representation of $\GL_2(\A)$, $\psi'$ a Dirichlet character, and $s'\in\C$ with $\re(s')\geq 0$. Then the same argument as in the $P_\beta$ case shows that $s'$ is real and positive.\\
\indent Now we once again apply Theorem \ref{thmgrbac} and Proposition \ref{propEh0} to find that the cohomology space
\[H^*(\sp_4,K_\infty^\circ;\mc{A}_{E,[P_\alpha],\varphi'}(\GSp_4)\otimes E),\]
if nontrivial, is made up of subquotients of the cohomology spaces
\begin{equation}
\label{eqcohlalpha}
H^*(\sp_4,K_\infty^\circ;\mc{L}_{P_\alpha(\A)}^{\GSp_4(\A)}(\tilde\pi'\otimes\psi',s')\otimes E)
\end{equation}
and
\begin{equation}
\label{eqcohindalpha}
H^*(\sp_4,K_\infty^\circ;\Ind_{P_\alpha(\A)}^{\GSp_4(\A)}((\tilde\pi'\boxtimes\psi')\otimes\Sym(\mf{a}_{P_\alpha,0})_{(2s'+1)\rho_{P_\alpha}})\otimes E).
\end{equation}
Just as in the $P_\beta$ case, the nonvanishing of either \eqref{eqcohlalpha} or \eqref{eqcohindalpha} implies that $\tilde\pi'=\tilde\pi_{F'}$ for a cuspidal holomorphic eigenform $F'$ of weight at least $2$, and that any irreducible subquotient of
\[H^*(\sp_4,K_\infty^\circ;\mc{A}_{E,[P_\alpha],\varphi'}(\GSp_4)\otimes E)\]
is nearly equivalent to an irreducible subquotient of
\[\iota_{P_\alpha(\A_f)}^{\GSp_4(\A_f)}((\tilde\pi'\boxtimes\psi')_f,s').\]
Now we use Proposition \ref{propdistsiegkling} to conclude that $\mc{L}_\beta(\tilde\pi_F\boxtimes 1,s)$ cannot also occur as a subquotient, which finishes the proof in the case of $P_\alpha$.\\
\indent \it Case of $B$. \rm Now we let $\varphi'$ be an associate class for $E$ and $[B]$. So $\varphi'$ contains a character of $T(\A)$ of the form
\[(\psi_1'\boxtimes\psi_2'\boxtimes\psi_3')\otimes e^{\langle H_B(\cdot),s_1'\alpha+s_2'\beta\rangle},\]
where $\psi_1',\psi_2',\psi_3'$ are Dirichlet characters and $s_1',s_2'\in\C$. Let us write
\[\psi'=\psi_1'\boxtimes\psi_2'\boxtimes\psi_3'\]
for short.\\
\indent We will study the piece $\mc{A}_{E,[B],\varphi'}(\GSp_4)$ of the Franke--Schwermer decomposition using the (Franke) filtration of Theorem \ref{thmfrfil}. By that theorem, there is a filtration on the space $\mc{A}_{E,[B],\varphi'}(\GSp_4)$ whose graded pieces are parametrized by certain quadruples $(Q,\nu,\Pi,\mu)$. For the convenience of the reader, we recall what these quadruples consist of now:
\begin{itemize}
\item $Q$ is a standard parabolic subgroup of $\GSp_4$;
\item $\nu$ is an element of $(\mf{t}\cap\mf{m}_{Q,0})^\vee$;
\item $\Pi$ is an automorphic representation of $M_Q(\A)$ occurring in
\[L_{\disc}^2(M_Q(\Q)A_Q(\R)^\circ\backslash M_Q(\A))\]
and which is spanned by values at, or residues at, the point $\nu$ of Eisenstein series parabolically induced from $(B\cap M_Q)(\A)$ to $M_Q(\A)$ by representations in $\varphi'$; and
\item $\mu$ is an element of $\mf{a}_{Q,0}^\vee$ whose real part in $\Lie(A_{\GSp_4}(\R)\backslash A_{M_Q}(\R))$ is in the closure of the positive cone, and such that $\nu+\mu$ lies in the Weyl orbit of $\Lambda+\rho$.
\end{itemize}
Then the graded pieces of $\mc{A}_{E,[B],\varphi'}(\GSp_4)$ are isomorphic to direct sums of $\GSp_4(\A_f)\times(\sp_4,K_\infty)$-modules of the form
\[\Ind_{Q(\A)}^{\GSp_4(\A)}(\Pi\otimes\Sym(\mf{a}_{Q,0})_{\mu+\rho_Q})\]
for certain quadruples $(Q,\nu,\Pi,\mu)$ of the form just described.\\
\indent For each of the four possible parabolic subgroups $Q$ and any corresponding quadruple $(Q,\nu,\Pi,\mu)$ as above, we will show using Proposition \ref{propdistsiegkling} that the cohomology
\begin{equation}
\label{eqcohborelgsp4}
H^*(\sp_4,K_\infty^\circ;\Ind_{Q(\A)}^{\GSp_4(\A)}(\Pi\otimes\Sym(\mf{a}_{Q,0})_{\mu+\rho_Q}))
\end{equation}
cannot have $\mc{L}_\beta((\tilde\pi_F\boxtimes 1)_f,s)$ as a subquotient, which will finish the proof.\\
\indent So first assume we have a quadruple $(Q,\nu,\Pi,\mu)$ as above where $Q=B$. Then $\mf{m}_{Q,0}=0$, forcing $\nu=0$. The entry $\Pi$ is the unitarization of a representation in $\varphi'$, and thus must be a character $\psi'$ of $T(\A)$ conjugate to $\psi_1'\boxtimes\psi_2'\boxtimes\psi_3'$. Finally, we have $\mu$ is Weyl conjugate to $\Lambda+\rho$.\\
\indent Therefore the cohomology \eqref{eqcohborelgsp4} is isomorphic, by Theorem \ref{thmcohind}, to a finite sum of copies of
\[\iota_{B(\A_f)}^{\GSp_4(\A_f)}(\psi_f',\mu).\]
By Proposition \ref{propdistsiegkling}, $\mc{L}_\beta((\tilde\pi_F\boxtimes 1)_f,s)$ cannot be a subquotient of this space, and we conclude in the case when $Q=B$.\\
\indent If now we have a quadruple $(Q,\nu,\Pi,\mu)$ where $Q=P_\alpha$, then $\nu$ is an integer multiple of $\alpha/2$ and $\mu$ is a multiple of $(\alpha+2\beta)/2$, and $\nu+\mu$ is conjugate to $\Lambda+\rho$. We find that $\Pi$ is a representation generated by residual Eisenstein series at the point $\nu$ and is therefore a subquotient of the normalized induction
\[\iota_{(B\cap M_\alpha)(\A)}^{M_\alpha(\A)}(\psi',\nu),\]
where $\psi'$ is a character of $T(\A)$ conjugate to $\psi_1'\boxtimes\psi_2'\boxtimes\psi_3'$. Then by \ref{thmcohind} and induction in stages, \eqref{eqcohborelgsp4} is isomorphic to a subquotient of a finite sum of copies of
\[\iota_{B(\A_f)}^{\GSp_4(\A_f)}(\psi_f',\nu+\mu).\]
We then conclude in this case as well using Proposition \ref{propdistsiegkling}.\\
\indent The case when $Q=P_\beta$ is completely similar, and we omit the details. When $Q=G$, it is once again similar, but easier since we do not need to use induction in stages. So we are done with the proof of the $[B]$-Eisenstein multiplicity.\\
\indent Finally, if we instead used $K_\infty$ instead of $K_\infty^\circ$ to compute cohomology, then all the multiplicities that were zero remain zero. The multiplicities that were $1$ remain $1$ because they followed from Proposition \ref{propcohindbeta}, which gets the same answer in both cases. The final claim about the action of the component group of $K_\infty$ follows.
\end{proof}
\subsection{Cuspidal multiplicity of Langlands quotients}
Despite being nontempered quotients of induced representations, some of the Langlands quotients we studied in the previous section can be found in cuspidal cohomology as well as Eisenstein cohomology. The purpose of this section is to explain how this happens.\\
\indent The occurrence of this phenomenon relies on the study of CAP representations, which were first studied in an automorphic context by Piatetski-Shapiro in \cite{PS}. These, by definition, are cuspidal automorphic representations which are nearly equivalent to an irreducible constituent of a parabolically induced representation. In our context, these show up in the proof of the following theorem.
\begin{theorem}
\label{thmcuspmultgsp4}
Let $F$ be a cuspidal holomorphic eigenform of even weight $k\geq 4$, and let $\epsilon$ be the sign of the functional equation for the $L$-function $L(\tilde\pi_F,s)$. Assume $L(\tilde\pi_F,1/2)=0$. Let $E$ be the irreducible representation of $\GSp_4(\C)$ of highest weight $\frac{k-4}{2}(\alpha+\beta)$. Then
\[m_{\cusp}^i(\mc{L}_\beta(\tilde\pi\boxtimes 1,1/6)_f,K_\infty^\circ,E)=\begin{cases}
1&\textrm{if }\epsilon=1\textrm{ and }i=2\textrm{ or }4;\\
2&\textrm{if }\epsilon=-1\textrm{ and }i=3;\\
0&\textrm{otherwise.}
\end{cases}\]
Consequently,
\[m^i(\mc{L}_\beta(\tilde\pi\boxtimes 1,1/6)_f,K_\infty^\circ,E)=\begin{cases}
1&\textrm{if }\epsilon=1\textrm{ and }i=2,3,\textrm{ or }4;\\
3&\textrm{if }\epsilon=-1\textrm{ and }i=3;\\
0&\textrm{otherwise.}
\end{cases}\]
\end{theorem}
\begin{proof}
In \cite{PS}, Piatetski-Shapiro proves that all CAP representations which are nearly equivalent to $\mc{L}_\beta(\tilde\pi\boxtimes 1,1/6)$ come from Saito--Kurokawa forms, and each Saito--Kurokawa form appears with multiplicity one. If $\epsilon=-1$, then the corresponding Saito--Kurokawa representation which, at finite places, is given by $\mc{L}_\beta(\tilde\pi\boxtimes 1,1/6)_f$, is in the (holomorphic) discrete series at infinity with Harish-Chandra parameter $\frac{k-4}{2}(\alpha+\beta)+\rho$. The $(\sp_4,K_\infty^\circ)$-cohomology of the archimedean component of this Saito--Kurokawa representation, with coefficients twisted by $E$, is therefore concentrated in middle degree $3$ and is $2$ dimensional. (See the remarks on discrete representations in Section \ref{secgroupgsp4}; the discrete series representations of $\GSp_4(\R)$ are sums to two such representations of $\Sp_4(\R)$.)\\
\indent If instead $\epsilon=1$, then the Saito--Kurokawa representation in question has archimedean component isomorphic to $\mc{L}_\beta(\tilde\pi\boxtimes 1,1/6)_\infty$. Its cohomology is therefore concentrated in degrees $2$ and $4$, and there it is isomorphic to the $(\mf{sl}_2,\O(2))$-cohomology of $\tilde\pi_{F,\infty}$. (Note $P_\beta(\R)\cap K_\infty^\circ$ contains all of $\O(2)$.) Since $\tilde\pi_{F,\infty}$ is the discrete series representation of $\GL_2(\R)$ of weight $k$, its cohomology is $1$ dimensional. Therefore we have justified the cuspidal multiplicity of $\mc{L}_\beta(\tilde\pi\boxtimes 1,1/6)_f$. The full multiplicity follows from adding the Eisenstein multiplicity computed in Theorem \ref{thmeismultgsp4}.\\
\indent For a nice account of the facts we used about the CAP representations appearing here, see Gan \cite{ganSK}
\end{proof}
We now make several remarks.
\begin{remark}
\label{remurban}
The above theorem corrects a computation made in the paper of Urban \cite{urbanev}, 5.5.3. There he obtains the same result except with the claim that
\[m_{[P_\beta]}^2(\mc{L}_\beta(\tilde\pi\boxtimes 1,1/6)_f,K_\infty^\circ,E)\]
equals $1$ instead of $0$. When we factor in this correction, this shows that the Euler--Poincar\'e multiplicity (equivalent to the alternating sum of our multiplicities $m^i$) discussed there is $1$ when $\epsilon=1$ and is $-3$ when $\epsilon=-1$.\\
\indent But this would seem to throw off the computation in \cite{urbanev} of the cuspidal overconvergent multiplicity of the critical $p$-stabilization of $\mc{L}_\beta(\tilde\pi\boxtimes 1,1/6)_f$. However, when we take into account the fact that $P_\beta(\R)\cap K_\infty^\circ$ contains the maximal compact subgroup $\O(2)$ of the $\GL_2(\R)$ factor of $M_\beta(\R)$, we see that all Eisenstein multiplicities there should be computed via $(\sl_2,\O(2))$-cohomology, instead of $(\sl_2,\SO(2))$-cohomology. Taking this into account makes Urban's overconvergent Eisenstein multiplicities equal to $1$ instead of $2$ when they are nonzero. The cuspidal overconvergent multiplicity is then still equal to $2(\epsilon-1)$, which is what was claimed in \cite{urbanev}.
\end{remark}
\begin{remark}
One could, in principle, use our methods to obtain analogous results as Theorems \ref{thmeismultgsp4} and \ref{thmcuspmultgsp4} in the case of $P_\alpha$ instead of $P_\beta$. To compute the cuspidal multiplicities for $P_\alpha$, one would instead need to use results of Howe--Piatetski-Shapiro \cite{HPS} and Soudry \cite{soudry}.
\end{remark}
\begin{remark}
In the case of $\G_2$, the results that allow us to compute the cuspidal multiplicity for Langlands quotients coming from the short root parabolic are contained in the work on Gan--Gurevich \cite{gangurs}. However, the corresponding results in the case of the long root parabolic are not known. There are partial results in another work of Gan and Gurevich \cite{gangurl}, but it does not give all the results we need. In particular, they say nothing about the CAP representations they obtain at infinity, and so we compute what these representations should be explicitly assuming Arthur's conjectures in Chapter \ref{chartpacket}.
\end{remark}
\section{The case of $\G_2$}
\label{chg2}
In this chapter, we carry out computations analogous to those in the previous chapter for Langlands quotients coming from the long root parabolic in $\G_2$. However, we note that it is not necessary to have read the previous chapter in order to read this one.
\subsection{The group $\G_2$}
\label{secg2}
We define $\G_2$ to be the split simple group over $\Q$ with Dynkin diagram as in Figure \ref{figg2dynkin}.
\begin{figure}[h]
\centering
\includegraphics[scale=.2]{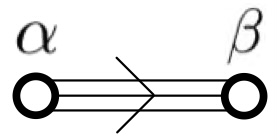}
\caption{The Dynkin diagram of $\G_2$}
\label{figg2dynkin}
\end{figure}
Fixing a maximal $\Q$-split torus $T$ in $\G_2$, we choose a long simple root $\alpha$ and a short simple root $\beta$, as notated in the Dynkin diagram.\\
\indent The group $\G_2$ has trivial center, so unlike $\Sp_4$, there are no central extensions of it which are nontrivial.\\
\indent Also different from $\Sp_4$ is that $\G_2$ does not have such a nice matricial definition. There is a faithful representation of $\G_2$ into $\GL_7$ that we will make some use of, and while it is possible to characterize the image of that representation in terms of the preservation of certain alternating $3$-forms, it is hard to make that characterization explicit in terms of matrices. Consequently, we will study $\G_2$ from the point of view of its root system, which we discuss now.
\subsubsection*{The root lattice}
The root lattice of $\G_2$ looks as in Figure \ref{figg2chamber}. There, the dominant chamber is shaded.
\begin{figure}[h]
\centering
\includegraphics[scale=.25]{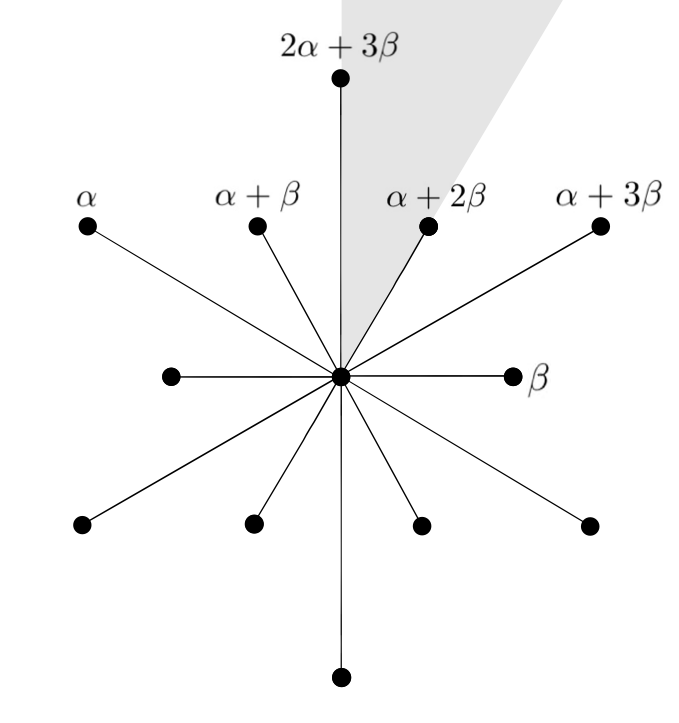}
\caption{The root lattice of $\G_2$}
\label{figg2chamber}
\end{figure}
Write $\Delta$ for the set of roots of $T$ in $\G_2$, and write $\Delta^+$ for the positive ones. So we have
\[\Delta^+=\{\alpha,\beta,\alpha+\beta,\alpha+2\beta,\alpha+3\beta,2\alpha+3\beta\}.\]
\indent One nice feature of $\G_2$ is that the $\Z$-span of the root lattice equals the character group of $T$:
\[X^*(T)=\Z\alpha\oplus\Z\beta.\]
Since the Cartan matrix of $\G_2$ has determinant $1$, an analogous fact holds for the cocharacter group:
\begin{equation}
\label{eqcochart}
X_*(T)=\Z\alpha^\vee\oplus\Z\beta^\vee.
\end{equation}
\subsubsection*{Parabolic subgroups}
Let $B$ denote the standard Borel subgroup of $\G_2$ with respect to our positive system of roots $\Delta^+$. We write $B=TU$ for its Levi decomposition. Besides $B$, there are two other proper standard parabolic subgroups, and they are maximal. Let $P_\alpha$ denote the standard parabolic subgroup whose Levi contains $\alpha$, and write $P_\alpha=M_\alpha N_\alpha$ for its Levi decomposition. Similarly define $P_\beta=M_\beta N_\beta$.\\
\indent The maximal torus $T$ is of course isomorphic to $\GL_1\times\GL_1$. In fact we fix an isomorphism $i_0:\GL_1\times\GL_1\to T$, defined by
\[i_0(t_1,t_2)=\alpha^\vee(t_1)\beta^\vee(t_2).\]
This is indeed an isomorphism by \eqref{eqcochart}.\\
\indent For $\gamma\in\Delta$ a root, write
\[\mathbf{x}_\gamma:\mb{G}_a\to\G_2\]
for the corresponding root group homomorphism, where $\mb{G}_a$ is the additive group scheme.\\
\indent The Levis $M_\alpha$ and $M_\beta$ are both isomorphic to $\GL_2$. We write
\[i_\alpha:\GL_2\to M_\alpha\quad\textrm{and}\quad i_\beta:\GL_2\to M_\beta\]
for the isomorphisms which send the upper triangular matrix $\sm{1&a\\ 0&1}$ in $\GL_2$ to the element $\mathbf{x}_\alpha(a)$ and $\mathbf{x}_\beta(a)$, respectively.\\
\indent We then have the following relations among these isomorphisms:
\[i_\alpha^{-1}(i_0(t_1,t_2))=\pmat{t_1t_2^{-1}& \\ &t_1^{-1}t_2^2},\qquad i_\beta^{-1}(i_0(t_1,t_2))=\pmat{t_2& \\ &t_1t_2^{-1}}.\]
We will often identify $T$ with $\GL_1\times\GL_1$ via $i_0$ and drop the notation from formulas. Similarly we will often identify $M_\alpha$ and $M_\beta$ with $\GL_2$ and drop $i_\alpha$ and $i_\beta$ from notation when it causes no confusion.
\subsubsection*{The standard representation}
The smallest fundamental weight of $\G_2$ is $\alpha+2\beta$, and the representation attached to it is seven dimensional. We denote it by $R_7$ and call it the \it standard representation \rm of $\G_2$; it is the representation one naturally gets when defining $\G_2$ through its action on traceless split octonions.\\
\indent Let $V_7$ be the space of $R_7$. This representation contains weight vectors for the seven weights given by the six short roots together with the zero weight; see Figure \ref{figr7weights}.
\begin{figure}[h]
\centering
\includegraphics[scale=.25]{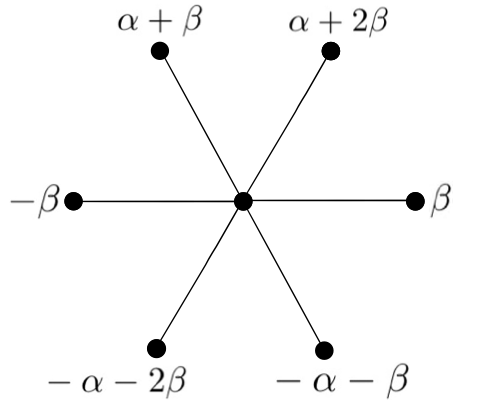}
\caption{The weights of $R_7$}
\label{figr7weights}
\end{figure}
For such a weight $\lambda$, choose a nonzero vector $v_\lambda\in V_7$ corresponding to that weight.\\
\indent Let us order our weight vectors as follows:
\[v_{-\alpha-2\beta},v_{-\alpha-\beta},v_{-\beta},v_0,v_\beta,v_{\alpha+\beta},v_{\alpha+2\beta}.\]
Then using the above list as an ordered basis represents $\G_2$ as $7\times 7$ matrices acting on the linear span of these seven weight vectors. We then have the following matrix representations of the standard Levi subgroups of $\G_2$. For $T$ we have
\begin{equation}
\label{eqr7t}
R_7(i_0(t_1,t_2))=\diag(t_2^{-1},t_1^{-1}t_2,t_1t_2^{-2},1,t_1^{-1}t_2^2,t_1t_2^{-1},t_2),
\end{equation}
and for $M_\alpha$ and $M_\beta$ we have
\begin{equation}
\label{eqr7alpha}
R_7\circ i_\alpha=\pmat{\det^{-1}&&&&\\ &\std^\vee &&&\\ &&1&&\\ \ &&&\std &\\ &&&&\det},
\end{equation}
where $\std$ is the standard representation of $\GL_2$, and
\begin{equation}
\label{eqr7beta}
R_7\circ i_\beta=\pmat{\std^\vee &&\\ &\Ad &\\ &&\std},
\end{equation}
where $\Ad=\Sym^2\otimes\det^{-1}$ is the adjoint representation of $\GL_2$. These can be seen by looking at strings in the directions of $\alpha$ and $\beta$ in the weight diagram as in Figure \ref{figr7levis}.
\begin{figure}[h]
\centering
\includegraphics[scale=.1666]{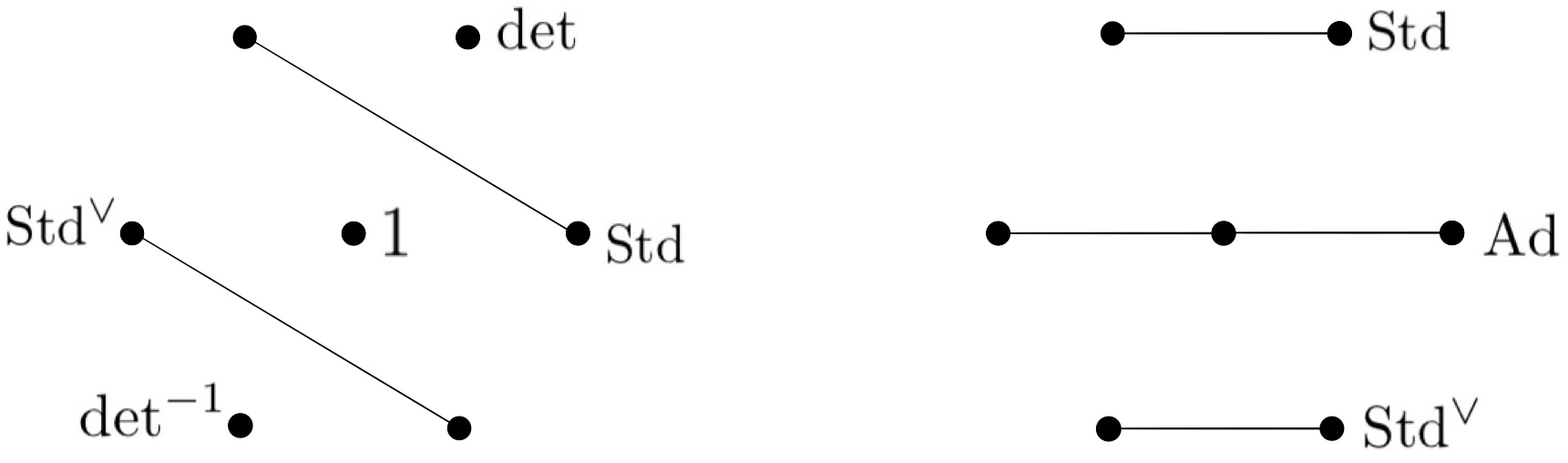}
\caption{The standard Levis of $\G_2$ under $R_7$}
\label{figr7levis}
\end{figure}
\subsubsection*{Duality}
As in the case of $\GSp_4$, the group $\G_2$ is self dual, and identifying $\G_2$ with its dual group switches the long and short simple roots.\\
\indent More explicitly, fix identifications $\GL_2^\vee\cong\GL_2$ and $\G_2\cong\G_2^\vee$ so that positive coroots correspond on the dual side to positive roots. Identify $M_\alpha$ and $M_\beta$ with $\GL_2$ via the maps $i_\alpha$ and $i_\beta$ introduced above. Then $M_\alpha^\vee$ and $M_\beta^\vee$ are identified with $\GL_2^\vee$, and we have commuting diagrams
\begin{equation}
\label{eqphialphag2}
\xymatrix{
\GL_2^\vee\ar[r]^-\sim \ar[d]^\sim & M_\alpha^\vee \ar@{^{(}->}[r] \ar[d]^\sim & \G_2^\vee\ar[d]^\sim \\
\GL_2 \ar[r]^-{i_\beta} & M_\beta \ar@{^{(}->}[r] & \G_2,}
\end{equation}
and
\begin{equation}
\label{eqphibetag2}
\xymatrix{
\GL_2^\vee\ar[r]^-\sim \ar[d]^\sim & M_\beta^\vee \ar@{^{(}->}[r] \ar[d]^\sim & \G_2^\vee\ar[d]^\sim \\
\GL_2 \ar[r]^-{i_\alpha} & M_\alpha \ar@{^{(}->}[r] & \G_2.}
\end{equation}
This is simpler than in the $\GSp_4$ case; the obvious identifications are the correct ones. However, the situation for the Borel is still a little bit complicated. Identifying $\GL_1\cong\GL_1^\vee$ and $T\cong T^\vee$, we have a commuting diagram
\begin{equation}
\label{eqphi0g21}
\xymatrix{
\GL_1^\vee\times\GL_1^\vee\ar[rr]^-\sim \ar[d]^\sim && T^\vee \ar@{^{(}->}[r] \ar[d]^\sim & \G_2^\vee\ar[d]^\sim \\
\GL_1\times\GL_1\ar[r]^-{\varphi_0} & \GL_1\times\GL_1 \ar[r]^-{i_0} & T \ar@{^{(}->}[r] & \G_2,}
\end{equation}
where $\varphi_0$ is given by
\begin{equation}
\label{eqphi0g22}
\varphi_0(t_1,t_2)=(t_1^3t_2^2,t_1^2t_2).
\end{equation}
\subsubsection*{The Weyl group}
Let $W=W(T,\G_2)$ be the Weyl group of $\G_2$. The group $W$ is isomorphic to the dihedral group $D_6$ with $12$ elements acting naturally on the root lattice.\\
\indent For $\gamma\in\Delta$, let $w_\gamma$ be the reflection about the line perpendicular to $\gamma$. Then $W$ is generated by the simple reflections $w_\alpha$ and $w_\beta$. As before, we use the following notation: Write $w_{\alpha\beta}=w_\alpha w_\beta$, $w_{\alpha\beta\alpha}=w_\alpha w_\beta w_\alpha$, and so on. Then
\[W=\{1,w_\alpha,w_\beta,w_{\alpha\beta},w_{\beta\alpha}, w_{\alpha\beta\alpha},w_{\beta\alpha\beta},w_{\alpha\beta\alpha\beta},w_{\beta\alpha\beta\alpha},w_{\alpha\beta\alpha\beta\alpha},w_{\beta\alpha\beta\alpha\beta},w_{-1}\}.\]
The elements above are written minimally in terms of products of the simple reflections $w_\alpha$ and $w_\beta$, except for the final element $-1$. This is the element that acts by negation on the root lattice, and it of length $6$, equal to both $w_{\alpha\beta\alpha\beta\alpha}$ and $w_{\beta\alpha\beta\alpha\beta}$.\\
\indent For $P=MN$ one of the standard parabolic subgroups of $\GSp_4$, we write as usual
\[W^P=\sset{w\in W}{w^{-1}\gamma>0\textrm{ for all positive roots }\gamma\textrm{ in }M}\]
for the set of representatives for the quotient $W(T,M)\backslash W$ of minimal length. Then
\[W^{P_\alpha}=\{1,w_\beta,w_{\beta\alpha},w_{\beta\alpha\beta},w_{\beta\alpha\beta\alpha},w_{\beta\alpha\beta\alpha\beta}\},\qquad W^{P_\beta}=\{1,w_\alpha,w_{\alpha\beta},w_{\alpha\beta\alpha},w_{\alpha\beta\alpha\beta},w_{\alpha\beta\alpha\beta\alpha}\},\]
and $W^B=W$.\\
\indent We note for later use that
\begin{equation}
\label{eqweyltg2}
i_0(t_1,t_2)^{w_\alpha}=i_0(t_1^{-1}t_2^3,t_2),\qquad i_0(t_1,t_2)^{w_\beta}=i_0(t_1,t_1t_2^{-1}).
\end{equation}
\subsubsection*{The group $\G_2(\R)$}
The real Lie group $\G_2(\R)$ is connected and has discrete series (see Section \ref{sectds} for a review of the classification of discrete series, particularly Theorem \ref{thmHC}).\\
\indent Fix a maximal compact torus $T_c$ in $\G_2(\R)$. Then $T_c$ is two dimensional and lies in a maximal compact subgroup of $\G_2(\R)$, which we denote by $K_\infty$. Then $K_\infty$ is connected and $6$ dimensional. In fact
\[K_\infty\cong\SU(2)\times\SU(2)/\mu_2,\]
where $\mu_2=\{\pm 1\}$ is diagonally embedded in $\SU(2)\times\SU(2)$.\\
\indent Let $\mf{t}_c$ be the complexified Lie algebra of $T_c$, and $\mf{k}$ that of $K_\infty$. We abuse notation and write $\Delta=\Delta(\mf{t}_c,\mf{g}_2)$ for the roots of $\mf{t}_c$ in $\mf{g}_2$. Let $\Delta_c=\Delta(\mf{t}_c,\mf{k})$ denote the set of compact roots. There are four roots in $\Delta_c$ consisting of a pair of short roots and a pair of long roots. The short compact roots are orthogonal to the long ones.\\
\indent Again, abusing notation, choose two simple roots $\alpha,\beta$ of $\mf{t}_c$ in $\mf{g}_2$ with $\alpha$ long and $\beta$ short, and choose them so that $\beta$ is compact. Then
\[\Delta_c=\{\pm\beta,\pm(2\alpha+3\beta)\}.\]
\indent The compact Weyl group $W_c=W(\mf{t}_c,\mf{k})$ has four elements and is isomorphic to $(\Z/2\Z)\oplus(\Z/2\Z)$. We in fact have
\[W_c=\{1,w_\beta,w_{\alpha\beta\alpha\beta\alpha},w_{-1}\},\]
and $w_{\alpha\beta\alpha\beta\alpha}$ equals the reflection across the line perpendicular to $2\alpha+3\beta$. It follows that the discrete series representations of $\G_2(\R)$ are parameterized by integral weights in the union of the three chambers between $\beta$ and $2\alpha+3\beta$ which are far enough from the walls of those chambers.
\subsection{Near equivalence and induced representations}
In this section we will study the parabolically induced representations whose Langlands quotients we will try to locate in cohomology later.\\
\indent Let $F$ be a cuspidal holomorphic eigenform, and let $\tilde\pi$ be the unitary automorphic representation of $\GL_2(\A)$ associated with it. We can then view $\tilde\pi$ as an automorphic representation of either $M_\alpha(\A)$ or $M_\beta(\A)$.\\
\indent Let $\delta_{M_\alpha(\A)}$ be the modulus character of $M_\alpha(\A)$, and $\delta_{M_\beta(\A)}$ that of $M_\beta(\A)$. Then for $A\in\GL_2(\A)$, we have
\[\delta_{M_\alpha(\A)}(A)=\vert\det A\vert^5,\qquad\delta_{M_\beta(\A)}(A)=\vert\det A\vert^3.\]
If $s\in\C$, we define the normalized parabolic inductions
\begin{equation}
\label{eqnormindmaxg2}
\iota_{P_\gamma(\A)}^{\G_2(\A)}(\tilde\pi_F,s)=\Ind_{P_\gamma(\A)}^{\G_2(\A)}(\tilde\pi_F\otimes\delta_{P_\gamma}^{s+1/2}),\qquad\gamma\in\{\alpha,\beta\}.
\end{equation}
We then have the following analogue of Proposition \ref{propequalmaxpara}.
\begin{proposition}
\label{propequalmaxparag2}
Let $\gamma\in\{\alpha,\beta\}$ be one of the simple roots of $\G_2$. Let $F,F'$ be cuspidal holomorphic eigenforms, and let $s,s'\in\R_{>0}$. If there are irreducible subquotients
\[\Pi\textrm{ of }\iota_{P_\gamma(\A)}^{\G_2(\A)}(\tilde\pi_F,s)\]
and
\[\Pi'\textrm{ of }\iota_{P_\gamma(\A)}^{\G_2(\A)}(\tilde\pi_{F'},s')\]
such that $\Pi$ and $\Pi'$ are nearly equivalent, then $\tilde\pi_F=\tilde\pi_{F'}$ and $s=s'$.
\end{proposition}
\begin{proof}
We prove this for the short root parabolic $P_\beta$ since the proof in the case of $P_\alpha$ is completely analogous.\\
\indent Let $p$ be a prime where the local components $\Pi_p$ and $\Pi_p'$ are unramified and isomorphic. Then $\tilde\pi_{F,p}$ and $\tilde\pi_{F',p}$ are unramified.\\
\indent Write $T_2$ for the standard diagonal torus of $\GL_2$ and $B_2$ for the standard upper triangular Borel in $\GL_2$. Let $\delta_{B_2(\Q_p)}$ be the usual modulus character of $B_2(\Q_p)$. Then by the results recalled in Section \ref{secsatgal}, there are characters $\chi_1,\chi_2,\chi_1',\chi_2'$ of $\Q_p^\times$ such that $\tilde\pi_{F,p}$ is the unramified subquotient of
\[\Ind_{B_2(\Q_p)}^{\GL_2(\Q_p)}((\chi_1\boxtimes\chi_2)\otimes\delta_{B_2(\Q_p)}^{1/2}),\]
and $\tilde\pi_{F',p}$ is the unramified subquotient of
\[\Ind_{B_2(\Q_p)}^{\GL_2(\Q_p)}((\chi_1'\boxtimes\chi_2')\otimes\delta_{B_2(\Q_p)}^{1/2}).\]
Here, $\chi_1\boxtimes\chi_2$ is the character of $T_2$ which evaluated at $\diag(x,y)\in T_2(\Q_p)$ gives the product $\chi_1(x)\chi_2(y)$, and similarly for $\chi_1'\boxtimes\chi_2'$. By temperedness, the characters $\chi_1,\chi_2,\chi_1',\chi_2$ are unitary.\\
\indent By induction in stages, $\Pi$ is the unramified subquotient of
\[\Ind_{B(\Q_p)}^{\G_2(\Q_p)}(\chi\delta_{P_\beta(\Q_p)}^s\delta_{B(\Q_p)}^{1/2}),\]
where $\chi$ is the character of $T$ given by $\chi=(\chi_1\boxtimes\chi_2)\circ i_\beta\circ i_0^{-1}$ (see the subsection on parabolic subgroups in Section \ref{secg2}) and similarly $\Pi_p'$ is the unramified subquotient of
\[\Ind_{B(\Q_p)}^{\G_2(\Q_p)}(\chi'\delta_{P_\beta(\Q_p)}^{s'}\delta_{B(\Q_p)}^{1/2}),\]
where $\chi'=(\chi_1'\boxtimes\chi_2')\circ i_\beta\circ i_0^{-1}$. The characters $\chi$ and $\chi'$ are unitary. Since $\Pi\cong\Pi'$, the characters 
\[\chi\delta_{P_\beta(\Q_p)}^s\quad\textrm{and}\quad \chi'\delta_{P_\beta(\Q_p)}^{s'}\]
are equal up to the Weyl group $W$; there is a $w\in W$ such that for all $x,y\in\Q_p^\times$, we have
\begin{equation}
\label{eqchidelta}
\chi\delta_{P_\beta(\Q_p)}^s(i_0(x,y)^w)=\chi'\delta_{P_\beta(\Q_p)}^{s'}(i_0(x,y)).
\end{equation}
\indent Now let $t\in\Q_p^\times$ and let
\[T=i_0(t^2,t).\]
Then we compute, using \eqref{eqweyltg2}, that
\begin{align*}
T=T^{w_\beta}&=i_0(t^2,t),\\
T^{w_\alpha}=T^{w_{\alpha\beta}}&=i_0(t,t),\\
T^{w_{\beta\alpha}}=T^{w_{\beta\alpha\beta}}&=i_0(t,1),\\
T^{w_{\alpha\beta\alpha}}=T^{w_{\alpha\beta\alpha\beta}}&=i_0(t^{-1},1),\\
T^{w_{\beta\alpha\beta\alpha}}=T^{w_{\beta\alpha\beta\alpha\beta}}&=i_0(t^{-1},t^{-1}),\\
T^{w_{\alpha\beta\alpha\beta\alpha}}=T^{w_{-1}}&=i_0(t^{-2},t^{-1}).
\end{align*}
Since $\det(i_\beta^{-1}(i_0(x,y)))=x$, the above gives
\[\vert\chi\delta_{P_\beta(\Q_p)}^s(T^w)\vert=\begin{cases}
p^{6s}&\textrm{if }w\in\{1,w_\beta\};\\
p^{3s}&\textrm{if }w\in\{w_\alpha,w_{\alpha\beta},w_{\beta\alpha},w_{\beta\alpha\beta}\};\\
p^{-3s}&\textrm{if }w\in\{w_{\alpha\beta\alpha},w_{\alpha\beta\alpha\beta}, w_{\beta\alpha\beta\alpha},w_{\beta\alpha\beta\alpha\beta}\};\\
p^{-6s}&\textrm{if }w\in\{w_{\alpha\beta\alpha\beta\alpha},w_{-1}\}.
\end{cases}\]
Comparing this to
\[\vert\chi'\delta_{P_\beta(\Q_p)}^{s'}(T)\vert=\vert t\vert^{6s'}\]
via \eqref{eqchidelta} gives, since $s,s'>0$,
\[s=s'\textrm{ and }w\in\{1,w_\beta\},\quad\textrm{or}\quad s=2s'\textrm{ and }w\in\{w_\alpha,w_{\alpha\beta},w_{\beta\alpha},w_{\beta\alpha\beta}\}.\]
But in this latter case we can then repeat the calculation with $T=i_0(t,1)$ instead. In this case we then find
\begin{align*}
T^{w_\alpha}=T^{w_{\alpha\beta}}&=(t^{-1},1),\\
T^{w_{\beta\alpha}}=T^{w_{\beta\alpha\beta}}&=(t^{-1},t^{-1}),
\end{align*}
and the same argument then rules out $w\in\{w_\alpha,w_{\alpha\beta},w_{\beta\alpha},w_{\beta\alpha\beta}\}$.\\
\indent Therefore $w\in\{1,w_\beta\}$ and $s=s'$. Then \eqref{eqchidelta} implies $\chi_1=\chi_1'$ and $\chi_2=\chi_2'$ if $w=1$, or $\chi_1=\chi_2'$ and $\chi_2=\chi_1'$ if $w=w_\beta$. In either case we have
\[\Ind_{B_2(\Q_p)}^{\GL_2(\Q_p)}((\chi_1\boxtimes\chi_2)\otimes\delta_{B_2(\Q_p)}^{1/2})\textrm{ and }\Ind_{B_2(\Q_p)}^{\GL_2(\Q_p)}((\chi_1'\boxtimes\chi_2')\otimes\delta_{B_2(\Q_p)}^{1/2})\]
have the same unramified subquotients, which means $\tilde\pi_{F,p}\cong\tilde\pi_{F',p}$.\\
\indent Now letting $p$ vary over all unramified primes for which $\Pi_p\cong\Pi_p'$ and applying strong multiplicity one for $\GL_2$ finishes the proof.
\end{proof}
Let $\psi_1,\psi_2$ be Dirichlet characters, and consider the character $\psi_1\boxtimes\psi_2$ of $T(\A)$ given by
\[(\psi_1\boxtimes\psi_2)(i_0(t_1,t_2))=\psi_1(t_1)\psi_2(t_2).\]
Let $\delta_{B(\A)}$ be the modulus character of $B(\A)$. We have
\[\delta_{B(\A)}^{1/2}=e^{\langle H_B(\cdot),\rho\rangle},\]
where $\rho=3\alpha+5\beta$ is half the sum of positive roots. If $s_1,s_2\in\C$, write
\begin{equation}
\label{eqnormindming2}
\iota_{B(\A)}^{\G_2(\A)}(\psi_1\boxtimes\psi_2;s_1,s_2)=\Ind_{B(\A)}^{\G_2(\A)}((\psi_1\boxtimes\psi_2)\otimes e^{\langle H_B(\cdot),s_1\alpha+s_2\beta+\rho\rangle})
\end{equation}
for the normalized induction.\\
\indent For the following we fix any prime $\ell$ and identify $\C$ and $\overline\Q_\ell$ via a fixed isomorphism.
\begin{proposition}
\label{propg2galrepB}
Let $\psi_1,\psi_2$ be Dirichlet characters, and let $m_1,m_2\in\Z$. Let $\Pi$ be an irreducible subquotient of
\[\iota_{B(\A)}^{\G_2(\A)}(\psi_1\boxtimes\psi_2;m_1,m_2).\]
Let $j_{T,G_2}$ be the inclusion of $T$ into $\G_2$. Then $\Pi$ has attached to it the Galois representation $G_\Q\to\G_2(\overline\Q_\ell)$ given by
\[j_{T,G_2}\circ i_0\circ\left((\chi_{\cyc}^{m_2}\psi_1^3\psi_2^2)\times(\chi_{\cyc}^{m_1}\psi_1^2\psi_2)\right),\]
where we have viewed $\psi_1,\psi_2$ as Galois characters via class field theory.
\end{proposition}
\begin{proof}
Let $p$ be a prime different from $\ell$ which is unramified for $\Pi$, and hence which not divide the conductors of the $\psi_i$'s. Let $\lambda_i=\psi_i(p)$ for $i=1,2$. Then the character
\begin{equation}
\label{eqpfg2charB}
(\psi_1\boxtimes\psi_2)\otimes e^{\langle H_B(\cdot),m_1\alpha+m_2\beta\rangle}
\end{equation}
of $\GL_1(\A)^2$ has Satake parameter at $p$
\[(p^{-(2m_1-m_2)}\lambda_1,p^{-(2m_2-3m_1)}\lambda_2)\in\GL_1(\overline{\Q}_\ell)^2.\]
\indent Identifying $(\GL_1)^2$ with $T$ on the dual side via the map $\varphi_0$ of \eqref{eqphi0g21} and \eqref{eqphi0g22} gives that the character \eqref{eqpfg2charB} has attached to it the Galois representation
\[i_0\circ\left((\chi_{\cyc}^{m_2}\psi_1^3\psi_2^2)\times(\chi_{\cyc}^{m_1}\psi_1^2\psi_2)\right).\]
Then we appeal to Proposition \ref{propindgalrep} to finish the proof.
\end{proof}
\begin{proposition}
\label{propg2galrepalpha}
Let $F$ be a holomorphic cuspidal eigenform of weight $k$ and let $m\in\Z$ with $m\equiv k-1\modulo{2}$. Let $\Pi$ be any irreducible subquotient of
\[\iota_{P_\alpha(\A)}^{\G_2(\A)}(\tilde\pi_F,m/10).\]
Let $j_{M_\beta,\G_2}$ be the inclusion $M_\beta\hookrightarrow\G_2$. Then $\Pi$ has attached to it the Galois representation $G_\Q\to\G_2(\overline\Q_\ell)$ given by
\[j_{M_\beta,\G_2}\circ i_\beta\circ(\rho_F\otimes\chi_{\cyc}^{(m-k+1)/2}),\]
where $\rho_F$ is the Galois representation attached to $F$ by Eichler--Shimura, Deligne, and Deligne--Serre (Theorem \ref{thmesdds}).
\end{proposition}
\begin{proof}
Let $p$ be a prime different from $\ell$ which is unramified for $\Pi$, and hence which is unramified for $\tilde\pi_F$. Let $\diag(\lambda_1,\lambda_2)\in\GL_2(\overline\Q_\ell)$ be a diagonal representative of the Satake parameter of $\tilde\pi_F$ at $p$. Then
\[\tilde\pi_F\otimes\delta_{P_\alpha(\A)}^{m/10}\]
has Satake parameter at $p$ represented by
\[p^{-m/2}\diag(\lambda_1,\lambda_2)\in\GL_2(\overline\Q_\ell),\]
because $\delta_{P_\alpha(\A)}$ acts as $\vert\det\vert^5$. Now we use the commutativity of \eqref{eqphialphag2} and Proposition \ref{propindgalrep} to conclude.
\end{proof}
\begin{proposition}
\label{propg2galrepbeta}
Let $F$ be a holomorphic cuspidal eigenform of weight $k$ and let $m\in\Z$ with $m\equiv k-1\modulo{2}$. Let $\Pi$ be any irreducible subquotient of
\[\iota_{P_\beta(\A)}^{\G_2(\A)}(\tilde\pi_F,m/6).\]
Let $j_{M_\alpha,\G_2}$ be the inclusion $M_\alpha\hookrightarrow\G_2$. Then $\Pi$ has attached to it the Galois representation $G_\Q\to\G_2(\overline\Q_\ell)$ given by
\[j_{M_\alpha,\G_2}\circ i_\alpha\circ(\rho_F\otimes\chi_{\cyc}^{(m-k+1)/2}),\]
where $\rho_F$ is the Galois representation attached to $F$ by Eichler--Shimura, Deligne, and Deligne--Serre (Theorem \ref{thmesdds}).
\end{proposition}
\begin{proof}
The proof is completely similar to that of \ref{propg2galrepalpha} above; switch $\alpha$ and $\beta$ and appeal to \eqref{eqphibetag2} instead of \eqref{eqphialphag2}.
\end{proof}
\indent Before giving the analogue of Proposition \ref{propdistsiegkling}, we need to prove a lemma about Galois representations attached to modular forms. The analogue of this lemma in the $\GSp_4$ case was not necessary because of the nice shape of the Levis of the standard parabolic subgroups in the standard representation of $\GSp_4$. Here in the $\G_2$ case, however, the blocks of $M_\beta$ in the standard representation $R_7$ include a symmetric square representation of $\GL_2$, and we will need the following lemma to distinguish representations factoring through it and those factoring through $M_\alpha$.
\begin{lemma}
\label{lemsumof3}
Let $F$ be a holomorphic cuspidal eigenform of weight $k\geq 2$, and let $\rho_F$ be its Galois representation into $\GL_2(\overline\Q_\ell)$ (Theorem \ref{thmesdds}). Then $\Sym^2\rho_F$ is either irreducible, or is the direct sum of two irreducible representations.
\end{lemma}
\begin{proof}
We separate the proof into two cases, first when $F$ does not have CM and second when it does.\\
\indent Assume $F$ does not have CM. By results of Momose \cite{momose} (See also \cite{loeimage}) we know then that the image of $\rho_F$ in $\GL_2(\overline\Q_\ell)$ can be conjugated to be either:
\begin{itemize}
\item an open subgroup of $\GL_2(\Z_\ell)$, or
\item an open subgroup of $B^\times$, where $B$ is a certain quaternion algebra over $\Q_\ell$.
\end{itemize}
In either case the image of $\rho_F$ is large enough for $\Sym^2\rho_F$ to be irreducible.\\
\indent Now assume $F$ has CM by an imaginary quadratic field $K$. Then $\rho_F$ is the induction
\[\rho_F\cong\Ind_{G_K}^{G_\Q}(\chi),\]
where $\chi$ is a Hecke character of $G_K$. Thus, if $c\in G_\Q$ is a complex conjugation, then writing $V$ for the space of $\rho_F$, there are linearly independent vectors $u,v\in V$ such that
\[gu=\chi(g)u,\qquad gv=\chi(cgc)v,\qquad\textrm{for }g\in G_K,\]
and
\[cu=v,\qquad cv=u.\]
Let us write $\chi_c$ for the character given by
\[\chi_c(g)=\chi(cgc)\]
for $g\in G_\Q$. If $v_1,v_2\in V$, write $v_1\otimes v_2=v_2\otimes v_1$ for the corresponding product in $\Sym^2(V)$.\\
\indent Now we have
\[c(u\otimes v)=u\otimes v,\]
and
\[g(u\otimes v)=(\chi\chi_c)(g)u\otimes v,\]
for $g\in G_K$. So the space spanned by $u\otimes v$ is invariant and gives the character $\chi'$ of $G_\Q$ which is given by $\chi\chi_c$ on $G_K$ and is trivial on $c$. Also,
\[c(u\otimes u)=v\otimes v,\qquad c(v\otimes v)=u\otimes u,\]
and
\[g(u\otimes u)=\chi^2(g)u\otimes u,\qquad g(v\otimes v)=\chi_c^2(g)v\otimes v,\]
for $g\in G_K$. Therefore the space spanned by $u\otimes u$ and $v\otimes v$ is also invariant, and $G_\Q$ acts on it as $\Ind_{G_K}^{G_\Q}(\chi^2)$. Thus
\[\Sym^2\rho_F\cong\chi'\oplus\Ind_{G_K}^{G_\Q}(\chi^2).\]
It now suffices to prove that $\Ind_{G_K}^{G_\Q}(\chi^2)$ is irreducible.\\
\indent To this end, we first note that
\[\Sym^2\rho_F|_{G_K}=\chi^2\oplus\chi_c^2\oplus\chi\chi_c.\]
Since $\Sym^2\rho_F$ is Hodge--Tate with Hodge--Tate weights $0$, $k-1$, and $2k-2$, it follows that either $\chi^2$ or $\chi_c^2$ is finite order, and the other is a finite order character times $\chi_{\cyc}^{2k-2}|_{G_K}$. Therefore $\chi^2$ and $\chi_c^2$ are distinct, because the evaluation of either character on a Frobenius element $\Frob_p$ in $G_K$ gives $p$-Weil numbers of different weights (since $k>1$) and by Chebotarev, there are infinitely many such Frobenius elements in $G_K$.\\
\indent Now assume that the space of $\Ind_{G_K}^{G_\Q}(\chi^2)$, spanned by $u\otimes u$ and $v\otimes v$, has an invariant vector
\[a(u\otimes u)+b(v\otimes v)\]
for some scalars $a,b$. We will show that this implies $a=b=0$, which will prove that $\Ind_{G_K}^{G_\Q}(\chi^2)$ is irreducible. Choose $g\in G_K$ with $\chi^2(g)\ne\chi_c^2(g)$. Then we have
\[g(a(u\otimes u)+b(v\otimes v))=a\chi^2(g)(u\otimes u)+b\chi_c^2(g)(v\otimes v),\]
which cannot be in the span of $a(u\otimes u)+b(v\otimes v)$ unless $a=0$ or $b=0$. Since $c$ switches $u\otimes u$ and $v\otimes v$, we must have both $a=0$ and $b=0$, which finishes the proof.
\end{proof}
\begin{remark}
In our applications, we actually only need this lemma for one single $\ell$, but it was essentially no harder to write down the proof for all $\ell$.
\end{remark}
\begin{remark}
\indent We thank Shuai Wang for bringing the following to our attention. There are examples of irreducible, two dimensional representations of finite groups whose symmetric squares do actually decompose as sums of three characters. It seems they tend to come from certain representations of dihedral groups of order divisible by $8$, though they can also come from other groups of order divisible by $8$ as well.\\
\indent Therefore we cannot hope to get by on the irreducibility of $\rho_F$ alone in proving the above lemma. Also, this shows that the hypothesis that weight $k\geq 2$ is essential, otherwise $\rho_F$ may factor through one of the aforementioned dihedral representations (for example if $\rho_F$ has image precisely $D_4$).
\end{remark}
\begin{proposition}
\label{propdistalphabetag2}
Let $F_\alpha,F_\beta$ be two holomorphic cuspidal eigenforms of weights $k_\alpha$ and $k_\beta$, respectively, and assume $k_\beta\geq 2$. Let $\psi_1$ and $\psi_2$ be Dirichlet characters, and let $m_\alpha,m_\beta,m_1,m_2\in\Z$. Assume that $m_\alpha\equiv k_\alpha-1\modulo{2}$ and $m_\beta\equiv k_\beta-1\modulo{2}$. Then given any irreducible subquotients
\[\Pi_\alpha\textrm{ of }\iota_{P_\alpha(\A)}^{\G_2(\A)}(\tilde\pi_{F_\alpha},m_\alpha/10)\]
and
\[\Pi_\beta\textrm{ of }\iota_{P_\beta(\A)}^{\G_2(\A)}(\tilde\pi_{F_\beta},m_\beta/6)\]
and
\[\Pi_0\textrm{ of }\iota_{B(\A)}^{\G_2(\A)}(\psi_1\boxtimes\psi_2;m_1,m_2),\]
we have that no two of $\Pi_\alpha$, $\Pi_\beta$, and $\Pi_0$ are nearly equivalent.
\end{proposition}
\begin{proof}
Let $\rho_\alpha$, $\rho_\beta$, and $\rho_0$ be, respectively, the Galois representations attached to $\Pi_\alpha$, $\Pi_\beta$, and $\Pi_0$ by Propositions \ref{propg2galrepalpha}, \ref{propg2galrepbeta}, and \ref{propg2galrepB}. We compose these with the standard representation $R_7$ and obtain, using \eqref{eqr7beta}, \eqref{eqr7alpha}, and \eqref{eqr7t},
\[R_7\circ\rho_\alpha=(\rho_{F_\alpha}\otimes\chi_{\cyc}^{(m_\alpha-k_\alpha+1)/2})\oplus(\rho_{F_\alpha}\otimes\chi_{\cyc}^{(m_\alpha-k_\alpha+1)/2})^\vee\oplus\Ad(\rho_{F_\alpha}\otimes\chi_{\cyc}^{(m_\alpha-k_\alpha+1)/2}),\]
\[R_7\circ\rho_\beta=1\oplus(\omega_{F_\beta}\chi_{\cyc}^{m_\beta})\oplus(\omega_{F_\beta}^{-1}\chi_{\cyc}^{-m_\beta})\oplus (\rho_{F_\beta}\otimes\chi_{\cyc}^{(m_\beta-k_\beta+1)/2})\oplus(\rho_{F_\beta}\otimes\chi_{\cyc}^{(m_\beta-k_\beta+1)/2})^\vee,\]
\begin{multline*}
R_7\circ\rho_0=1\oplus(\chi_{\cyc}^{m_1}\psi_1^2\psi_2) \oplus(\chi_{\cyc}^{-m_1}\psi_1^{-2}\psi_2^{-1}) \oplus(\psi_1\psi_2\chi_{\cyc}^{m_2-m_1}) \oplus(\psi_1^{-1}\psi_2^{-1}\chi_{\cyc}^{m_1-m_2})\\
\oplus(\psi_1\chi_{\cyc}^{2m_1-m_2}) \oplus(\psi_1^{-1}\chi_{\cyc}^{m_2-2m_1}).
\end{multline*}
Here $\omega_{F_\beta}$ is the nebentypus of $F_\beta$. We see that the first of these representations is either the sum of $3$ or $4$ irreducible representations by Lemma \ref{lemsumof3}, that the second is the sum of $5$ irreducible representations, and the last is the sum of $7$. Therefore we are done by invoking Proposition \ref{propdistgalrep}.
\end{proof}
\subsection{Eisenstein multiplicity of Langlands quotients}
\label{seceismultg2}
We compute in this section the Eisenstein multiplicity of Langlands quotients coming from the long root parabolic $P_\alpha$. What follows will be highly analogous to the content of Section \ref{seceismultgsp4} where we computed the Eisenstein multiplicity of Langlands quotients coming from the Siegel (short root) parabolic of $\GSp_4$. It is interesting to note that the roles of the long and short root parabolics switch when passing from $\GSp_4$ to $\G_2$.\\
\indent For standard parabolics $P$ in $\G_2$, we will make use of the normalized parabolic induction functors $\iota_{P(\A)}^{\G_2(\A)}$ defined in \eqref{eqnormindmaxg2} and \eqref{eqnormindming2}, and their similarly defined finite adelic analogues $\iota_{P(\A_f)}^{\G_2(\A_f)}$.
\begin{proposition}
\label{propcohindalpha}
Let $E$ be an irreducible, finite dimensional representation of $\G_2(\C)$, and say that $E$ has highest weight $\Lambda$. Write
\[\Lambda=c_1(2\alpha+3\beta)+c_2(\alpha+2\beta)\]
with $c_1,c_2\in\Z_{\geq 0}$. Let $F$ be a holomorphic cuspidal eigenform of weight $k$ and trivial nebentypus, and let $s\in\C$ with $\re(s)\geq 0$. Assume
\[H^i(\mf{g}_2,K_\infty;\Ind_{P_\alpha(\A)}^{\G_2(\A)}(\tilde\pi_F\otimes\Sym(\mf{a}_{P_\alpha,0})_{(2s+1)\rho_{P_\alpha}})\otimes E)\ne 0.\]
Then either:
\begin{enumerate}[label=(\roman*)]
\item We have
\[i=4,\qquad k=2c_1+c_2+4,\qquad s=\frac{c_2+1}{10},\]
and
\begin{multline*}
H^4(\mf{g}_2,K_\infty;\Ind_{P_\alpha(\A)}^{\G_2(\A)}(\tilde\pi_F\otimes\Sym(\mf{a}_{P_\alpha,0})_{(2s+1)\rho_{P_\alpha}})\otimes E)
\cong\iota_{P_\alpha(\A_f)}^{\G_2(\A_f)}(\tilde\pi_{F,f},(c_2+1)/10),
\end{multline*}
or,
\item We have
\[i=5,\qquad k=c_1+c_2+3,\qquad s=\frac{3c_1+c_2+4}{10},\]
and
\begin{multline*}
H^5(\mf{g}_2,K_\infty;\Ind_{P_\alpha(\A)}^{\G_2(\A)}(\tilde\pi_F\otimes\Sym(\mf{a}_{P_\alpha,0})_{(2s+1)\rho_{P_\alpha}})\otimes E)\\
\cong\iota_{P_\alpha(\A_f)}^{\G_2(\A_f)}(\tilde\pi_{F,f},(3c_1+c_2+4)/10),
\end{multline*}
or,
\item We have
\[i=6,\qquad k=c_1+2,\qquad s=\frac{3c_1+2c_2+5}{10},\]
and
\begin{multline*}
H^6(\mf{g}_2,K_\infty;\Ind_{P_\alpha(\A)}^{\G_2(\A)}(\tilde\pi_F\otimes\Sym(\mf{a}_{P_\alpha,0})_{(2s+1)\rho_{P_\alpha}})\otimes E)\\
\cong\iota_{P_\alpha(\A_f)}^{\G_2(\A_f)}(\tilde\pi_{F,f},(3c_1+2c_2+5)/10).
\end{multline*}
\end{enumerate}
\end{proposition}
\begin{proof}
Let $\mf{t}$ be the complexified Lie algebra of $T$. Note that we have a decomposition
\[\mf{t}=(\mf{m}_{\alpha,0}\cap\mf{t})\oplus\mf{a}_{P_\alpha,0},\]
and $(\alpha+2\beta)$ acts by zero on the first component while $\alpha$ acts by zero on the second. We also have
\[W^{P_\alpha}=\{1,w_\beta,w_{\beta\alpha},w_{\beta\alpha\beta},w_{\beta\alpha\beta\alpha},w_{\beta\alpha\beta\alpha\beta}\},\]
and one computes
\begin{align*}
-(\Lambda+\rho)&=-(c_1+1)\frac{\beta}{2}-(3c_1+2c_2+5)\frac{\alpha+2\beta}{2},\\
-w_\beta(\Lambda+\rho)&=-(c_1+c_2+2)\frac{\beta}{2}-(3c_1+c_2+4)\frac{\alpha+2\beta}{2},\\
-w_{\beta\alpha}(\Lambda+\rho)&=-(2c_1+c_2+3)\frac{\beta}{2}-(c_2+1)\frac{\alpha+2\beta}{2},\\
-w_{\beta\alpha\beta}(\Lambda+\rho)&=-(2c_1+c_2+3)\frac{\beta}{2}+(c_2+1)\frac{\alpha+2\beta}{2},\\
-w_{\beta\alpha\beta\alpha}(\Lambda+\rho)&=-(c_1+c_2+2)\frac{\beta}{2}+(3c_1+c_2+4)\frac{\alpha+2\beta}{2},\\
-w_{\beta\alpha\beta\alpha\beta}(\Lambda+\rho)&=-(c_1+1)\frac{\beta}{2}+(3c_1+2c_2+5)\frac{\alpha+2\beta}{2}.
\end{align*}
\indent Now by Theorem \ref{thmcohind}, in order for our cohomology space to be nontrivial, there needs to be a $w\in W^{P_\alpha}$ with
\[-w(\Lambda+\rho)|_{\mf{a}_{P_\alpha,0}}=2s\rho_{P_\alpha}=10s\frac{\alpha+2\beta}{2},\]
and
\[-w(\Lambda+\rho)|_{\mf{m}_{\alpha,0}}=\pm(k-1)\frac{\alpha}{2}.\]
Therefore, since $\re(s)\geq 0$, we see from the formulas for each $-w(\Lambda+\rho)|_{\mf{a}_{P_\beta},0}$ that $w$ can only equal $w_{\beta\alpha\beta}$, $w_{\beta\alpha\beta\alpha}$, or $w_{\beta\alpha\beta\alpha\beta}$.\\
\indent In the case that $w=w_{\beta\alpha\beta}$, we get that
\[k-1=+(2c_1+c_2+3),\]
with this choice of sign because $k-1\geq 0$, and
\[6s=c_2+1.\]
We also have that the length $\ell(w_{\beta\alpha\beta})$ of $w_{\beta\alpha\beta}$ is $3$. Finally, since
\[\rho=\frac{\alpha}{2}+5\frac{\alpha+2\beta}{2},\]
we have
\[(w_{\beta\alpha\beta}(\Lambda+\rho)-\rho)|_{\mf{m}_{\alpha,0}}= (2c_1+c_2+2)\frac{\alpha}{2}=(k-2)\frac{\alpha}{2}.\]
Therefore, the isomorphism of Theorem \ref{thmcohind} in our case is
\begin{multline*}
H^i(\mf{g}_2,K_\infty;\Ind_{P_\alpha(\A)}^{\G_2(\A)}(\tilde\pi_F\otimes\Sym(\mf{a}_{P_\alpha,0})_{(2s+1)\rho_{P_\alpha}})\otimes E)\\
\cong\iota_{P_\alpha(\A_f)}^{\G_2(\A_f)}(\tilde\pi_{F,f},(c_2+1)/10)\otimes H^{i-3}(\mf{m}_{\alpha,0},K_\infty\cap P_\alpha(\R);\tilde\pi_{F,\infty}\otimes F_{k-2}),
\end{multline*}
where $F_{k-2}$ is the representation of $\mf{m}_{\alpha,0}$ of highest weight $(k-2)(\alpha/2)$.\\
\indent Now, since $k-1=2c_1+c_2+3>0$, the representation $\tilde\pi_{F,\infty}$ is the discrete series representation of $\GL_2(\R)$ of weight $k$, and therefore has nontrivial cohomology when tensored with $F_{k-2}$ in degree $1$ and degree $1$ only. Since $K_\infty\cap\GL_2(\R)$ is a maximal compact subgroup of $\GL_2(\R)$, the cohomology of $\tilde\pi_{F,\infty}$ in degree $1$ is $1$ dimensional. The claim (i) of our proposition is now immediate.\\
\indent The claims (ii) and (iii) are completely similar, using instead the length $4$ element $w_{\beta\alpha\beta\alpha}$ and the length $5$ element $w_{\beta\alpha\beta\alpha\beta}$, respectively; we omit the details.
\end{proof}
We now prove an analogue of Lemma \ref{lemholoESgsp4} in our context.
\begin{lemma}
\label{lemholoESg2}
Let $F$ be a holomorphic cuspidal eigenform of weight $k\geq 2$ and trivial nebentypus. For any flat section $\phi_s\in\iota_{P_\alpha(\A)}^{\G_2(\A)}(\tilde\pi_F,s)$, the Eisenstein series $E(\phi,2s\rho_{P_\alpha})$ does not have a pole for $\re(s)>0$ except perhaps if $s=1/10$. If furthermore
\[L(1/2,\tilde\pi_F,\Sym^3)=0,\]
then $E(\phi,2s\rho_{P_\alpha})$ is also holomorphic at $s=1/10$.
\end{lemma}
\begin{proof}
This is an easy consequence of what is done in the paper of \v{Z}ampera \cite{zampera}, but let us quickly explain how this is proved, since we have set up the tools to do so already.\\
\indent It suffices to prove the lemma for $\phi=\bigotimes_v\phi_v$ decomposable into local sections. Write $E(\phi,s)= E(\phi,2s\rho_{P_\beta})$. By Theorem \ref{thmconstterm}, the constant term of $E(\phi,s)$ along $P_\beta$ (and hence along $B$) is zero, and the constant term along $P_\alpha$ is
\[E_{P_\alpha}(\phi,s)=\phi_s+M(\phi,w_{\beta\alpha\beta\alpha\beta})_{-2s\rho_{P_\beta}}.\]
Then we apply Theorem \ref{thmLSmethod}; in our current setting the root $\gamma$ of that theorem is $\alpha$, and $\tilde\beta=\rho_{P_\alpha}/5$, and adjusting for this gives
\[M(\phi,w_{\beta\alpha\beta\alpha\beta})_{-2s\rho_{P_\alpha}}=\prod_{j=1}^m\frac{L^S(5js,\tilde\pi_F,R_i^\vee)}{L^S(5js+1,\tilde\pi_F,R_i^\vee)}\bigotimes_{v\notin S}\phi_{v,s}^{w_{\beta\alpha\beta\alpha\beta},\sph}\otimes\bigotimes_{v\in S}M_v(\phi_{v,s},w_{\beta\alpha\beta\alpha\beta})_{-2s\rho_{P_\alpha}},\]
where $S$ is a finite set of places such that for $v\notin S$, $\phi_{v,s}$ is spherical, and $\phi_{v,s}^{w_{\beta\alpha\beta\alpha\beta},\sph}$ are certain spherical vectors. Also, the representations $R_i$ of $M_\beta^\vee$ can be determined from the action of the Levi of $P_\alpha$ on its unipotent radical; there are two of them, and $R_1$ is the representation $\Sym^3\otimes\det^{-1}$ of $\GL_2$, and $R_2$ is the determinant. Thus the quotient of $L$-functions is
\[\frac{L^S(5s,\tilde\pi_F,\Sym^3)\zeta^S(10s)}{L^S(5s+1,\tilde\pi_F,\Sym^3)\zeta^S(10s+1)}.\]
\indent Now by Harish-Chandra, the local intertwining operators are all holomorphic for $\re(s)>0$ since $\tilde\pi_F$ is tempered. So we only have to worry about the poles and zeros of the $L$-functions in the quotient above. Again since $\re(s)>0$, the $L$-functions in the denominator do not vanish as they are in the range of convergence. By a result of Kim and Shahidi \cite{KS}, the symmetric cube $L$-function is entire, and so the only pole in the numerator comes from the $\zeta$-function at $s=1/10$. But if $L(1/2,\tilde\pi_F,\Sym^3)=0$, this zero cancels with the pole from the $\zeta$-function.\\
\indent Since the poles of $E(\phi,s)$ are determined by the poles of the constant term at all standard proper parabolics, we are done.
\end{proof}
Now fix $F$ a holomorphic cuspidal eigenform of weight $k\geq 2$ and trivial nebentypus. For $s\in\C$ with $\re(s)>0$, let us write
\[\mc{L}_\alpha(\tilde\pi_F,s)=\textrm{Langlands quotient of }\iota_{P_\alpha(\A)}^{\G_2(\A)}(\tilde\pi_F,s).\]
This notion was introduced just before Theorem \ref{thmgrbac}.\\
\indent We now compute the Eisenstein multiplicity of this Langlands quotient (see Definition \ref{defmult}).
\begin{theorem}
\label{thmeismultg2}
Let $E$ be an irreducible, finite dimensional representation of $\G_2(\C)$, and say that $E$ has highest weight $\Lambda$. Write
\[\Lambda=c_1(2\alpha+3\beta)+c_2(\alpha+2\beta)\]
with $c_1,c_2\in\Z_{\geq 0}$. Let $F$ be a holomorphic cuspidal eigenform of weight $k$ and trivial nebentypus, and let $s\in\C$ with $\re(s)\geq 0$. If $c_2=0$ and $k=2c_1+4$, also assume that
\[L(1/2,\tilde\pi_F,\Sym^3)=0.\]
Then
\[m_{[P_\alpha]}^i(\mc{L}_\alpha(\tilde\pi_F,s),K_\infty,E)=\begin{cases}
1&\textrm{if }i=4,\,\,k=2c_1+c_2+4,\,\,s=(c_2+1)/10\\
&\textrm{or if }i=5,\,\,k=c_1+c_2+3,\,\,s=(3c_1+c_2+4)/10\\
&\textrm{or if }i=6,\,\,k=c_1+2,\,\,s=(3c_1+2c_2+5)/10;\\
0&\textrm{otherwise},
\end{cases}\]
and
\[m_{[P_\beta]}^i(\mc{L}_\alpha(\tilde\pi_F,s),K_\infty,E)=m_{[B]}^i(\mc{L}_\alpha(\tilde\pi_F,s),K_\infty,E)=0.\]
Therefore we also have
\[m_{\Eis}^i(\mc{L}_\alpha(\tilde\pi_F,s),K_\infty,E)= m_{[P_\alpha]}^i(\mc{L}_\alpha(\tilde\pi_F,s),K_\infty,E).\]
\end{theorem}
\begin{proof}
From the Franke--Schwermer decomposition (Theorem \ref{thmfsdecomp}) we have that the Eisenstein cohomology decomposes as
\[H_{\Eis}^i(\mf{g}_2,K_\infty;\mc{A}_E(\G_2)\otimes E)=\bigoplus_{P\in\{P_\alpha,P_\beta,B\}}\bigoplus_{\varphi\in\Phi_{E,[P]}}H^i(\mf{g}_2,K_\infty;\mc{A}_{E,[P],\varphi}(\G_2)\otimes E).\]
We will study the summands corresponding to $P_\alpha$, $P_\beta$, and $B$ in what follows.\\
\indent \it Case of $P_\alpha$. \rm Let $\varphi'$ be an associate class of cuspidal automorphic representations for $E$ and $[P_\alpha]$ as in Section \ref{sectfsdecomp}. Then $\varphi'$ contains a cuspidal automorphic representation of $M_\alpha(\A)\cong\GL_2(\A)$, and which therefore must be of the form
\[\tilde\pi'\otimes\delta_{M_\beta(\A)}^{s'}\]
where $\tilde\pi'$ is a unitary cuspidal automorphic representation of $\GL_2(\A)$ and $s'\in\C$. After possibly conjugating by $w_{\beta\alpha\beta\alpha\beta}$, we may even assume $\re(s')\geq 0$.\\
\indent First, we note that the infinitesimal character of $\mc{A}_{E,[P_\alpha],\varphi'}(\G_2)$ as a $(\mf{g}_2,K_\infty)$-module must match that of $E$. The former is given by the Weyl orbit of $\lambda_{\tilde\pi'}+2s'\rho_{P_\alpha}$, where $\lambda_{\tilde\pi'}$ is the infinitesimal character of $\tilde\pi'$, and the latter is given by the Weyl orbit of $\Lambda+\rho$. But the weight $\Lambda+\rho$ is regular and real, and so since $\lambda_{\pi'}$ is a multiple of the root $\alpha$ and $\rho_{P_\alpha}$ is a multiple of the root $\alpha+2\beta$, it follows that $\lambda_{\pi'}$ and $s'$ are real and nonzero. In particular, $s'>0$ since we assumed $\re(s')\geq 0$.\\
\indent Now we apply Theorem \ref{thmgrbac} and Proposition \ref{propEh0} to find that the cohomology space
\[H^*(\mf{g}_2,K_\infty;\mc{A}_{E,[P_\alpha],\varphi'}(\G_2)\otimes E),\]
if nontrivial, is made up of subquotients of the cohomology spaces
\begin{equation}
\label{eqcohlalphag2}
H^*(\mf{g}_2,K_\infty;\mc{L}_\alpha(\tilde\pi',s')\otimes E)
\end{equation}
and
\begin{equation}
\label{eqcohindalphag2}
H^*(\mf{g}_2,K_\infty;\otimes\Ind_{P_\alpha(\A)}^{\G_2(\A)}(\tilde\pi'\otimes\Sym(\mf{a}_{P_\alpha,0})_{(2s'+1)\rho_{P_\alpha}})\otimes E).
\end{equation}
\indent We claim that if \eqref{eqcohlalphag2} is nonzero, then $\tilde\pi'$ is cohomological. This will imply that $\tilde\pi'$ is attached to a cuspidal holomorphic eigenform of weight at least $2$. To start, we split into two cases: Either $\tilde\pi_\infty'$ is tempered or nontempered. Of course, by Selberg's conjecture, the latter possibility should not occur, but we will use the following ad-hoc argument to bypass a dependence on this conjecture.\\
\indent So assume now, for sake of contradiction, both that the cohomology space \eqref{eqcohlbetag2} is nontrivial and that $\tilde\pi_\infty'$ is nontempered. By the Langlands classification for real groups, $\tilde\pi_\infty'$ is the Langlands quotient of a representation induced from a character, say $\chi$, of $T(\R)$, and then $\mc{L}_\alpha(\tilde\pi',s')_\infty$ is the Langlands quotient of a representation induced from $\chi\delta_{P_\alpha(\R)}^{s'}$. If $\mc{L}_\alpha(\tilde\pi',s')_\infty\otimes E$ has nontrivial $(\mf{g}_2,K_\infty)$-cohomology, then by \cite{BW}, Theorem VI.1.7 (iii) (or rather, the analogue of this theorem with twisted coefficients) so does the (normalized) induced representation
\[\iota_{B(\R)}^{\GSp_4(\R)}(\chi\delta_{P_\alpha(\R)}^{s'}).\]
By \cite{BW}, Theorem III.3.3 and induction in stages, the induction
\[\iota_{(B\cap\GL_2)(\R)}^{\GL_2(\R)}(\chi\delta_{P_\alpha(\R)}^{s'})\]
has nontrivial $(\sl_2,\O(2))$-cohomology when twisted by some finite dimensional representation of $\GL_2(\C)$, and hence so does
\[\iota_{(B\cap\GL_2)(\R)}^{\GL_2(\R)}(\chi)\]
since $\delta_{P_\beta(\R)}$ is trivial on $\SL_2(\R)$. Thus by \cite{BW}, Theorem VI.1.7 (ii), $\tilde\pi_\infty'$, which is the Langlands quotient of this induction, also has cohomology. But the cohomological cusp forms for $\GL_2$ are the holomorphic modular forms, which are in particular tempered at infinity. This is a contradiction.\\
\indent Therefore, still assuming \eqref{eqcohlalphag2} is nonzero, we must have $\tilde\pi_\infty'$ is tempered. Then by (the twisted version of) \cite{BW}, Lemma VI.1.5,
\[H^*(\mf{g}_2,K_\infty;\iota_{P_\alpha(\R)}^{\G_2(\R)}(\tilde\pi_\infty',s')\otimes E)\ne 0.\]
But by \cite{BW}, Theorem III.3.3, this is computed in terms of the cohomology of $\tilde\pi_\infty'$ itself, and we conclude that $\tilde\pi'$ is cohomological, as desired.\\
\indent If instead \eqref{eqcohindalphag2} is nonzero, then we can use Theorem \ref{thmcohind} to conclude that $\tilde\pi'$ is cohomological. In any case, if
\[H^*(\mf{g}_2,K_\infty;\mc{A}_{E,[P_\alpha],\varphi'}(\G_2)\otimes E)\ne 0,\]
then $\tilde\pi'=\tilde\pi_{F'}$ for some  cuspidal holomorphic eigenform $F'$ of weight at least $2$. Furthermore, any irreducible subquotient of this cohomology space must be an irreducible subquotient of either \eqref{eqcohindalphag2} or \eqref{eqcohlbeta}. The former, by Theorem \ref{thmcohind} is a sum of copies of
\[\iota_{P_\beta(\A_f)}^{\GSp_4(\A_f)}((\tilde\pi_{F'}\boxtimes\psi')_f,s'),\]
while the latter is a sum of copies of the Langlands quotient of this induction. In particular, they are all nearly equivalent and occur in this induction.\\
\indent So if we now assume that
\[H^*(\mf{g}_2,K_\infty;\mc{A}_{E,[P_\alpha],\varphi'}(\G_2)\otimes E)\]
contains $\mc{L}_\alpha(\tilde\pi_{F},s)_f$ as a subquotient, then since we have shown $s'>0$, by Proposition \ref{propequalmaxparag2}, $\tilde\pi'=\tilde\pi_F$ and $s=s'$.\\
\indent Therefore we have just shown that $\varphi'$ contains $\tilde\pi_F\otimes\delta_{P_\alpha(\A)}^s$. Since no two classes $\varphi'$ overlap, this determines $\varphi'$ uniquely. By Proposition \ref{propEh0}, Proposition \ref{lemholoESg2} and our vanishing assumption on the symmetric cube $L$-function of $\tilde\pi_F$, we have
\[\mc{A}_{E,[P_\alpha],\varphi}(\G_2)\cong\Ind_{P_\alpha(\A)}^{\G_2(\A)}(\tilde\pi_F\otimes\Sym(\mf{a}_{P_\alpha,0})_{(2s+1)\rho_{P_\alpha}}),\]
and then Proposition \ref{propcohindalpha} gives the $[P_\beta]$-Eisenstein multiplicities claimed.\\
\indent \it Case of $P_\beta$. \rm Let $\varphi'$ this time be an associate class for $E$ and $P_\beta$. Then $\varphi'$ contains a representation of the form
\[\tilde\pi'\otimes\delta_{M_\alpha(\A)}^{s'}\]
with $\tilde\pi'$ a unitary cuspidal automorphic representation of $\GL_2(\A)$ and $s'\in\C$ with $\re(s')\geq 0$. Then the same argument as in the $P_\alpha$ case shows that $s'$ is real and positive.\\
\indent Now we once again apply Theorem \ref{thmgrbac} and Proposition \ref{propEh0} to find that the cohomology space
\[H^*(\mf{g}_2,K_\infty;\mc{A}_{E,[P_\beta],\varphi'}(\G_2)\otimes E),\]
if nontrivial, is made up of subquotients of the cohomology spaces
\begin{equation}
\label{eqcohlbetag2}
H^*(\mf{g}_2,K_\infty;\mc{L}_{P_\beta(\A)}^{\G_2(\A)}(\tilde\pi'\otimes\psi',s')\otimes E)
\end{equation}
and
\begin{equation}
\label{eqcohindbetag2}
H^*(\mf{g}_2,K_\infty;\Ind_{P_\alpha(\A)}^{\GSp_4(\A)}(\tilde\pi'\otimes\Sym(\mf{a}_{P_\beta,0})_{(2s'+1)\rho_{P_\beta}})\otimes E).
\end{equation}
Just as in the $P_\alpha$ case, the nonvanishing of either \eqref{eqcohlbetag2} or \eqref{eqcohindbetag2} implies that $\tilde\pi'=\tilde\pi_{F'}$ for a cuspidal holomorphic eigenform $F'$ of weight at least $2$, and that any irreducible subquotient of
\[H^*(\mf{g}_2,K_\infty;\mc{A}_{E,[P_\beta],\varphi'}(\G_2)\otimes E)\]
is nearly equivalent to an irreducible subquotient of
\[\iota_{P_\beta(\A_f)}^{\G_2(\A_f)}(\tilde\pi'_f,s').\]
Now we use Proposition \ref{propdistalphabetag2} to conclude that $\mc{L}_\beta(\tilde\pi_F,s)$ cannot also occur as a subquotient, which finishes the proof in the case of $P_\beta$.\\
\indent \it Case of $B$. \rm Now we let $\varphi'$ be an associate class for $E$ and $[B]$. So $\varphi'$ contains a character of $T(\A)$ of the form
\[(\psi_1'\boxtimes\psi_2')\otimes e^{\langle H_B(\cdot),s_1'\alpha+s_2'\beta\rangle},\]
where $\psi_1',\psi_2'$ are Dirichlet characters and $s_1',s_2'\in\C$. Let us write
\[\psi'=\psi_1'\boxtimes\psi_2'\boxtimes\psi_3'\]
for short.\\
\indent We will study the piece $\mc{A}_{E,[B],\varphi'}(\G_2)$ of the Franke--Schwermer decomposition using the (Franke) filtration of Theorem \ref{thmfrfil}. By that theorem, there is a filtration on the space $\mc{A}_{E,[B],\varphi'}(\G_2)$ whose graded pieces are parametrized by certain quadruples $(Q,\nu,\Pi,\mu)$. For the convenience of the reader, we recall what these quadruples consist of now:
\begin{itemize}
\item $Q$ is a standard parabolic subgroup of $\G_2$;
\item $\nu$ is an element of $(\mf{t}\cap\mf{m}_{Q,0})^\vee$;
\item $\Pi$ is an automorphic representation of $M_Q(\A)$ occurring in
\[L_{\disc}^2(M_Q(\Q)A_Q(\R)^\circ\backslash M_Q(\A))\]
and which is spanned by values at, or residues at, the point $\nu$ of Eisenstein series parabolically induced from $(B\cap M_Q)(\A)$ to $M_Q(\A)$ by representations in $\varphi'$; and
\item $\mu$ is an element of $\mf{a}_{Q,0}^\vee$ whose real part in $\Lie(A_{M_Q}(\R))$ is in the closure of the positive cone, and such that $\nu+\mu$ lies in the Weyl orbit of $\Lambda+\rho$.
\end{itemize}
Then the graded pieces of $\mc{A}_{E,[B],\varphi'}(\G_2)$ are isomorphic to direct sums of $\G_2(\A_f)\times(\mf{g}_2,K_\infty)$-modules of the form
\[\Ind_{Q(\A)}^{\G_2(\A)}(\Pi\otimes\Sym(\mf{a}_{Q,0})_{\mu+\rho_Q})\]
for certain quadruples $(Q,\nu,\Pi,\mu)$ of the form just described.\\
\indent For each of the four possible parabolic subgroups $Q$ and any corresponding quadruple $(Q,\nu,\Pi,\mu)$ as above, we will show using Proposition \ref{propdistalphabetag2} that the cohomology
\begin{equation}
\label{eqcohborelg2}
H^*(\mf{g}_2,K_\infty;\Ind_{Q(\A)}^{\G_2(\A)}(\Pi\otimes\Sym(\mf{a}_{Q,0})_{\mu+\rho_Q}))
\end{equation}
cannot have $\mc{L}_\alpha(\tilde\pi_{F,f},s)$ as a subquotient, which will finish the proof.\\
\indent So first assume we have a quadruple $(Q,\nu,\Pi,\mu)$ as above where $Q=B$. Then $\mf{m}_{Q,0}=0$, forcing $\nu=0$. The entry $\Pi$ is the unitarization of a representation in $\varphi'$, and thus must be a character $\psi'$ of $T(\A)$ conjugate to $\psi_1'\boxtimes\psi_2'$. Finally, we have $\mu$ is Weyl conjugate to $\Lambda+\rho$.\\
\indent Therefore the cohomology \eqref{eqcohborelgsp4} is isomorphic, by Theorem \ref{thmcohind}, to a finite sum of copies of
\[\iota_{B(\A_f)}^{\GSp_4(\A_f)}(\psi_f',\mu).\]
By Proposition \ref{propdistalphabetag2}, $\mc{L}_\beta((\tilde\pi_F\boxtimes 1)_f,s)$ cannot be a subquotient of this space, and we conclude in the case when $Q=B$.\\
\indent If now we have a quadruple $(Q,\nu,\Pi,\mu)$ where $Q=P_\alpha$, and $\nu+\mu$ is an integral weight because it is conjugate to $\Lambda+\rho$. We find that $\Pi$ is a representation generated by residual Eisenstein series at the point $\nu$ and is therefore a subquotient of the normalized induction
\[\iota_{(B\cap M_\alpha)(\A)}^{M_\alpha(\A)}(\psi',\nu),\]
where $\psi'$ is a character of $T(\A)$ conjugate to $\psi_1'\boxtimes\psi_2'$. Then by \ref{thmcohind} and induction in stages, \eqref{eqcohborelg2} is isomorphic to a subquotient of a finite sum of copies of
\[\iota_{B(\A_f)}^{\G_2(\A_f)}(\psi_f',\nu+\mu).\]
We then conclude in this case as well using Proposition \ref{propdistalphabetag2}.\\
\indent The case when $Q=P_\beta$ is completely similar, and we omit the details. When $Q=G$, it is once again similar, but easier since we do not need to use induction in stages. So we are done.
\end{proof}
\subsection{Arthur's conjectures and the cuspidal multiplicity of Langlands quotients}
\label{seccuspmultg2}
We would like now to determine the cuspidal multiplicity of the Langlands quotient we studied in Theorem \ref{thmeismultg2}. Unfortunately, not enough information is known about the CAP representations which can occur in the cuspidal spectrum of $\G_2$. So our computation will have to rely on some conjectures.\\
\indent Recall that a cuspidal automorphic representation is CAP if it is nearly equivalent to an irreducible subquotient of a parabolically induced representation. A point of view put forth by Gan and others is that CAP representations should be studied through the lens of Arthur's conjectures, as we explain now.\\
\indent In his celebrated work \cite{artconj}, Arthur introduced a series of conjectures which, for a reductive $\Q$-group $G$, classify the representations occurring in the space $L^2(G(\Q)\backslash G(\A))$. The data involved in this classification decomposes into local data, and so part of this classification is to build packets of representations of $G(\Q_v)$ for every place $v$. Of particular importance for us will be the shape of these local packets at $v=\infty$, and so we start (as Arthur did in \cite{artconj}) by reviewing these conjectures for real groups.
\subsubsection*{Arthur's conjecture for real groups}
Let $W_\R$ be the Weil group of $\R$. Recall that $W_\R$ is the union $\C^\times\cup\C^\times j$ where the element $j$ has the properties that $j^2=-1$ and
\[jzj^{-1}=\overline{z},\qquad z\in\C^\times.\]
The group $W_\R$ comes equipped with a natural multiplicative map
\[\vert\cdot\vert:W_\R\to\R_{>0}\]
extending the usual absolute value on $\C^\times$ and for which $\vert j\vert=1$.\\
\indent Now let $\mbf{G}$ be a real reductive group. Attached to $\mbf{G}$ we have the complex dual group $\mbf{G}^\vee(\C)$ and the $L$-group
\[\L{\mbf{G}}=\mbf{G}^\vee(\C)\rtimes W_\R;\]
we will not need to recall how the action of $W_\R$ on $\mbf{G}^\vee(\C)$ is defined here, but we will remark that it is trivial if $\mbf{G}$ is split.\\
\indent Langlands classified the irreducible admissible representations of $\mbf{G}$ in terms of certain homomorphisms $\psi:W_\R\to\L{\mbf{G}}$, viewed up to conjugacy under $\mbf{G}^\vee(\C)$, called \it Langlands parameters. \rm The classification is finite-to-one from representations to parameters, the preimage of any parameter under the classification being called an \it $L$-packet. \rm Certain properties of parameters correspond to certain properties of the representations in the corresponding $L$-packets; for example, if the projection of the image of a parameter $\phi$ onto $\mbf{G}^\vee(\C)$ is bounded, then $\phi$ is called \it tempered \rm because all of the representations in the corresponding $L$-packet are tempered.\\
\indent In formulating his conjectures, Arthur needed to define a new kind of parameter; the goal of his definition of parameters is not to classify representations of $\mbf{G}$, but rather to define the local (in our case, archimedean) components of a classification of certain representations of the adelic points of a group which has $\mbf{G}$ as its real factor. An \it Arthur parameter, \rm as we will call it, is a homomorphism
\[\psi:W_\R\times\SL_2(\C)\to\L{\mbf{G}},\]
viewed up to conjugacy under $\mbf{G}^\vee(\C)$, whose restriction to $W_\R$ is a tempered Langlands parameter.\\
\indent There are at least two ways to obtain a Langlands parameter from an Arthur parameter $\psi$, and the correct way perhaps is not the one suggested by the definition. Instead, given an Arthur parameter $\psi$, we define the attached Langlands parameter $\phi_\psi:W_\R\to\L{\mbf{G}}$ to be given by
\[\phi_\psi(w)=\psi\left(w,\pmat{\vert w\vert^{1/2}&\\ &\vert w\vert^{-1/2}}\right).\]
The statement of Arthur's conjecture for $\mbf{G}$ will involve the $L$-packet attached to the parameter $\phi_\psi$.\\
\indent It would be unreasonable for us to recall here all of the ingredients necessary to completely define everything that appears in the statement of Arthur's conjecture, but we will recall some of these ingredients now before stating the conjecture, albeit a minimal amount.\\
\indent Fix now an Arthur parameter $\psi$ for $\mbf{G}$. Write
\[\widetilde{C}_\psi=Z(\im(\psi),\mbf{G}^\vee(\C))\]
for the centralizer of the image of $\psi$ in $\mbf{G}^\vee(\C)$, and define the finite group
\[C_\psi=\widetilde{C}_\psi/\widetilde{C}_\psi^\circ Z(\L{\mbf{G}},\mbf{G}^\vee(\C)).\]
We can make the same definition for the Langlands parameter $\phi_\psi$ to get a group $\widetilde{C}_{\phi_\psi}$ and a finite group $C_{\phi_\psi}$, and we get a natural map
\[C_\psi\to C_{\phi_\psi},\]
which is surjective. Hence we get an injective map
\[\widehat{C}_{\phi_\psi}\to\widehat{C}_\psi,\]
where the hat denotes the set of irreducible characters of the group which it decorates.\\
\indent Let $\psi$ be an Arthur parameter for $\mbf{G}$. Arthur's conjecture asserts that there is a unique triple $(A_\psi,\epsilon_\psi,\langle\cdot,\cdot\rangle)$ where
\begin{itemize}
\item $A_\psi$ is a finite set of irreducible representations of $\mbf{G}$,
\item $\epsilon_\psi:A_\psi\to\{\pm 1\}$ is a function, and
\item $\pi\mapsto\langle\cdot,\pi\rangle$ is a function $A_\psi\to\widehat{C}_\psi$,
\end{itemize}
satisfying certain properties. Among these are that $A_\psi$ contains the $L$-packet for $\phi_\psi$, $\epsilon_\psi$ equals $1$ on this $L$-packet, and that for $\pi\in A_\psi$, we have that $\pi	$ appears in the $L$-packet for $\phi_\psi$ if and only if $\langle\cdot,\pi\rangle$ is in $\widehat{C}_{\phi_\psi}$. There are two more properties that these triples are expected to satisfy (labelled (ii) and (iii) in \cite{artconj}, Conjecture 1.3.3). The property (ii) is that a certain distribution built out of the triple and $\psi$ is stable, and the property (iii) is an identity involving these triples for endoscopic groups of $\mbf{G}$; it asserts that the distributions constructed in (ii) for $\mbf{G}$ and its endoscopic groups are related by transfer.\\
\indent In any case, we do not actually need the precise statement of this conjecture; we will check this conjecture in Chapter \ref{chartpacket} for a candidate triple attached to a particular Arthur parameter for $\G_2(\R)$ by proving that our triple is an instance of a general construction of Adams--Johnson \cite{adjo}, who prove that their construction satisfies the properties asserted by Arthur's conjecture.\\
\indent Given a triple $(\Pi_\psi,\epsilon_\psi,\langle\cdot,\cdot\rangle)$ as described above, let us call the component $\Pi_\psi$ the \it Arthur packet \rm attached to $\psi$.
\subsubsection*{Arthur's global conjecture}
Arthur's archimedean conjecture discussed above also has an analogue for nonarchimedean local fields, as long as one replaces the Weil group with the Weil--Deligne group. There is also a global conjecture which Arthur formulates (at least for split groups) in Section 2 of \cite{artconj}.\\
\indent So let $G$ be a split reductive group over $\Q$. To make a global conjecture, one must replace the Weil--Deligne group from the local situation with the conjectural Langlands group $L_\Q$. For any place $v$, the group $L_\Q$ should come equipped with embeddings $W_v\to L_\Q$, where we use $W_v$ to denote the Weil--Deligne group of $\Q_v$ unless $v$ is archimedean, in which case we use it to denote the Weil group.\\
\indent An \it Arthur parameter \rm is then a $G^\vee(\C)$-conjugacy class of maps
\[\psi:L_\Q\times\SL_2(\C)\to\L{G}\]
satisfying certain properties. Here $\L{G}$ is the global $L$-group of $G$. Restriction of such a parameter $\psi$ to $W_v$ then gives a local Arthur parameter $\psi_v$ at $v$.\\
\indent One can also make definitions of $\widetilde{C}_\psi$ and $C_\psi$ in the global setting, analogous to those made in the local setting. Then there are maps
\[\widetilde{C}_\psi\to\widetilde{C}_{\psi_v},\qquad C_\psi\to C_{\psi_v}.\]
We consider the set $A_\psi$ to be the set of all representations of the form $\pi=\otimes_v'\pi_v$ for $\pi_v\in A_{\psi_v}$. For such a $\pi$, we define $\langle\cdot,\cdot\rangle$ by
\[\langle s,\pi\rangle=\prod_v\langle s_v,\pi_v\rangle,\]
where $s\in C_\psi$ and $s_v$ is its image in $C_{\psi_v}$, and $\langle\cdot,\pi_v\rangle$ is the function appearing in Arthur's local conjecture.\\
\indent Then Arthur conjectures the following. First of all, the representations occurring in $L^2(G(\Q)\backslash G(\A))$ all occur in some $A_\psi$, and if $\psi$ is such that $\widetilde{C}_\psi$ is finite, then the representations in $L^2(G(\Q)\backslash G(\A))$ which lie in $A_\psi$ all occur in the discrete spectrum $L_{\disc}^2(G(\Q)\backslash G(\A))$.\\
\indent Furthermore, he gives a formula for the multiplicity with which these representations occur in the discrete spectrum: There should be an integer $d_\psi>0$ and a homomorphism $\xi_\psi:C_\psi\to\{\pm 1\}$ such that the multiplicity $m_\pi$ for which any $\pi\in A_\psi$ occurs in $L^2_{\disc}(G(\Q)\backslash G(\A))$ is given by
\[m_\pi=\frac{d_\psi}{\# C_\psi}\sum_{s\in C_\psi}\langle s,\pi\rangle\xi_\psi(s).\]
If $\psi$ is such that $\widetilde{C}_\psi$ is finite, let us call $A_\psi$ the \it Arthur packet \rm attached to $\psi$.
\subsubsection*{An Arthur packet for $\G_2$}
In \cite{gangurl}, Gan and Gurevich made a study of certain automorphic representations of $\G_2(\A)$ which are CAP with respect to the long root parabolic $P_\alpha$, and in Section 13 of that paper, they interpret what Arthur's conjectures would mean in terms of those CAP representations. More precisely, they define a certain Arthur parameter for $\G_2$ and explain the shape of the corresponding Arthur packets, both globally and locally. We now recall their work.\\
\indent First, let $\pi$ be a unitary cuspidal automorphic representation of $\GL_2(\A)$ with trivial central character. Then $\pi$ can be viewed as a representation of $\PGL_2(\A)$ and there should correspond to $\pi$ a global Langlands parameter
\[\phi_\pi:L_\Q\to\SL_2(\C).\]
We remark that $\SL_2(\C)$ is the dual group of $\PGL_2$, and that since $\PGL_2$ is split, we may replace the $L$-group of $\PGL_2$ with just the dual group in the definition of Langlands parameter.\\
\indent Now from $\phi_\pi$ Gan and Gurevich construct an Arthur parameter for $\G_2$ as follows. Let $\gamma$ and $\gamma'$ be two orthogonal roots of $\G_2$. Assume $\gamma$ is short, so that $\gamma'$ is long. Let $\SL_\gamma$ be the $\SL_2$ subgroup of $\G_2$ corresponding to the $\sl_2$-triple coming from $\gamma$, and similarly for $\SL_{\gamma'}$. Then because $\gamma$ and $\gamma'$ are orthogonal, $\SL_\gamma$ and $\SL_{\gamma'}$ centralize each other. Let $j_\gamma$ be the inclusion $\SL_\gamma\hookrightarrow\G_2$, and similarly define $j_{\gamma'}$. Then we can make the following composition, which we take to be our Arthur parameter $\psi$.
\[L_\Q\times\SL_2(\C)\xrightarrow{(\phi_\pi,\id)}\SL_2(\C)\times\SL_2(\C)\xrightarrow{\sim}\SL_\gamma(\C)\times\SL_{\gamma'}(\C)\xrightarrow{j_\gamma\times j_{\gamma'}}\G_2(\C).\]
The last map in this composition is well defined because $\SL_\gamma$ and $\SL_{\gamma'}$ centralize each other, and in fact its kernel is $\mu_2=\{\pm 1\}$ diagonally embedded. Since $\G_2$ is split and self dual, we may view Arthur parameters for $\G_2$ as maps into $\G_2(\C)$.\\
\indent Now we can start to look at the multiplicity formula. According to \cite{gangurl}, we have
\[\widetilde C_\psi=C_\psi\cong\Z/2\Z.\]
Therefore the representations occurring in $L^2(\G_2(\Q)\backslash\G_2(\A))$ and $A_\psi$ are discrete. We also have $d_\psi=1$ and
\[\xi_\psi(c)=\begin{cases}
1&\textrm{if }\epsilon(\pi,\Sym^3,1/2)=1;\\
(-1)^c&\textrm{if }\epsilon(\pi,\Sym^3,1/2)=-1,
\end{cases}\]
for $c\in\Z/2\Z$. Here $\epsilon(\pi,\Sym^3,1/2)$ is the sign of the functional equation for the symmetric cube $L$-function of $\pi$.\\
\indent Of course, to get further information, we have to inspect the local situation. Write $\pi=\otimes_v'\pi_v$. The Langlands parameter $\phi_\pi$ decomposes into local parameters $\phi_{\pi_v}$ which are the parameters attached by the local Langlands correspondence for $\GL_2$ to the representations $\pi_v$. The local Arthur parameter $\psi_v$ then equals the composition
\[W_v\times\SL_2(\C)\xrightarrow{(\phi_{\pi_v},\id)}\SL_2(\C)\times\SL_2(\C)\xrightarrow{\sim}\SL_\gamma(\C)\times\SL_{\gamma'}(\C)\xrightarrow{j_\gamma\times j_{\gamma'}}\G_2(\C).\]
\indent Again according to \cite{gangurl}, the local component group $C_{\psi_v}$ is isomorphic to $\Z/2\Z$ if $\pi_v$ is discrete series, and is trivial otherwise. The local Arthur packet $A_{\psi_v}$ should have two elements if $\pi_v$ is discrete series, and should have one element otherwise. Let us write
\[A_{\psi_v}=\{\Pi_v^+,\Pi_v^-\},\qquad\textrm{if }\pi_v\textrm{ is discrete series,}\]
and otherwise
\[A_{\psi_v}=\{\Pi_v^+\},\qquad\textrm{if }\pi_v\textrm{ is not discrete series.}\]
Here $\Pi_v^+$ should be the representation which is attached to $\phi_{\psi_v}$ by the local Langlands correspondence. Hence
\[\Pi_v^+=\mc{L}_\alpha(\pi_v,1/10),\]
which is the Langlands quotient of the unitary induction of $\pi_v\otimes\vert\det\vert^{1/2}$ from $M_\alpha(\Q_v)$. (We will explain later in more detail why the parameter $\psi_v$ corresponds to this representation, at least in the archimedean case.) Then if $\pi_v$ is discrete series, we have
\[\langle c,\Pi_v^+\rangle=1,\qquad\langle c,\Pi_v^-\rangle=(-1)^c.\]
\indent Feeding all this back into the multiplicity formula above gives the following. If $\Pi\in\tilde{A}_\psi$ with $\Pi=\otimes_v'\Pi_v$ and with each $\Pi_v$ in $A_{\psi_v}$, then we have $m(\Pi)=1$ if and only if $\epsilon(\pi,\Sym^3,1/2)=1$ and $\Pi_v=\Pi_v^-$ for an even number of $v$, or $\epsilon(\pi,\Sym^3,1/2)=-1$ and $\Pi_v=\Pi_v^-$ for an odd number of $v$. Otherwise $m(\Pi)=0$.\\
\indent Based on what we have seen, we feel it is reasonable to make the following conjecture.
\begin{conjecture}
\label{conjpacket}
Let $\pi$ be a unitary cuspidal automorphic representation of $\GL_2(\A)$. Write $\pi=\otimes_v'\pi_v$ and write $\mc{L}_\alpha(\pi_v,1/10)$ for the Langlands quotient of the unitary induction of $\pi_v\otimes\vert\det\vert^{1/2}$ from $M_\alpha(\Q_v)$ to $\G_2(\Q_v)$.
\begin{enumerate}[label=(\alph*)]
\item Let $S$ be the set of places $v$ for which $\pi_v$ is discrete series. For every $v\in S$, there is a representation $\Pi_v^-$ of $\G_2(\Q_v)$, different from $\mc{L}_\alpha(\pi_v,1/10)$, such that the following holds. Let $S'\subset S$ be a subset. Then
\[\Pi=\bigotimes_{v\in S'}\Pi_v^-\otimes\sideset{}{'}\bigotimes_{v\notin S'}\mc{L}_\alpha(\pi_v,1/10)\]
occurs in $L_{\disc}^2(\G_2(\Q)\backslash\G_2(\A))$ with either multiplicity zero or one, and it does so with multiplicity one if and only if either $\epsilon(\pi,\Sym^3,1/2)=1$ and $\#S$ is even, or $\epsilon(\pi,\Sym^3,1/2)=-1$ and $\#S$ is odd.
\item If $L(\pi,\Sym^3,1/2)=0$, then the representations $\Pi$ above which occur in the discrete spectrum are cuspidal.
\item If $\pi_\infty$ is the discrete series of $\GL_2(\R)$ of even weight $k\geq 4$, then $\Pi_\infty^-$ is the discrete series representation of $\G_2(\R)$ with Harish-Chandra parameter $\frac{k-4}{2}(2\alpha+3\beta)+\rho$.
\end{enumerate}
\end{conjecture}
Of course, part (a) of this conjecture is just a slight reformulation of what was said above, and what was expected in \cite{gangurl}. Part (b) was also expected by \cite{gangurl}, and is more generally reflective of the expected behavior for CAP forms. Part (c), on the other hand, will require some explanation, and Chapter \ref{chartpacket} will be devoted to justifying it. Essentially, Adams and Johnson \cite{adjo} have made a general construction of packets corresponding to a certain type of archimedean Arthur parameters. What we will show is that the Arthur parameter $\psi_\infty$ constructed just above is of this type, and that the corresponding Adams--Johnson construction yields a packet of two representations. We will explicitly compute these two representations and show that one is the Langlands quotient $\mc{L}_\alpha(\pi_v,1/10)$ while the other is the discrete series representation with Harish-Chandra parameter $\frac{k-4}{2}(2\alpha+3\beta)+\rho$ from our conjecture.\\
\indent We remark that this discrete series representation is the one which is called ``quaternionic of weight $k/2$" by Gan--Gross--Savin \cite{ggs}. This class of quaternionic discrete series is an analogue of the holomorphic discrete series for groups such as $\GSp_4$. In fact, the analogue of our conjecture holds for $\GSp_4$ and its Siegel parabolic (as partially discussed in the proof of Theorem \ref{thmcuspmultgsp4}) and one even gets holomorphic discrete series in that case.
\subsubsection*{Back to cohomology}
We can now state what consequences Conjecture \ref{conjpacket} has for cohomology. We consider again the Langlands quotients $\mc{L}_\alpha(\tilde\pi,1/10)$ from Section \ref{seceismultg2}.
\begin{theorem}
\label{thmcuspmultg2}
Let $F$ be a cuspidal holomorphic eigenform of even weight $k\geq 4$. Assume $L(\tilde\pi_F,\Sym^3,1/2)=0$. Let $E$ be the irreducible representation of $\G_2(\C)$ of highest weight $\frac{k-4}{2}(2\alpha+3\beta)$. Assume Conjecture \ref{conjpacket}. Then
\[m_{\cusp}^i(\mc{L}_\alpha(\tilde\pi,1/10)_f,K_\infty,E)=\begin{cases}
1&\textrm{if }\epsilon(\tilde\pi_F,\Sym^3,1/2)=1\textrm{ and }i=3\textrm{ or }5,\\
&\textrm{or if }\epsilon(\tilde\pi_F,\Sym^3,1/2)=-1\textrm{ and }i=4;\\
0&\textrm{otherwise.}
\end{cases}\]
Consequently, under Conjecture \ref{conjpacket}, we have
\[m^i(\mc{L}_\alpha(\tilde\pi,1/10)_f,K_\infty,E)=\begin{cases}
1&\textrm{if }\epsilon(\tilde\pi_F,\Sym^3,1/2)=1\textrm{ and }i=3,4,\textrm{ or }5;\\
2&\textrm{if }\epsilon(\tilde\pi_F,\Sym^3,1/2)=-1\textrm{ and }i=4;\\
0&\textrm{otherwise.}
\end{cases}\]
\end{theorem}
\begin{proof}
The theorem just follows from the description of the archimedean components of the representations $\Pi$ appearing in Conjecture \ref{conjpacket}. Indeed, the discrete series representation appearing there must be cohomological in middle degree $4$, and the Langlands quotient is cohomological in one degree above and below middle. The multiplicities are $1$ and not higher because $K_\infty$ is connected.
\end{proof}
\section{The archimedean Arthur packet for $\G_2$}
\label{chartpacket}
In this chapter, we compute what should be the archimedean Arthur packet discussed in Section \ref{seccuspmultg2} above. We are indebted to Jeffrey Adams, who suggested to us that this packet might be constructed via cohomological induction.
\subsection{Cohomological induction}
\label{seccohind}
In this section we recall a few facts about the cohomological induction functors of Zuckerman. Everything we discuss in this section is contained in the reference of Knapp--Vogan \cite{knvo}. We will not actually need give the definition of the cohomological induction functors because we will be able to study them explicitly enough using certain properties which we will give instead. The interested reader may refer to Chapter V of \cite{knvo}.\\
\indent We now set some notation that will be in play throughout this section. Let $\mbf{G}$ be a real reductive Lie group with complexified Lie algebra $\mf{g}$. We fix $\mbf K$ a maximal compact subgroup of $\mbf{G}$ and $\theta$ a Cartan involution which gives $\mbf K$. Let $\mf{k}$ be the complexified Lie algebra of $\mbf K$.\\
\indent We fix a $\theta$-stable Cartan subalgebra $\mf{t}\subset\mf{g}$. Let $\langle\cdot,\cdot\rangle$ be the pairing on $\mf{t}^\vee$ induced by the Killing form. We also fix a $\theta$-stable parabolic subalgebra $\mf{q}$ of $\mf{g}$ containing $\mf{t}$, and we let $\mf{l}$ be the Levi subalgebra of $\mf{q}$ containing $\mf{t}$, and $\mf{u}$ its nilpotent radical. Then $\mf{q}=\mf{l}\oplus\mf{u}$. Let $\mbf{L}$ be the Levi subgroup of $\mbf{G}$ corresponding to $\mf{l}$. Note that $\mbf{L}\cap\mbf K$ is a maximal compact subgroup of $\mbf{L}$.\\
\indent Finally, if $\mf{h}\subset\mf{g}$ is a Lie subalgebra which is stable under the adjoint action of the Cartan subalgebra $\mf{t}$, we write $\rho(\mf{h})\in\mf{t}^\vee$ for half the sum of the roots of $\mf{t}$ in $\mf{h}$.\\
\indent For $i\geq 0$, we consider the cohomological induction functors $\mc{R}^i$ from $(\mf{l},\mbf{L}\cap\mbf K)$-modules to $(\mf{g},\mbf K)$-modules as defined in Section V.1 of \cite{knvo}. These are normalized so that if $Z$ is an $(\mf{l},\mbf{L}\cap\mbf K)$-module with infinitesimal character given by $\Lambda\in\mf{t}^\vee$, the $\mc{R}^i(Z)$ has infinitesimal character given by $\Lambda+\rho(\mf{u})$ (\cite{knvo}, Corollary 5.25). We recall the following fact about the functors $\mc{R}^i$.
\begin{theorem}
\label{thmcohindirr}
Let $Z$ be an irreducible $(\mf{l},\mbf{L}\cap\mbf K)$-module with infinitesimal character given by $\Lambda\in\mf{t}^\vee$. Write $S=\dim_\C(\mf{u}\cap\mf{k})$. Assume
\[\re\langle\Lambda+\rho(\mf{u}),\gamma\rangle>0\]
for all roots $\gamma$ of $\mf{t}$ in $\mf{u}$. Then $\mc{R}^i(Z)=0$ for $i\ne S$ and $\mc{R}^S(Z)$ is nonzero and irreducible.
\end{theorem}
\begin{proof}
This is part of Theorem 0.50 in \cite{knvo}.
\end{proof}
If $\Lambda\in\mf{t}^\vee$ is a weight such that
\[\re\langle\Lambda+\rho(\mf{u}),\gamma\rangle>0\]
for all roots $\gamma$ of $\mf{t}$ in $\mf{u}$, like in the theorem above, then we say $\Lambda$ is in the \it good range. \rm Modules whose infinitesimal characters are sufficiently far in the good range are nice for us because they will make the spectral sequence we are about to discuss degenerate.\\
\indent Now this spectral sequence will be the one for cohomological induction in stages. It is slightly tricky to state with our current notation because the functors $\mc{R}^i$ have a normalization built into them which will need to be undone when writing down this spectral sequence.\\
\indent In the following, we will consider another $\theta$-stable parabolic subalgebra $\mf{q}'$ contained in $\mf{q}$ and containing $\mf{t}$. Write $\mf{q}'=\mf{l}'\oplus\mf{u}'$ for the Levi decomposition, and let $\mbf{L}'$ be the Levi subgroup of $\mbf{G}$ corresponding to $\mf{l}'$. We will also view the weight $-2\rho(\mf{u}')$ as a character of $\mbf L'$, and we will let $\C_{-2\rho(\mf{u}')}$ be the associated one dimensional $(\mf{l}',\mbf{L}'\cap\mbf K)$-module. Similarly, we consider the one dimensional $(\mf{l}',\mbf{L}'\cap\mbf K)$-module $\C_{-2\rho(\mf{u}'\cap\mf{l})}$, and the one dimensional $(\mf{l},\mbf{L}\cap\mbf K)$-module $\C_{-2\rho(\mf{u})}$, both similarly defined.
\begin{theorem}
\label{thmspecseq}
With the notation as above, for $(\mf{l}',\mbf{L}'\cap\mbf K)$-modules $Z$, there is a convergent, first-quadrant spectral sequence
\[\mc{R}^i(\mc{R}^j(Z\otimes\C_{-2\rho(\mf{u}'\cap\mf{l})})\otimes\C_{-2\rho(\mf{u})})\Longrightarrow\mc{R}^{i+j}(Z\otimes\C_{-2\rho(\mf{u}')}).\]
\end{theorem}
\begin{proof}
This is Theorem 11.77 of \cite{knvo}. (See also the Formula (11.73) there for the discrepancy in notation which forced us to twist by characters in each step.)
\end{proof}
Note that nothing about cohomologically induced modules is made explicit by the two theorems in this section; no result here tells us how to actually compute a given cohomologically induced module. The results of the next section will begin to do this, and can be combined with the spectral sequence above to obtain even more information.
\subsection{Discrete series and Harish-Chandra's classification}
\label{sectds}
In this section we classify discrete series representations in a manner which is classical, and then recast this classification using cohomolgical induction. Let us begin by setting some notation that will be used throughout this section.\\
\indent Like the previous section we fix a $\mbf{G}$ a real reductive Lie group with complexified Lie algebra $\mf{g}$. However, now we assume that $\mbf{G}$ contains a compact Cartan subgroup, say $\mbf T_c$. By results of Harish-Chandra, this assumption is equivalent to the assumption that $\mbf{G}$ has discrete series representations.\\
\indent Let $\mbf K$ be a maximal compact subgroup containing $\mbf T_c$. We furthermore assume that $\mbf K$ is connected, and hence so is $\mbf{G}$. Let $\mf{k}$ and $\mf{t}_c$ denote, respectively, the complexified Lie algebras of $\mbf K$ and $\mbf T_c$. If $\theta$ is the Cartan involution of $\mbf{G}$ which gives $\mbf K$, then everything we have just defined is $\theta$-stable. Finally, let us write $W=W(\mf{t}_c,\mf{g})$ for the Weyl group of $\mf{t}_c$ in $\mf{g}$ and $W_c=W(\mf{t}_c,\mf{k})$ for the compact Weyl group.\\
\indent We call a weight of $\mf{t}_c$ \it analytically integral \rm if it is the differential of a character $\mbf T_c\to\C^\times$. Harish-Chandra classified the discrete series representations of $\mbf{G}$ in terms of certain analytically integral weights of $\mf{t}_c$. Here is his classification.
\begin{theorem}[Harish-Chandra]
\label{thmHC}
Let $\Lambda$ be a regular weight of $\mf{t}_c$. So the weight $\Lambda$ determines a dominant Weyl chamber in $\mf{t}_c^\vee$ and hence also an ordering on the roots of $\mf{t}_c$ in $\mf{g}$, and we let $\rho_\Lambda$ denote half the sum of the roots which are positive with respect to this ordering.\\
\indent Then there is a bijection between $W_c$-orbits of regular weights $\Lambda$ of $\mf{t}_c$ such that $\Lambda-\rho_\Lambda$ is dominant and analytically integral, and discrete series representations of $\mbf{G}$ with trivial central character. Let $\pi_\Lambda$ be the discrete series representation corresponding to such a weight $\Lambda$. Then this bijection is determined by the following property.\\
\indent The ordering determined by $\Lambda$ also determines an ordering on the compact roots (that is, the roots of $\mf{t}_c$ in $\mf{k}$). Let
\[\pi_\Lambda|_K=\bigoplus_{\Lambda'} V_{\Lambda'}^{m_{\Lambda'}}\]
be the decomposition of $\pi_\Lambda$ into its $\mbf K$-types, where the sum is over all analytically integral weights $\Lambda'$ of $\mf t_c$ which are dominant with respect to the positive compact roots, and $V_{\Lambda'}$ is the irreducible representation of $\mf{k}$ with highest weight $\Lambda'$. Let $\rho_{\Lambda,c}$ denote half the sum of the positive compact roots. Then the property determining $\pi_\Lambda$ in terms of $\Lambda$ is that the smallest $\Lambda'$ for which $m_{\Lambda'}$ is nonzero is
\[\Lambda'=\Lambda+\rho_\Lambda-2\rho_{\Lambda,c}.\]
For this particular $\Lambda'$ we have $m_{\Lambda'}=1$.\\
\indent Finally, any discrete series representation can be obtained from one of the ones above by a central twist.
\end{theorem}
In the setting of this theorem, we call (the $W_c$ orbit of) $\Lambda$ the \it Harish-Chandra parameter \rm of the discrete series representation $\pi_\Lambda$ (or any of its central twists) and we call the representation $V_{\Lambda'}$ with
\[\Lambda'=\Lambda+\rho_\Lambda-2\rho_{\Lambda,c}.\]
the \it lowest $\mbf K$-type \rm of $\pi_{\Lambda}$.\\
\indent Now we explain how to get discrete series representations by cohomologically inducing characters. Let $\Lambda$ be a regular weight of $\mf{t}_c$, and as in Theorem \ref{thmHC}, order the roots of $\mf{t}_c$ in $\mf{g}$ so that $\Lambda$ is dominant and let $\rho_\Lambda$ be half the sum of positive roots. Let $\mf{b}$ be the Borel subalgebra of $\mf{g}$ which is determined by the positive roots in this ordering, and let $\mf{b}=\mf{t}\oplus\mf{u}_0$ be its Levi decomposition. Then $\mf{t}$ is just the sum of $\mf{t}_c$ and the center of $\mf{g}$. Let $\mbf{T}$ be the corresponding maximal torus in $\mbf{G}$.\\
\indent We will consider the cohomological induction from $(\mf{t},\mbf{T}\cap \mbf K)$-modules to $(\mf{g},\mbf K)$-modules with respect to the parabolic subalgebra $\mf{b}$. If $\Lambda-\rho_\Lambda$ is analytically integral, we view it as a character of $\mbf{T}$ which has a trivial action of the center of $\mbf{G}$, and also as a $(\mbf{t},\mbf{T}\cap \mbf K)$-module.
\begin{theorem}
\label{thmcohindds}
With notation as above, assume $\Lambda-\rho_{\Lambda}$ is dominant and analytically integral. Then the cohomologically induced module $\mc{R}(\Lambda-\rho_\Lambda)$ is isomorphic to (the $(\mf{g},\mbf K)$-module of) the discrete series representation of $\mbf{G}$ with Harish-Chandra parameter $\Lambda$.
\end{theorem}
\begin{proof}
This is Theorem 11.178(a) in \cite{knvo}.
\end{proof}
We remark that the weight $\Lambda-\rho_\Lambda$ in the theorem is by definition in the good range, in the sense of Theorem \ref{thmcohindirr}. Therefore it makes sense to drop the cohomological degree from the notation $\mc{R}$.
\subsection{The Adams--Johnson construction}
\label{secajconst}
We will now recall the main results of the work of Adams--Johnson \cite{adjo} as interpreted in terms of Arthur parameters. Section 3 of \cite{adjo} explains the connection between these results and Arthur's conjectures quickly, but there is also an article of Arthur \cite{artunipconj} which explains this in more detail, and which is very explicit about the parameters involved. We mostly follow this latter article.\\
\indent Most of this section is devoted to explaining the construction of the parameters that are relevant to the Adams--Johnson construction. After constructing these parameters, the construction of the associated packet by Adams--Johnson will be easy to describe using cohomological induction.\\
\indent We keep the notation of the previous section, and in particular we will be working with the objects $\mbf{G}$, $\mbf T_c$, $\mbf K$, $\mf{g}$, $\mf{t}_c$, $\mf{k}$, $\theta$, $W$, and $W_c$ defined there. As before, we also write $\mbf{T}$ for the maximal torus of $\mbf{G}$ containing $\mbf{T}_c$, and $\mf{t}$ for its Lie algebra. Let $\mf{q}$ be a $\theta$-stable parabolic subalgebra with Levi factor $\mf{l}$ containing $\mf{t}$, and let $\mbf{L}$ be the corresponding Levi subgroup of $\mbf{G}$. Let $\mf{u}$ be the nilpotent radical of $\mf{q}$.\\
\indent We will consider, in what follows, the $L$-group of $\mbf{G}$,
\[\L{\mbf{G}}=\mbf{G}^\vee(\C)\rtimes W_\R\]
and also that of $\mbf{L}$. Here $W_\R$ is the Weil group of $\R$,
\[W_\R=\C^\times\cup\C^\times j,\]
where $j^2=-1$ and $jzj^{-1}=\overline{z}$ for $z\in\C^\times$.\\
\indent Choose a maximal torus in $\mbf{G}^\vee(\C)$ and identify it with $\mbf{T}^\vee(\C)$. Then $\mbf{L}^\vee(\C)$ is identified with a Levi subgroup of $\mbf{G}^\vee(\C)$ containing $\mbf{T}^\vee(\C)$.
\subsubsection*{Construction of $\xi$}
\indent The first order of business is to construct an embedding $\xi$ of $\L{\mbf{L}}$ into $\L{\mbf{G}}$. This is done on pp. 30-31 of \cite{artunipconj} and we recall here the process.\\
\indent We already have an embedding $\mbf{L}^\vee(\C)\hookrightarrow\L{\mbf{G}}$ because we have embedded $\mbf{L}^\vee(\C)$ into $\mbf{G}^\vee(\C)$. So to get the embedding $\xi$ of $L$-groups, we only need to describe where to send elements of $W_\R$. We will describe $\xi(z)$ for $z\in\C^\times$ and $\xi(j)$ separately.\\
\indent First, for $z\in\C^\times$, let $t_z$ be the unique element of $\mbf{T}^\vee(\C)$ such that
\[\Lambda^\vee(t_z)=z^{\langle\Lambda^\vee,\rho(\mf{u})\rangle}\overline{z}^{-\langle\Lambda^\vee,\rho(\mf{u})\rangle},\]
for any character $\Lambda^\vee$ of $\mbf{T}^\vee(\C)$ (equivalently, $\Lambda^\vee$ is a cocharacter of $\mbf{T}(\C)$). Here, as before, $\rho(\mf{u})$ is half the sum of the roots of $\mf{t}$ in $\mf{u}$. Then the map $z\mapsto t_z$ is a homomorphism $\C^\times\to\mbf{T}^\vee(\C)$, and we set
\[\xi(z)=t_z\rtimes z.\]
\indent We now describe $\xi(j)$. Let $n_L$ be any element of the derived group of $\mbf{L}^\vee(\C)$ normalizing $\mbf{T}^\vee(\C)$ and such that $\Ad(n_L)$ sends the positive roots of $\Lie(\mbf{T}^\vee(\C))$ in $\Lie(\mbf{L}^\vee(\C))$ to negative ones. Similarly define the element $n_G$ in the derived group of $\mbf{G}^\vee(\C)$. Then we declare
\[\xi(j)=n_Gn_L^{-1}\rtimes j.\]
\indent Thus we have defined the embedding $\xi:\L{\mbf{L}}\to\L{\mbf{G}}$. It depends on certain choices, but only up to conjugation in $\L{\mbf{G}}$.
\subsubsection*{Construction of $\psi$}
Now we construct an Arthur parameter $\psi$. Fix a character $\Lambda:\mbf{L}\to\C^\times$. We also denote by $\Lambda$ the restriction of this character to $\mbf{T}$, and the weight of $\mf{t}$ which that gives. We assume that $\Lambda$ is dominant with respect to the roots of $\mf{t}$ in $\mf{u}$.\\
\indent This character $\Lambda$, when viewed as a one dimensional representation of $\mbf{L}$, determines by the archimedean local Langlands correspondence, a Langlands parameter
\[\phi_{\Lambda}:W_\R\to\L{\mbf{L}}\]
whose image lies in $Z(\mbf{L}^\vee(\C))\rtimes W_\R$, where we have written $Z(\mbf{L}^\vee(\C))$ for the center of $\mbf{L}^\vee(\C)$.\\
\indent Now let
\[\psi_L:W_\R\times\SL_2(\C)\to\L{\mbf{L}}\]
be the Arthur parameter for $\mbf{L}$ determined by the requirements that
\[\psi_L|_{W_\R}=\phi_{\Lambda}\]
and that $\sm{1&1\\ 0&1}$ maps to a principal unipotent element in $\mbf{L}^\vee(\C)$. Then we define the Arthur parameter $\psi$ for $\mbf{G}$ by
\[\psi=\xi\circ\psi_L,\]
where $\xi$ is the embedding above.
\subsubsection*{Construction of the Adams--Johnson packet}
Let $w\in W$ be a Weyl group element. We use $w$ to twist our parabolic subalgebra $\mf{q}$ in the following way. If $\Delta(\mf{q})$ denotes the set of roots of $\mf{t}$ in $\mf{q}$, we let $\mf{q}_w$ be the parabolic subalgebra of $\mf{g}$ containing precisely all the roots $w\gamma$ for $\gamma\in\Delta(\mf{q})$, along with $\mf{t}$. Let $\mf{l}_w$ be the Levi factor containing $\mf{t}$ and let $\mbf{L}_w$ be the Levi subgroup of $\mbf{G}$ corresponding to $\mf{l}_w$.\\
\indent Now Lemma 2.5 (1) of \cite{adjo} states that all the Levis $\mbf{L}_w$ for $w\in W$ are inner forms of each other. Therefore they have the same $L$-groups. So let $\phi_{\Lambda,w}:W_\R\to\L{\mbf{L}}_w$ be the Langlands parameter given by $\phi_\Lambda$, but viewed as a parameter for $\mbf{L}_w$. Then $\phi_{\Lambda,w}$ corresponds to a one dimensional representation of $\mbf{L}_w$, which we denote by $\Lambda_w$.\\
\indent We may now define the Adams--Johnson packet. Let $\mc{R}=\mc{R}^S$ be the cohomological induction functor of Section \ref{seccohind} (see in particular Theorem \ref{thmcohindirr}).
\begin{definition}
The \it Adams--Johnson \rm packet attached to the Arthur parameter $\psi$ constructed above is the set
\[\mr{AJ}_\psi=\sset{\mc{R}(\Lambda_w)}{w\in W},\]
where, for $w\in W$, the cohomological induction of $\Lambda_w$ is taken with respect to the parabolic subalgebra $\mf{q}_w$.
\end{definition}
Actually, it can be that a lot of the representations in $\mr{AJ}_\psi$ corresponding to different elements $w$ are equal. In fact, it is noted in \cite{adjo} that $w,w'\in W$ lie in the same double coset in $W_c\backslash W/W(\mf{t},\mf{l})$ if and only if $\mc{R}(\Lambda_w)\cong\mc{R}(\Lambda_{w'})$. Arthur notes in \cite{artunipconj} that this set $W_c\backslash W/W(\mf{t},\mf{l})$ of double cosets is in bijection with the component group $C_\psi$ attached to $\psi$.\\
\indent Now we have the following theorem, which is the main result of \cite{adjo} as interpreted by Arthur \cite{artunipconj}.
\begin{theorem}[Adams--Johnson]
For each $\psi$ as above, there is a function $\epsilon_\psi:\mr{AJ}_\psi\to\{\pm 1\}$ and a pairing $\langle\cdot,\cdot\rangle$ between $C_\psi$ and $\mr{AJ}_\psi$, such that the triples $(\mr{AJ}_\psi,\epsilon_\psi,\langle\cdot,\cdot\rangle)$ satisfy the conclusion of Arthur's conjecture (\cite{artconj}, Conjecture 1.3.3).
\end{theorem}
\begin{proof}
This is the main result of \cite{adjo}; See Theorems 2.13 and 2.21 there. Arthur \cite{artunipconj}, Section 5, also describes how to get the objects $\epsilon_\psi$ and $\langle\cdot,\cdot\rangle$ from the objects appearing in \cite{adjo}.
\end{proof}
\begin{remark}
Strictly speaking, although it is mentioned in \cite{adjo} and \cite{artunipconj} that $\mr{AJ}_\psi$ contains the $L$-packet attached to the Langlands parameter $\phi_\psi$ associated with $\psi$, a proof of this is written down in neither of these references. We will be able to check this directly, however, for the packet that we obtain for $\G_2(\R)$ in the next section.
\end{remark}
\subsection{Determination of the packet for $\G_2(\R)$}
Recall that in Section \ref{seccuspmultg2} we constructed a global Arthur parameter for $\G_2$ whose associated packet should contain the CAP forms that are nearly equivalent to the Eisenstein series considered there. This Arthur parameter has a local archimedean component, and in this section we will recall how it is constructed and denote it $\psi'$.\\
\indent This notation suggests that there will be another Arthur parameter in play, and indeed, we will construct one via the process of the previous section. This other parameter will be denoted $\psi$. But the parameters $\psi$ and $\psi'$ will turn out to be equal, which means that we can apply the methods of the previous section and obtain an Adams--Johnson packet for $\psi$, or equivalently, for $\psi'$. We then determine the representations in this packet explicitly. They will turn out to be the Langlands quotient and the discrete series representation discussed in Conjecture \ref{conjpacket}.\\
\indent To be consistent with the rest of this chapter, we change some of the notation used in Section \ref{secg2}. In particular, let us write $\mbf{G}_2=\G_2(\R)$ for the real split $\G_2$, $\mbf{K}$ for a fixed maximal compact subgroup of $\mbf{G}_2$, and $\mbf{T}_c$ for a fixed maximal torus contained in $\mbf{K}$. Let $\mf{g}_2$, $\mf{k}$, and $\mf{t}_c$ be the respective complexified Lie algebras. We still write $\alpha$ and $\beta$, respectively, for fixed long and short simple roots of $\mf{t}_c$ in $\mf{g}_2$, and we assume we have chosen $\mbf{K}$ so that $\pm\beta$ and $\pm(2\alpha+3\beta)$ are the compact roots. We fix $\theta$ the Cartan involution giving $\mbf{K}$. Write $W$ for the Weyl group of $\mf{t}_c$ in $\mf{g}_2$ and $W_c$ for the Weyl group of $\mf{t}_c$ in $\mf{k}$.\\
\indent On the dual side of things, we have $\mbf{G}_2^\vee(\C)=\G_2(\C)$, and we fix a maximal torus in $\G_2(\C)$, identifying it with $\mbf{T}_c^\vee(\C)$. Passage to the dual side switches the long and short simple roots, so $\alpha^\vee$ becomes a short simple root for $\mbf{T}_c^\vee(\C)$ in $\G_2(\C)$, and $\beta^\vee$ becomes a long simple root. In order to distinguish when we are on the dual side and when we are not, we will denote roots of $\mbf{T}_c^\vee(\C)$ in $\G_2(\C)$ with a prime and thus write $\beta'=\alpha^\vee$ and $\alpha'=\beta^\vee$. Then we have
\begin{align*}
(\alpha+\beta)^\vee &=\alpha'+3\beta', &(\alpha+2\beta)^\vee &=2\alpha'+3\beta',\\
(\alpha+3\beta)^\vee &=\alpha'+\beta', &(2\alpha+3\beta)^\vee &=\alpha'+2\beta'.
\end{align*}
\subsubsection*{The parameter $\psi'$}
Fix throughout this section an even integer $k\geq 4$. We denote by $\pi$ the discrete series representation of $\GL_2(\R)$ with trivial central character. Then $\pi$ may be viewed as a representation of $\PGL_2(\R)$. Let
\[\phi:W_\R\times\L{\PGL_2(\R)}\]
be its Langlands parameter. This can be made explicit. For one thing, the $L$-group of $\PGL_2(\R)$ is just $\SL_2(\C)\times W_\R$, and then $\phi$ takes the following form. For $z\in\C^\times$, we have
\begin{equation}
\label{eqphiofz}
\phi(z)=\pmat{(z/\overline{z})^{(k-1)/2}& 0\\ 0&(z/\overline{z})^{-(k-1)/2}}\times z,
\end{equation}
which is an element of $\SL_2(\C)\times\C^\times\subset\SL_2(\C)\times W_\R$, and
\begin{equation}
\label{eqphiofj}
\phi(j)=\pmat{0& -1\\ 1&0}\times j.
\end{equation}
(See, for example \cite{praparam}, Proposition 2.) Note that the quantity $(z/\overline{z})^{(k-1)/2}$ should be interpreted as
\[(z/\overline{z})^{(k-1)/2}=\vert z\vert^{k-1}\overline{z}^{-(k-1)}.\]
\indent For $\gamma'$ a root of $\mbf{T}_c^\vee(\C)$ in $\G_2(\C)$, let $\SL_{\gamma'}(\C)\subset\G_2(\C)$ be the $\SL_2(\C)$ associated with $\gamma'$. It is generated by the images of the unipotent root group homomorphisms $\mbf{x}_{\pm\gamma'}$ corresponding to $\pm\gamma'$. If $\gamma_1'$ and $\gamma_2'$ are orthogonal roots, then the elements in the image of $\mbf{x}_{\gamma_1'}$ and $\mbf{x}_{\gamma_2'}$ commute. Hence $\SL_{\gamma_1'}(\C)$ and $\SL_{\gamma_2'}(\C)$ centralize each other. We thus get a map
\begin{equation}
\label{eqSLmapg2}
\SL_{\gamma_1'}(\C)\times\SL_{\gamma_2'}(\C)\to\G_2(\C).
\end{equation}
Now the maximal torus $\mbf{T}_c^\vee(\C)$ in $\G_2(\C)$ is just the image under this map of the product of the diagonal tori in $\SL_{\gamma_1'}(\C)$ and $\SL_{\gamma_2'}(\C)$. The character group of $\mbf{T}_c^\vee(\C)$ is generated by its root lattice, and it is visible from the root lattice of $\G_2$ that the characters $\gamma_1'$ and $\gamma_2'$ generate an index $2$ subgroup of the character group of $\mbf{T}_c^\vee(\C)$. It follows that the map above has a kernel of order $2$. The character $(\gamma_1'+\gamma_2')/2$ is a root of $\mbf{T}_c^\vee(\C)$, and it generates the whole character group along with $\gamma_1'$ and $\gamma_2'$. All three of these characters, when lifted to $\SL_{\gamma_1'}(\C)\times\SL_{\gamma_2'}(\C)$, are trivial on the diagonally embedded $\mu_2=\{\pm 1\}$, and so in fact the kernel of the above map is this $\mu_2$. Thus we identify
\[\SL_{\gamma_1'}(\C)\times\SL_{\gamma_2'}(\C)/\mu_2\]
as a subgroup of $\G_2(\C)$ in this way. It contains $\mbf{T}_c^\vee(\C)$ and is simply the subgroup generated by the inages of the unipotent root group homomorphisms $\mbf{x}_{\pm\gamma_1'}$ and $\mbf{x}_{\pm\gamma_2'}$.\\
\indent Now we can define the Arthur parameter $\psi'$. It is the composition
\begin{align*}
W_\R\times\SL_2(\C)\xrightarrow{(\phi,\id)}(\SL_2(\C)\times W_\R)\times\SL_2(\C)&\xrightarrow{\sim}\SL_{\beta'}(\C)\times\SL_{2\alpha'+3\beta'}(\C)\times W_\R\\
&\to\G_2(\C)\times W_\R=\L{\mbf{G}}_2,
\end{align*}
where middle map leaves the order of the $\SL_2$'s the same and only rearranges the placement of $W_\R$, and the last map is the product of the map from \eqref{eqSLmapg2} with the identity map of $W_\R$. To be clear, $\phi$ is mapping into the subgroup $\SL_{\beta'}(\C)\times W_\R$ of $\L{\mbf{G}}_2$, so the image of the restriction of $\psi'$ to $W_\R$ lands in that subgroup.\\
\indent The choice of the pair $(\beta',2\alpha'+3\beta')$ of orthogonal roots doesn't really matter, as long as $\phi$ is mapping to the short root $\SL_2(\C)$. Any other choice of orthogonal roots would lead to an Arthur parameter which is conjugate to $\psi'$.
\subsubsection*{The Levi $\mbf{L}_{1,1}$}
We now begin working towards constructing a parameter $\psi$ via the constructions from Section \ref{secajconst}, and therefore we must start by constructing a Levi subgroup of $\mbf{G}_2$.\\
\indent First, let $\mf{q}_{1,1}$ be the parabolic subalgebra of $\mf{g}_2$ whose Levi $\mf{l}_{1,1}$ contains the roots $\pm(\alpha+2\beta)$ along with $\mf{t}_c$, and whose nilpotent radical $\mf{u}_{1,1}$ contains the roots $-(\alpha+3\beta)$, $-\beta$, $\alpha$, $\alpha+\beta$, and $2\alpha+3\beta$. (These five roots are the ones lying above the line containing $\alpha+2\beta$ in the root diagram.) Then let $\mbf{L}_{1,1}$ be the Levi subgroup of $\mbf{G}_2$ containing $\mbf{T}_c$ and corresponding to $\mf{l}_{1,1}$. The notation is justified by the following lemma.
\begin{lemma}
\label{lemL11}
The Levi $\mbf{L}_{1,1}$ is isomorphic to $\U(1,1)$
\end{lemma}
\begin{proof}
First we look at the complexified situation. The group $\mbf{L}_{1,1}(\C)$ is the Levi subgroup of $\G_2(\C)$ containing $\mbf{T}_c(\C)$ and the images of the unipotent root group homomorphisms $\mbf{x}_{\pm(\alpha+2\beta)}$. Therefore, as $\alpha$ is orthogonal to $\alpha+2\beta$, $\mbf{L}_{1,1}(\C)$ is the subgroup of
\[\SL_{\alpha+2\beta}(\C)\times\SL_\alpha(\C)/\mu_2\]
generated by the first factor and the diagonal torus from the second factor.\\
\indent Now we take real points. The group of real points in the diagonal torus of $\SL_\alpha(\C)$ is a one dimensional subtorus of $\mbf{T}_c$, hence is a circle $\U(1)$, and the group of real points of $\SL_{\alpha+2\beta}(\C)$ is a form of $\SL_2(\R)$. Since the root $\alpha+2\beta$ is noncompact, this form is noncompact and is thus $\SL_2(\R)$ itself.\\
\indent We conclude that
\[\mbf{L}_{1,1}\cong\SL_2(\R)\times\U(1)/\mu_2,\]
and we are done since this latter group is $\U(1,1)$.
\end{proof}
\subsubsection*{The embedding $\xi$}
We now describe the embedding $\xi$ constructed in Section \ref{secajconst} in our current context of $\mbf{L}_{1,1}$. The complexification $\mbf{L}_{1,1}(\C)$ is the Levi subgroup of $\G_2(\C)$ containing $\alpha+2\beta$ and the torus $\mbf{T}_c(\C)$, and therefore the dual group $\mbf{L}_{1,1}^\vee(\C)$ is the Levi containing $(\alpha+2\beta)^\vee=2\alpha'+3\beta'$ and $\mbf{T}_c^\vee(\C)$. The group $\mbf{L}_{1,1}^\vee(\C)$ is therefore the subgroup of
\[\SL_{\beta'}(\C)\times\SL_{2\alpha'+3\beta'}(\C)/\mu_2\]
containing the factor $\SL_{2\alpha'+3\beta'}(\C)$ and the diagonal torus in the factor $\SL_{\beta'}(\C)$. For
\[(A,B)\in\SL_{\beta'}(\C)\times\SL_{2\alpha'+3\beta'}(\C),\]
let $[A,B]$ denote the image of $(A,B)$ modulo $\mu_2$, viewed as an element of $\G_2(\C)$.\\
\indent To describe explicitly the embedding
\[\xi:\L{\mbf{L}}_{1,1}\hookrightarrow\L{\mbf{G}}_2,\]
we need to describe explicitly the elements $t_z$, $n_{L_{1,1}}$, and $n_{G_2}$ from Section \ref{secajconst}. Recall that for $z\in\C^\times$, $t_z\in\mbf{T}_c^\vee(\C)$ was defined by the requirement that for any character $\Lambda^\vee$ of $\mbf{T}_c^\vee(\C)$, we have
\[\Lambda^\vee(t_z)=z^{\langle\Lambda^\vee,\rho(\mf{u_{1,1}})\rangle}\overline{z}^{-\langle\Lambda^\vee,\rho(\mf{u_{1,1}})\rangle},\]
where $\rho(\mf{u}_{1,1})$ is half the sum of roots of $\mf{t}_c$ in $\mf{u}_{1,1}$. This half sum equals $3\alpha/2$, and so we are requiring
\[\Lambda^\vee(t_z)=(z/\overline{z})^{\langle\Lambda^\vee,3\alpha/2\rangle}.\]
Since $\alpha^\vee=\beta'$, it follows that
\[t_z=\left[\pmat{(z/\overline{z})^{3/2}&0\\ 0&(z/\overline{z})^{-3/2}},1\right]\in\SL_{\beta'}(\C)\times\SL_{2\alpha'+3\beta'}(\C)/\mu_2.\]
\indent Next we have the element $n_{G_2}$, which is any element of $\G_2(\C)$ that normalizes $\mbf{T}_c^\vee(\C)$ and whose adjoint action negates every positive root. Thus we can take
\[n_{G_2}=\left[\pmat{0&-1\\ 1&0},\pmat{0&-1\\ 1&0}\right].\]
The adjoint action of this element negates the orthogonal roots $\beta'$ and $2\alpha'+3\beta'$ and therefore acts as negation on the whole root lattice.\\
\indent Finally, we have the element $n_{L_{1,1}}$, which is any element of the derived group of $\mbf{L}_{1,1}^\vee(\C)$ which normalizes $\mbf{T}_c^\vee(\C)$ and whose adjoint action negates every positive root of $\mbf{L}_{1,1}^\vee(\C)$. The derived group of $\mbf{L}_{1,1}^\vee(\C)$ is just the subgroup $\SL_{2\alpha'+3\beta'}(\C)$, and so we may take
\[n_{L_{1,1}}=\left[1,\pmat{0&-1\\ 1&0}\right].\]
We thus have
\[n_{G_2}n_{L_{1,1}}^{-1}=\left[\pmat{0&-1\\ 1&0},1\right].\]
\indent The embedding $\xi$ is then just defined to be the usual inclusion
\[\mbf{L}_{1,1}^\vee(\C)\hookrightarrow\G_2(\C)\subset\L{\mbf{G}}_2\]
on the subgroup $\mbf{L}_{1,1}^\vee(\C)$ of $\L{\mbf{L}}_{1,1}$, and on the Weil group it is defined by the rules $\xi(z)=t_z\rtimes z$ for $z\in\C^\times$, and $\xi(j)=n_{G_2}n_{L_{1,1}}^{-1}\rtimes j$. Thus
\begin{equation}
\label{eqthreehalves}
\xi(z)=\left[\pmat{(z/\overline{z})^{3/2}&0\\ 0&(z/\overline{z})^{-3/2}},1\right]\rtimes z,\qquad z\in\C^\times,
\end{equation}
and
\[\xi(j)=\left[\pmat{0&-1\\ 1&0},1\right]\rtimes j.\]
\subsubsection*{The parameter $\psi$}
To construct our Arthur parameter $\psi$, we must start with a character $\Lambda$ of $\mbf{L}_{1,1}$. Such characters can be specified by weights of $\mbf{T}_c$ that are multiples of $\alpha/2$. The weight $\alpha$ itself corresponds to the determinant character of $\mbf{L}_{1,1}\cong\U(1,1)$. Thus we set
\[\Lambda=\frac{k-4}{2}\alpha.\]
The archimedean local Langlands correspondence attaches to this character the Langlands parameter
\[\phi_\Lambda:W_\R\to\L{\mbf{L}}_{1,1}\]
given by
\[\phi_\Lambda(z)=\left[\pmat{(z/\overline{z})^{(k-4)/2}&0\\ 0&(z/\overline{z})^{-(k-4)/2}},1\right]\rtimes z\in\mbf{L}_{1,1}^\vee(\C)\]
for $z\in\C^\times$ and
\[\phi_\Lambda(j)=1\rtimes j.\]
\indent We now define a parameter $\psi_{L_{1,1}}$ for the Levi, as in Section \ref{secajconst}. This requires us to choose a principal unipotent element in $\mbf{L}_{1,1}^\vee(\C)$, and we choose $[1,\sm{1&1\\ 0&1}]$. Thus $\psi_{L_{1,1}}$ is defined to be the Arthur parameter
\[\psi_{L_{1,1}}:W_\R\times\SL_2(\C)\to\L{\mbf{L}}_{1,1}\]
given by
\[\psi_{L_{1,1}}(w,1)=\phi_\Lambda(w),\qquad w\in W_\R\]
and
\[\psi_{L_{1,1}}\left(1,\pmat{1&1\\ 0&1}\right)=\left[1,\pmat{1&1\\ 0&1}\right]\rtimes 1.\]
It follows that the restriction of $\psi_{L_{1,1}}$ to $\SL_2(\C)$ is just the identification of $\SL_2(\C)$ with $\SL_{2\alpha'+3\beta'}(\C)$.\\
\indent Finally, we define
\[\psi=\xi\circ\psi_{L_{1,1}}.\]
We have the following lemma, which is now not so difficult.
\begin{lemma}
The parameters $\psi$ and $\psi'$ are equal.
\end{lemma}
\begin{proof}
We check that $\psi$ and $\psi'$ coincide on $\C^\times$, $j$, and $\SL_2(\C)$. We have, for $z\in\C^\times$,
\begin{align*}
\psi(z)=\xi(\psi_{L_{1,1}}(z))&=\xi\left(\left[\pmat{(z/\bar{z})^{(k-4)/2}&0\\ 0&(z/\bar{z})^{-(k-4)/2}},1\right]\rtimes z\right)\\
&=\left(\left[\pmat{(z/\bar{z})^{(k-4)/2}&0\\ 0&(z/\bar{z})^{-(k-4)/2}},1\right]\cdot t_z\right)\rtimes z\\
&=\left[\pmat{(z/\bar{z})^{(k-1)/2}&0\\ 0&(z/\bar{z})^{-(k-1)/2}},1\right]\rtimes z\qquad\textrm{(by \eqref{eqthreehalves})}\\
&=\psi'(z),
\end{align*}
where the last equality is just the definition of $\psi'$, along with \eqref{eqphiofz}. Also,
\[\xi(\psi_{L_{1,1}}(j))=\xi(1\rtimes j)=n_{G_2}n_{L_{1,1}}^{-1}\rtimes j=\left(\pmat{0&-1\\ 1&0},1\right)\rtimes j=\psi'(j)\]
by \eqref{eqphiofj}. Finally, $\psi$ and $\psi'$ coincide on $\SL_2(\C)$, because when restricted to $\SL_2(\C)$, both become the inclusion of $\SL_{2\alpha'+3\beta'}(\C)$ into $\L{\mbf{G}}_2$. Therefore we have $\psi=\psi'$, as desired.
\end{proof}
\subsubsection*{The character $\Lambda_w$}
Consider the set of double cosets
\[W_c\backslash W/W(\mf{t}_c,\mf{l}_{1,1}).\]
This set has two elements and a representative for the nontrivial coset is given by the Weyl group element that rotates the root lattice by $\pi/3$ clockwise. Let $w$ be this element.\\
\indent We consider the parabolic subalgebra $\mf{q}_2$ which contains $w\gamma$ for every root $\gamma$ of $\mf{t}_c$ in $\mf{q}_{1,1}$. Thus $\mf{q}_2$ contains the all the positive roots along with $-\beta$. The Levi subalgebra of $\mf{q}_2$ containing $\mf{t}_c$ contains the roots $\pm\beta$.\\
\indent Let $\mbf{L}_2$ be the Levi subgroup of $\mbf{G}_2$ containing $\mf{t}_c$ and corresponding to $\mf{l}_2$. Again, the notation is suggestive of the following lemma.
\begin{lemma}
\label{lemL2U2}
The Levi $\mbf{L}_2$ is isomorphic to $\U(2)$.
\end{lemma}
\begin{proof}
The proof is completely similar to the proof of Lemma \ref{lemL11}, except we replace the root $\alpha+2\beta$ there by $\beta$. Then $\beta$ is compact, so the form of $\SL_2(\R)$ that appears here is $\SU(2)$.
\end{proof}
Similar to what was discussed above for $\mbf{L}_{1,1}$, the dual group $\mbf{L}_2^\vee(\C)$ of $\mbf{L}_2$ is the subgroup of
\[\SL_{\alpha'+2\beta'}(\C)\times\SL_{\alpha'}(\C)/\mu_2\]
generated by the factor $\SL_{\alpha'}$ along with the diagonal torus in $\SL_{\alpha'+2\beta'}(\C)$. The Langlands parameter $\phi$ can be conjugated to give a parameter that sends $z\in\C^\times$ to
\[\pmat{(z/\overline{z})^{(k-4)/2}&0\\ 0&(z/\overline{z})^{-(k-4)/2}}\in\SL_{\alpha'+2\beta'}(\C)\]
and which sends $j$ to $1\rtimes j$. This is how we view the parameter $\phi$ as a Langlands parameter for $\mbf{L}_2$.\\
\indent Corresponding to this parameter via the archimedean local Langlands correspondence is the character we call $\Lambda_w$; it is the character of $\mbf{L}_2$ which acts on $\mbf{T}_c$ via the weight $\frac{k-4}{2}(2\alpha+3\beta)$.
\subsubsection*{The packet $\mr{AJ}_\psi$}
We can now construct our packet $\mr{AJ}_\psi$. By definition, it consists of the cohomologically induced representations
\[\mc{R}(\Lambda)\textrm{ and }\mc{R}(\Lambda_w).\]
The following is the main result of this section.
\begin{theorem}
\label{thmajpacketg2}
We have that
\[\mc{R}(\Lambda)\cong\mc{L}_\alpha(\pi,1/10),\]
the Langlands quotient of the parabolic induction of $\pi$ from the long root parabolic, where $\pi$ is the discrete series representation of $\GL_2(\R)$ of weight $k$, and we have that $\mc{R}(\Lambda_w)$ is the discrete series representation of $\G_2(\R)$ with Harish-Chandra parameter $\frac{k-4}{2}(2\alpha+3\beta)+\rho$, where $\rho=3\alpha+5\beta$ is half the sum of positive roots.\\
\indent Thus the Adams--Johnson packet attached to $\psi=\psi'$ consists of $\mc{L}_\alpha(\pi,1/10)$ and this discrete series representation.
\end{theorem}
\begin{proof}
We study $\mc{R}(\Lambda_w)$ first, using the spectral sequence of Theorem \ref{thmspecseq}. Let $\mf{b}$ be the standard Borel subalgebra of $\mf{g}_2$ containing $\mf{t}_c$, and $\mf{u}$ its unipotent radical. Let $Z$ be the one dimensional $(\mf{t}_c,\mbf{T}_c)$-module given by the character
\[\frac{k-4}{2}(2\alpha+3\beta)+(6\alpha+10\beta).\]
We will induce first from $\mf{t}_c$ to $\mf{l}_2$, and then from $\mf{l}_2$ to $\mf{g}_2$. The relevant degrees $S$ for the cohomological inductions (see Theorem \ref{thmcohindirr}) are both $S=1$; since $\mf{l}_2$ is compact, the degree $S$ for the induction from $\mbf{T}_c$ to $\mbf{L}_2$ just equals the dimension of the unipotent radical of the Borel subalgebra of $\mf{l}_2$, and for $\mbf{L}_2$ to $\mbf{G}_2$, the degree is the number of compact roots not in $\mf{l}_2$. Both of these numbers are $1$.\\
\indent Now the first step is to induce $Z\otimes\C_{-2\rho(\mf{u}\cap\mf{l}_2)}$ to $\mbf{L}_2$. The weight $2\rho(\mf{u}\cap\mf{l}_2)$ equals $\beta$, and hence
\[\mc{R}^1(Z\otimes\C_{-2\rho(\mf{u}\cap\mf{l}_2)})\]
is the discrete series representation of $\mbf{L}_2$ with Harish-Chandra parameter
\[\frac{k-4}{2}(2\alpha+3\beta)+(6\alpha+9\beta)+\frac{1}{2}\beta,\]
by Theorem \ref{thmcohindds}. Since $\mbf{L}_2$ is compact by Lemma \ref{lemL2U2}, this is just the character of $\mbf{L}_2$ given by
\[\frac{k-4}{2}(2\alpha+3\beta)+(6\alpha+9\beta).\]
Now $2\rho(\mf{u}_2)=6\alpha+9\beta$, so
\[\mc{R}^1(Z\otimes\C_{-2\rho(\mf{u}\cap\mf{l}_2)})\otimes\C_{-2\rho(\mf{u}_2)}\]
is the one dimensional representation of $\mbf{L}_2$ given by $\Lambda_w$, and therefore
\[\mc{R}^1(\mc{R}^1(Z\otimes\C_{-2\rho(\mf{u}\cap\mf{l}_2)})\otimes\C_{-2\rho(\mf{u}_2)})=\mc{R}(\Lambda_w).\]
\indent By Theorem \ref{thmspecseq}, $\mc{R}(\Lambda_w)$ is the term $E_2^{1,1}$ of a spectral sequence converging to $\mc{R}(Z\otimes\C_{-2\rho(\mf{u})})$. The other terms in this spectral sequence vanish by Theorem \ref{thmcohindirr}, so we have
\[\mc{R}(\Lambda_w)=\mc{R}(Z\otimes\C_{-2\rho(\mf{u})}).\]
But $2\rho(\mf{u})=6\alpha+10\beta$, so $Z\otimes\C_{-2\rho(\mf{u})}$ is the character $\frac{k-4}{2}(2\alpha+3\beta)$. Thus by Theorem \ref{thmcohindds} again, $\mc{R}(\Lambda_w)$ is just the discrete series representation of $\mbf{G}_2$ with Harish-Chandra parameter
\[\frac{k-4}{2}(2\alpha+3\beta)+\rho,\]
as desired.\\
\indent Now we show that $\mc{R}(\Lambda)$ is the Langlands quotient claimed. The key to this is a theorem in Vogan's book \cite{voganbook}, Theorem 6.6.15, which links the composition of ordinary parabolic induction with cohomological induction with the composition in the opposite order. Instead of recalling the theorem in general, we explain what it means in our special case. It requires three types of data as input: We need what Vogan calls $\theta$-stable data, character data, and cuspidal data, which are defined in general in Definitions 6.5.1, 6.6.1, and 6.6.11, respectively, in Vogan's book. Moreover, there is a bijection between these first two kind of data (Proposition 6.6.2 in \cite{voganbook}) and a surjective map from pieces of character data to pieces of cuspidal data (Proposition 6.6.12 in \cite{voganbook}). Pieces of $\theta$-stable data are used to construct cohomological inductions of parabolically induced representations and in our case will be used to realize the representation $\mc{R}(\Lambda)$. On the other hand, cuspidal data are used to construct parabolic inductions of discrete series representations and will be used to realize $\mc{L}_\alpha(\pi,1/10)$. Theorem 6.6.15 in Vogan's book will then state that these two constructions coincide. We note that this theorem is stated in terms of Langlands subrepresentations instead of Langlands quotients, so we have to make a few minor adjustments.\\
\indent To build the $\theta$-stable data we need, we first construct a certain $\theta$-stable maximal torus of $\mbf{G}_2$. Let $\mbf{T}_0$ be the center of $\mbf{L}_{1,1}$. Let $\mbf A$ be the $\theta$-stable maximal split torus in the derived group of $\mbf{L}_{1,1}$. Then $\mbf{H}=\mbf{T}_0\mbf{A}$ is a maximal torus in $\mbf{G}_2$. It is neither split nor compact. Let $\mu:\mbf{T}_0\to\C^\times$ be given by
\[\mu=\Lambda|_{\mbf{T}_0}=\frac{k-4}{2}\alpha|_{\mbf{T}_0}.\]
Fix a minimal parabolic subgroup $\mbf{B}_{1,1}$ in $\mbf{L}_{1,1}$ containing $\mbf{H}$, and let $\nu:\mbf{A}\to\C^\times$ be the character given by
\[\nu=\delta_{\mbf{B}_{1,1}}^{-1/2}|_{\mbf{A}}.\]
Then the quadruple $(\mf{q}_{1,1},\mbf{H},\mu,\nu)$ is a piece of $\theta$-stable data in the sense of \cite{voganbook}. We write $\mu\otimes\nu$ for the character of $\mbf{H}$ given by $\mu$ on $\mbf{T}_0$ and by $\nu$ on $\mbf{A}$, and we construct the representation (called a \it standard module \rm for our data) given by
\begin{equation}
\label{eqstdmod1}
\mc{R}(\Ind_{\mbf{B}_{1,1}}^{\mbf{L}_{1,1}}((\mu\otimes\nu)\otimes\delta_{\mbf{B}_{1,1}}^{1/2})).
\end{equation}
Of course, in the parabolic induction, the characters $\nu$ and $\delta_{\mbf{B}_{1,1}}^{1/2}$ cancel, and the parabolic induction thus becomes
\[\Ind_{\mbf{B}_{1,1}}^{\mbf{L}_{1,1}}(\mu\otimes 1).\]
By definition of $\mu$, this contains $\Lambda$ as its unique subrepresentation. Since cohomological induction is exact in the good range, we see that $\mc{R}(\Lambda)$ is a subrepresentation of \eqref{eqstdmod1}.\\
\indent Now we construct a piece of character data from $(\mf{q}_{1,1},\mbf{H},\mu,\nu)$ as in \cite{voganbook}. For us this will be a pair $(\mbf{H},\Gamma)$ where $\Gamma:\mbf{H}\to\C^\times$ is a character satisfying certain properties. (Actually, Vogan's definition contains also the data of a character of the complexified Lie algebra $\mf{h}$ of $\mbf{H}$, but that character is determined from the differential of $\Gamma$.) We set $\Gamma|_{\mbf{A}}=\nu$, and we let $\Gamma|_{\mbf{T}_0}$ be the product of $\mu$ with the restriction to $\mbf{T}_0$ of the character $\det(\mf{g}_2^{\theta=-1}\cap\mf{u}_{1,1})$. This latter character is equal to the sum of noncompact roots in $\mf{u}_{1,1}$, and is therefore given by
\[\alpha+(\alpha+\beta)-(\alpha-3\beta)=\alpha-2\beta=2\alpha-(\alpha+2\beta).\]
Its restriction to $\mbf{T}_0$ is therefore given by $2\alpha$, and thus
\[\Gamma|_{\mbf{T}_0}=\frac{k}{2}\alpha|_{\mbf{T}_0}.\]
\indent From $(\mbf{H},\Gamma)$ we construct another piece of data, which Vogan calls cuspidal data. Consider the centralizer of $\mbf{A}$ in $\mbf{G}_2$; this is a Levi subgroup of $\mbf{G}_2$, and we write $\mbf{M}\mbf{A}$ for its Langlands decomposition. The torus $\mbf{A}$ was a maximal split torus in a short root $\SL_2(\R)$, and it follows that $\mbf{M}$ is a long root $\SL_2(\R)$ in $\mbf{G}_2$. Therefore there is a long root parabolic $\mbf{P}=\mbf{M}\mbf{A}\mbf{N}$ in $\mbf G_2$.\\
\indent A piece of cuspidal data constructed from $(\mbf{H},\Gamma)$ will consist of the Levi $\mbf{M}\mbf{A}$, along with a character of $\mbf{A}$, which will is given by $\Gamma|_{\mbf{A}}=\nu$, and also a discrete series representation $\pi_0$ of $\mbf{M}\mbf{A}$. This latter representation is given as the cohomological induction of $\mu'\otimes\nu$ from $\mbf{H}$ to $\mbf{M}\mbf{A}$, where $(\mf{q}',\mbf{H},\mu',\nu)$ is the $\theta$-stable data for $\mbf{M}\mbf{A}$ obtained from the restriction of the character data $(\mbf{H},\Gamma)$ to $\mbf{M}\mbf{A}$. In this data, the $\theta$-stable parabolic $\mf{q}'$ is the intersection of $\mf{q}_{1,1}$ with the complexified Lie algebra $\mf{m}\oplus\mf{a}$ of $\mbf{M}\mbf{A}$. It contains the noncompact root $\alpha$ in its radical. The character $\mu'$ is the restriction of $\Gamma$ to $\mbf{T}_0$ multiplied by the inverse of the sum of the noncompact roots in the radical of $\mf{q}'$. Thus it is equal to $\frac{k-2}{2}\alpha|_{\mbf{T}_0}$.\\
\indent The Theorem 6.6.15 in \cite{voganbook} then asserts that \eqref{eqstdmod1} is isomorphic to
\begin{equation}
\label{eqstdmod2}
\Ind_{\mbf{P}}^{\mbf{G}_2}(\mc{R}(\mu'\otimes\nu)\otimes\delta_{\mbf{P}}^{1/2}).
\end{equation}
The cohomological induction in this expression is, by Theorem \ref{thmcohindds}, the twist of the discrete series representation of $\mbf{M}\mbf{A}$ of weight $k$ by the character $\det^{-1/2}$. Since $\mbf{P}$ is a long root parabolic, $\det^{-1/2}=\delta_{\mbf{P}}^{-1/10}|_{\mbf{M}\mbf{A}}$, and we get that \eqref{eqstdmod2}, and also hence \eqref{eqstdmod1}, are isomorphic to the normalized induction
\[\iota_{M_\alpha(\R)}^{\mbf{G}_2}(\pi,-1/10).\]
Since $\pi$ is self dual, the unique irreducible subrepresentation of this is, by dualizing, isomorphic to $\mc{L}_\alpha(\pi,1/10)$, and this is isomorphic to $\mc{R}(\Lambda)$ by above. This is what we wanted to prove.
\end{proof}
\begin{remark}
We make one more comparison to the $\GSp_4$ case. For $\GSp_4$, the normalized induced representation $\iota_{M_\beta(\R)}^{\GSp_4(\R)}(\pi\boxtimes 1,1/6)$ from the Siegel parabolic contains as a subrepresentation a member of the large discrete series. This large discrete series representation has Harish-Chandra parameter is related, by the Weyl group element which rotates the root lattice clockwise by an angle of $\pi/2$, to the holomorphic discrete series which are the archimedean components of the CAP forms appearing in the proof of Theorem \ref{thmcuspmultgsp4}.\\
\indent For $\G_2$, on the other hand, it is possible to make a (somewhat lengthy) computation using the work of Blank \cite{blank} to show that the induced representation $\iota_{M_\alpha(\R)}^{\G_2(\R)}(\pi,1/10)$ contains the discrete series representation of $\G_2(\R)$ with Harish-Chandra parameter
\[\frac{k-4}{2}(\alpha+3\beta)+(\alpha+4\beta).\]
This parameter is the one obtained from the Harish-Chandra parameter of $\mc{R}(\Lambda_w)$ by applying the rotation $w$ introduced above. The Harish-Chandra parameter of $\mc{R}(\Lambda_w)$, as discussed before, is the one called quaternionic of weight $k/2$ in \cite{ggs}. These quaternionic discrete series are supposed to be analogous to the holomorphic ones, and so we once again see that the theory of representations induced from the long root parabolic of $\G_2$ corresponds well to the theory of those induced from the Siegel parabolic in $\GSp_4$.
\end{remark}
\printbibliography
\end{document}